\pdfoutput=1
\documentclass[11pt, a4paper, reqno]{amsart}

\author{Alessandro Arlandini}
\address[A.A.]{Warwick Mathematics Institute, University of Warwick, Coventry, UK}
\curraddr{Equinor ASA, Oslo, Norway}
\author{David Loeffler}
\address[D.L.]{Faculty of Mathematics \& Computer Science, UniDistance Suisse, Brig, Switzerland}
\urladdr{\href{http://orcid.org/0000-0001-9069-1877}{0000-0001-9069-1877}}

\title{On the factorisation of the \texorpdfstring{$p$}{p}-adic Rankin--Selberg \texorpdfstring{$L$}{L}-function in the supersingular case}
\thanks{Funded by: Engineering \& Physical Sciences Research Council Doctoral Training Partnership (AA); ERC Horizon 2020 Consolidator Grant ``Shimura varieties and the BSD conjecture'' (DL)}

\usepackage[utf8]{inputenc}
\usepackage[T1]{fontenc}
\usepackage{lmodern}
\usepackage[british]{babel}

\usepackage{mathtools, amssymb, amsthm, mathrsfs, multirow, mdframed, tikz-cd}
\usepackage[shortlabels]{enumitem}

\usepackage[colorlinks, linkcolor=blue, citecolor=blue, urlcolor=blue, pagebackref, pdfusetitle]{hyperref}
\usepackage{cleveref}
\usepackage{fullpage}

\hypersetup{%
 pdfencoding=auto,
 pdfauthor={Alessandro Arlandini and David Loeffler},
    pdftitle={On the factorisation of the \texorpdfstring{$p$}{p}-adic Rankin--Selberg \texorpdfstring{$L$}{L}-function in the supersingular case},
    pdfsubject={Number Theory},
}

\usepackage[final]{microtype}

\usepackage[autostyle]{csquotes}

\theoremstyle{definition}
\newtheorem{definition}{Definition}[subsection]

\theoremstyle{plain}
\newtheorem{theorem}[definition]{Theorem}
\newtheorem{proposition}[definition]{Proposition}
\newtheorem{lemma}[definition]{Lemma}
\newtheorem{corollary}[definition]{Corollary}
\newtheorem{thmintro}{Theorem}

\theoremstyle{remark}
\newtheorem*{remark}{Remark}

\DeclareMathOperator{\Sym}{Sym}
\DeclareMathOperator{\TSym}{TSym}
\DeclareMathOperator{\Hom}{Hom}
\DeclareMathOperator{\Ext}{Ext}
\DeclareMathOperator{\Res}{Res}
\DeclareMathOperator{\Fil}{Fil}
\DeclareMathOperator{\GL}{GL}
\DeclareMathOperator{\sgn}{sgn}
\DeclareMathOperator{\cl}{cl}
\DeclareMathOperator{\Spec}{Spec}
\DeclareMathOperator{\et}{\acute{e}t}
\DeclareMathOperator{\etbar}{\overline{\acute{e}t}}
\DeclareMathOperator{\syn}{syn}
\DeclareMathOperator{\rig}{rig}
\DeclareMathOperator{\dR}{dR}
\DeclareMathOperator{\fp}{fp}
\DeclareMathOperator{\mot}{\mathscr{M}}
\DeclareMathOperator{\pr}{pr}

\newcommand{\numberset}{\mathbb}
\newcommand{\C}{\numberset{C}}
\newcommand{\Q}{\numberset{Q}}
\newcommand{\R}{\numberset{R}}
\newcommand{\Z}{\numberset{Z}}
\newcommand{\N}{\numberset{N}}

\newcommand{\BdR}{B_{\dR}}
\newcommand{\Bcrys}{B_{\mathrm{crys}}}

\renewcommand{\theta}{\vartheta}
\renewcommand{\epsilon}{\varepsilon}
\renewcommand{\phi}{\varphi}

\newcommand{\HP}{\mathcal{H}}
\newcommand{\HS}{\mathscr{H}}
\newcommand{\CH}{\mathrm{CH}}

\newcommand{\ideal}{\mathfrak}
\newcommand{\ringint}{\mathcal}
\newcommand{\Frob}{\mathrm{Fr}}
\newcommand{\Ell}{\mathcal{E}}
\newcommand{\KS}{\mathscr{W}}
\newcommand{\theory}{\mathcal{T}}
\newcommand{\gal}[2]{\mathrm{Gal}({#1}/{#2})}

\newcommand{\BF}{\mathrm{BF}}
\newcommand{\lBF}{\widetilde{\mathrm{BF}}}
\newcommand{\lcBF}{\widetilde{\Xi}}
\newcommand{\cBF}{\Xi}
\newcommand{\MH}{\mathrm{MH}}
\newcommand{\del}{\mathcal{D}}
\newcommand{\AJ}{\mathrm{AJ}}

\newcommand{\ups}{\mathcal{S}}

\newcommand{\regulator}[1]{r_{#1}}
\newcommand{\rD}{\regulator{\del}}
\newcommand{\ret}{\regulator{\et}}
\newcommand{\rsyn}{\regulator{\syn}}

\newcommand{\Weight}{\mathcal{W}}

\newcommand{\Lgeom}{L_p^{\mathrm{geom}}}
\newcommand{\Limp}{L^{\mathrm{imp}}}

\newcommand{\ud}{\mathrm{d}}
\renewcommand{\ge}{\geqslant}
\renewcommand{\geq}{\ge}
\renewcommand{\le}{\leqslant}
\renewcommand{\leq}{\le}

\newcommand{\myqed}{\pushQED{\qed} \qedhere \popQED}

\begin{document}

\numberwithin{equation}{subsection}
\microtypesetup{protrusion=false}

\begin{abstract}
Given a cusp form $f$ which is supersingular at a fixed prime $p$ away from the level, and a Coleman family $F$ through one of its $p$-stabilisations, we construct a $2$-variable meromorphic $p$-adic $L$-function for the symmetric square of $F$, denoted $\Limp_p(\Sym^2 F)$. We prove that this new $p$-adic $L$-function interpolates values of complex imprimitive symmetric square $L$-functions, for the various specialisations of the family $F$.

We use this $p$-adic $L$-function to prove a $p$-adic factorisation formula, expressing the geometric $p$-adic $L$-function attached to the Rankin--Selberg convolution of $f$ with itself as the product of the $p$-adic symmetric square $L$-function of $f$ and a Kubota-Leopoldt $L$-function. This extends a result of Dasgupta in the ordinary case.

\bigskip

\noindent\textsc{R\'esum\'e.}
Pour une forme modulaire cuspidale \(f\), supersingulière en un nombre premier fixé \(p\) ne divisant pas le niveau, et une famille de Coleman \(F\) passant par l’une de ses \(p\)-stabilisations, nous construisons une fonction \(L\) \(p\)-adique méromorphe à deux variables pour le carré symétrique de \(F\), notée \(\Limp_p(\Sym^2 F)\). Nous démontrons que cette nouvelle fonction \(L\) \(p\)-adique interpole les valeurs de fonctions \(L\) complexes imprimitives de carré symétrique, pour les diverses spécialisations de la famille \(F\).

Nous utilisons cette fonction \(L\) \(p\)-adique pour démontrer une formule de factorisation \(p\)-adique, exprimant la fonction \(L\) \(p\)-adique géométrique attachée à la convolution de Rankin--Selberg de \(f\) avec elle-même comme le produit de la fonction \(L\) \(p\)-adique de carré symétrique de \(f\) et d’une fonction \(L\) de Kubota-Leopoldt. Cela étend un résultat de Dasgupta dans le cas ordinaire.
\end{abstract}
\maketitle

\centerline{\emph{D\'edi\'e \`a Henri Darmon, avec respect et admiration}}

\setcounter{tocdepth}{1}
\tableofcontents
\microtypesetup{protrusion=true}

\section{Introduction}

 \subsection{Statement of the results}

  Let $f$ be a normalised cuspidal modular newform of level $N$, weight $k + 2 \ge 2$, and character $\psi$. It is easy to see that the Rankin--Selberg convolution $L$-function of $f$ with itself factors as a product of two simpler factors:
  \begin{equation}
   \label{eq:complexfact}
   L(f\otimes f,s) = L(\Sym^2 f,s)L(\psi,s-k-1).
  \end{equation}
  Here we denote by $L(f\otimes f, s)$ and $L(\Sym^2 f, s)$ the primitive Rankin--Selberg and symmetric square complex $L$-functions. This corresponds to the evident decomposition of the tensor product Galois representation
  \begin{equation}
   \label{eq:decomposition-rho}
   \rho_{f,v} \otimes \rho_{f,v} \simeq \Sym^2(\rho_{f,v}) \oplus \det \rho_{f,v}.
  \end{equation}

  Now let $p$ be an odd prime not dividing $N$ at which $f$ is ordinary (i.e.~$a_p(f)$ is a $p$-adic unit). Dasgupta~\cite{Dasgupta-factorization} proved that the factorisation \eqref{eq:complexfact} has a $p$-adic analogue:
  \begin{equation}
   \label{intro:dasgupta}
   L_p(f\otimes f,\sigma) = L_p(\Sym^2 f,\sigma) L_p(\psi,\sigma - k - 1).
  \end{equation}
  Here $\sigma$ is a continuous $p$-adic character of $\Z_p^\times$, and $L_p(\Sym^2 f)$ is Schmidt's symmetric square $p$-adic $L$-function. The function $L_p(f\otimes f, \sigma)$ is the $p$-adic $L$-function defined by specialising Hida's $p$-adic Rankin--Selberg $L$-function for families of modular forms, to a pair of coincident weights. However, this is a significantly deeper theorem than the straightforward factorisation of the complex $L$-function, and its proof relies on the use of Euler systems and motivic cohomology.

  In the present paper, we generalise Dasgupta's result to allow $f$ to be non-ordinary. First, we construct a $p$-adic symmetric square $L$-function for a Coleman family passing through $f$. We then prove a generalisation of~\eqref{intro:dasgupta} to the case when $f$ is not necessarily ordinary by using said function.

  To state the main theorems, let $\Weight$ be the weight space, i.e.~the rigid analytic space classifying $p$-adic characters of $\Z_p^\times$; recall that there is an injection $\Z \hookrightarrow \Weight(\Q_p)$, mapping $n$ to the character $z\mapsto z^n$. We will also assume that the Satake parameters of $f$ at $p$ are distinct, i.e.\ that $\alpha_f\neq\beta_f$, and label them so that\footnote{This is automatic if $f$ is non-ordinary, but our arguments also apply in the ordinary case as long as $\alpha_f$ is not the ``bad'' root of valuation $k + 1$.} $v_p(\alpha_f) < k + 1$. Under this assumption, there exists an affinoid disc $U\subseteq \Weight$ around $k$, and a Coleman family $F$ over $U$, passing through the $p$-stabilisation of $f$ associated to $\alpha_f$. 
  
  \begin{thmintro}[Construction of a two-variable $\Sym^2$ $p$-adic $L$-function]
   \label{intro:final1}
   After possibly shrinking $U$, there exists a $2$-variable meromorphic $p$-adic $L$-function $\Limp_p(\Sym^2 F)\colon U\times\Weight \to \overline{\Q}_p$ with the following property: for all pairs $(t, j) \in \N^2$, with $t \in U \cap \N$, $0 \le j \le t$, and $j$ even, we have
   \[ \Limp_p(\Sym^2 F)(t, 1+j) = (*) \cdot \Limp(\Sym^2 f_t, 1 + j) \]
   where $f_t$ is the weight $t + 2$ specialization of $F$, and $(*)$ is an explicit factor, as long as the non-critical $p$-adic $L$-value $L_p(\psi, j - t)$ does not vanish. 
   
   If $\psi$ and $\psi^2$ both have conductor $N$, then an analogous interpolation formula (with slightly different explicit factors) also holds for odd $j$ in the range $t+1 \le j \le 2t + 1$, again supposing $L_p(\psi, j - t) \ne 0$.
  \end{thmintro}

  Here $\Limp(\Sym^2 f)$ is the imprimitive symmetric square complex $L$-function, see Subsection~\ref{subsec:primitive-imprimitive-functions} below. For the precise form of the interpolation formula we refer to Theorem~\ref{thm:first-half-final} (for $0 \le j \le t$) and Theorem~\ref{thm:interpolation-righthalf} (for $t + 1 \le j \le 2t + 1$). Note that the interpolation formula determines $\Limp_p(\Sym^2 F)$ uniquely on $U \times \Weight_-$, where $\Weight_-$ is the subspace of $\Weight$ classifying odd characters; and when $\psi$ and $\psi^2$ are primitive, $\Limp_p(\Sym^2 F)$ is uniquely determined on all of $U \times \Weight$.
  
  \begin{thmintro}[Dasgupta-style factorization formula]
   \label{intro:final2}
   Let $\Lgeom(F,F)$ be the $3$-variable geometric $p$-adic Rankin--Selberg $L$-function constructed in \cite{loefflerzerbes16}. Then, under the same hypotheses as above, the following factorisation of $p$-adic $L$-functions holds:
   \[
    \Lgeom(F,F)(\kappa,\kappa,\sigma) = \Limp_p(\Sym^2 F)(\kappa,\sigma) \cdot L_p(\psi, \sigma - \kappa - 1)
   \]
   for every $(\kappa,\sigma)\in U\times\Weight$.
  \end{thmintro}

  This gives the hoped-for generalisation of~\eqref{intro:dasgupta} to include the $p$-supersingular case. Finally, we also prove that, under the aforementioned stronger condition on $\psi$, the $p$-adic function defined above satisfies a functional equation.

  \begin{thmintro}[$p$-adic symmetric square functional equation]
   \label{intro:final3}
   Suppose that $\psi$, $\psi^2$ are primitive modulo $N$. Then $\Limp_p(\Sym^2 F)$ satisfies the following functional equation for every $(\kappa, \sigma)\in U\times\Weight$:
   \[
    \Limp_p(\Sym^2 F)(\kappa,2\kappa + 3 - \sigma) = \epsilon(\Sym^2 F)(\kappa, \sigma) \cdot \Limp_p(\Sym^2 F^*)(\kappa,\sigma),
   \]
   where $\epsilon(\Sym^2 F)$ is the unique analytic function on $U \times \Weight$ that interpolates the symmetric square $\epsilon$-factor.
  \end{thmintro}

 \subsection{Strategy of the proof}

  Our strategy is strongly inspired by that of Dasgupta in the ordinary case \cite{Dasgupta-factorization}, but with an important difference, which is that the non-ordinary setting forces us to embrace the use of higher $K$-theory.

  Let us briefly recall Dasgupta's strategy. Suppose $k = 0$, so $f$ has weight 2, and that $f$ is ordinary. Then both sides of \eqref{intro:dasgupta} are measures on $\Z_p^\times$, so they are uniquely determined by their values at $p$-adic characters $\sigma = 1 + \chi$, for $\chi$ of finite order. For each even Dirichlet character $\chi$ of $p$-power conductor, Dasgupta uses Bloch's intersection pairing on higher Chow groups to define a \emph{Beilinson--Flach unit} $b_{f, \chi}$, lying in the $\chi$-eigenspace of $\Z[\mu_{p^\infty}]^\times$. It follows from the regulator formulae of \cite{bertolinidarmon14, leiloefflerzerbes14} that this unit has the following properties:
  \begin{itemize}
   \item its $p$-adic logarithm is (up to explicit correction factors) $L_p(f \otimes f, 1 + \chi)$,
   \item its complex logarithm is (up to explicit factors) $L'(f \otimes f, \chi, 1)$.
  \end{itemize}

  It turns out that the unit eigenspace in which $b_{f, \chi}$ lies is 1-dimensional, and has a canonical generator: the cyclotomic unit $u_\chi$, whose complex logarithm is $L'(\chi, 0)$. So $b_{f, \chi}$ must be a rational multiple of $u_{\chi}$, and the complex-analytic factorisation shows that this rational multiple is the critical $L$-value $L(\Sym^2 f, \chi, 1)$. It follows that the $p$-adic logarithms of $b_{f, \chi}$ and $u_{\chi}$ are related by the same rational factor, showing that \eqref{intro:dasgupta} holds at $\sigma = 1 + \chi$. It follows that \eqref{intro:dasgupta} holds identically over the subset $\Weight_- \subset \Weight$ parametrising characters with $\sigma(-1) = -1$; and the result for the subspace $\Weight_+$ parametrising even characters follows from the functional equation (which interchanges $\Weight_+$ and $\Weight_-$). Lastly, we may generalise from weight 2 forms to forms of arbitrary weight by deforming in a Hida family (in which weight 2 forms are dense).

  For non-ordinary forms, two problems arise. Firstly, the $p$-adic Rankin--Selberg $L$-function is no longer a measure, but a possibly un-bounded locally analytic distribution; so its values at the family of characters $\sigma = 1 + \chi$ above are no longer sufficient to uniquely determine its values everywhere. Secondly, a similar ``non-density'' also holds in the variable $\kappa$: although we can deform non-ordinary $p$-adic modular forms in families (Coleman families), weight 2 specialisations are no longer dense in these.

  We solve these difficulties using the following two innovations:
  \begin{itemize}
   \item using analytic continuation from algebraic characters (of the form $x \mapsto x^n$ for $n \in \Z$) rather than from finite-order characters;
   \item replacing unit groups with more general higher $K$-groups of cyclotomic fields (equivalently, motivic cohomology groups).
  \end{itemize}

  For a form $f$ of arbitrary weight $k + 2 \ge 2$, and any even $j$ in the interval $0 \le j \le k$, we construct a class $b_{f, j} \in H^1_{\mot}(\Q(\mu_N), 1 + k - j)$, the \emph{Beilinson--Flach $K$-element}, whose image under complex (resp.~$p$-adic) regulator captures the complex (resp.~$p$-adic) $L$-value of $f \otimes f$ at $1 + j$. As before, the class $b_{f, j}$ lies in a 1-dimensional character eigenspace, which has a canonical basis given by the \emph{cyclotomic element} of Beilinson and Soul\'e (a higher $K$-theory analogue of the cyclotomic unit); and the complex regulator formula, together with the factorisation of $L(f \otimes f, s)$, shows that the constant of proportionality between these two classes is a critical $\Sym^2$ $L$-value, from which the factorisation of the $p$-adic $L$-value follows.

  Vitally, if we allow both $k$ and $j$ to vary (deforming $f$ in a Coleman family), the set of $(k, j)$ at which this theorem applies is sufficiently dense to uniquely determine the $p$-adic $L$-functions, so we can deduce the factorisation formula over $\Weight_-$, and the factorisation over $\Weight_+$ follows via the functional equation as before.

  (A complicating factor is that there does not appear to be a complete account in the literature of the construction of a $p$-adic symmetric square $L$-function for Coleman families. So we shall in fact \emph{construct} our $p$-adic symmetric square $L$-function as a by-product of the factorisation formula, rather than appealing to a pre-existing construction.)

 \subsection*{Acknowledgements}

  This paper is based on the first author's University of Warwick PhD thesis~\cite{arlandini-thesis}, supervised by the second author. (The second author apologises both to the reader, and to the first author, for having taken so long to revise these results for publication.) It is a pleasure to thank Guido Kings, Dmitrii Krekov, Francesco Lemma, Aprameyo Pal, Giovanni Rosso, Guhan Venkat and Sarah Zerbes for illuminating conversations on the topics of this project; and the anonymous referee for numerous corrections and comments which rescued us from multiple blunders.
  

\section{\texorpdfstring{$L$}{L}-functions and factorisations}

 \subsection{Notation and preliminaries}

  Denote by $\HP$ the upper complex half plane, on which $\GL_2^+(\R)$ acts on the left via Möbius transformations. For $N\in\N_{\geq 1}$, we denote by $\Gamma_1(N)$ and $\Gamma_0(N)$ the usual congruence subgroups of $\mathrm{SL}_2(\Z)$.

  We write $M_r(\Gamma)$ and $S_r(\Gamma)$ for the spaces of modular and cusp forms of level $\Gamma$ and weight $r \geq 1$. When $\Gamma = \Gamma_1(N)$ we use the shorthands $M_r(N)$ and $S_r(N)$, and for $\chi$ a Dirichlet character modulo $N$ we denote by $M_r(N,\chi)$ and $S_r(N, \chi)$ the subspaces of forms transforming via $\chi$ under the diamond operators.

  For every $m\in \N_{\geq 1}$, the Hecke operator at $m$ is denoted $T_m$. For every $d\in (\Z/N\Z)^{\times}$, the diamond operator associated to $d$ is denoted $\langle d \rangle$. The commutative algebra generated over $\Q$ by all Hecke and diamond operators is denoted $\mathbb{T}$ and called Hecke algebra. 
  The Petersson inner product will be denoted by
  \[
   \langle h_1,h_2 \rangle_{\Gamma_1(N)} \coloneqq \int_{\Gamma_1(N) \backslash \HP} h_1(z)\overline{h_2(z)} y^{k - 2} \, \mathrm{d}x\, \mathrm{d}y, \quad z = x + iy.
  \]

  For $h \in S_r(N, \psi)$ we let $h^*(z) \coloneqq \overline{h(-\overline{z})} \in S_r(N, \psi^{-1})$, whose $q$-expansion is given by $a_n(h^*) = \overline{a_n(h)}$. For $h$ a normalised eigenform, we write $K_h = \Q(\{a_n(h)\}_{n\geq 1}) \subset \C$, which is a finite extension of $\Q$.

  Fix once and for all an odd prime $p\in\N$ and an embedding $\overline{\Q} \hookrightarrow \overline{\Q}_p$. We also normalise the $p$-adic valuation on $\overline{\Q}_p$ so that $v_p(p) = 1$. Suppose that $h$ is a normalised eigenform. We say that $h$ is \emph{ordinary} at $p$ if $v_p(a_p(h))=0$, that it is \emph{supersingular} at $p$ if $v_p(a_p(h))>0$. We call $v_p(a_p(h))$ the \emph{slope} of $h$ at $p$.

  For a normalized eigenform $h$ of weight $k + 2$ and level $N$, and a prime $\ell \nmid N$, the Hecke polynomial of $h$ at $\ell$ is the quadratic polynomial
  \[
   X^2 - a_\ell(h)X + \psi(\ell)\ell^{k+1} \in K_h[X].
  \]
  The two roots $\alpha_{h,\ell}$, $\beta_{h,\ell}$ of this polynomial are the Satake parameters of $h$ at $\ell$. The two corresponding $\ell$-\emph{stabilisations} of $h$ are defined as
  \[
   h_{\alpha_{h,\ell}}(z) \coloneqq h(z) - \beta_{h,\ell}h(\ell z), \qquad h_{\beta_{h,\ell}}(z) \coloneqq h(z) - \alpha_{h,\ell}h(\ell z).
  \]
  These are forms of level $\Gamma_1(N) \cap \Gamma_0(\ell)$, with the same Hecke eigenvalues of $h$ away from $\ell$. The eigenvalues of $h_{\alpha_{h,\ell}}$ and $h_{\beta_{h,\ell}}$ under the Hecke operator at $\ell$ are $\alpha_{h,\ell}$ and $\beta_{h,\ell}$ respectively.

  For every prime $\ell$, we denote by $\Frob_\ell$ the arithmetic Frobenius at $\ell$ in $\gal{\overline{\Q}}{\Q}$ (well-defined up to conjugacy and inertia).

  For a Dirichlet character $\chi$ modulo $M$, we write $\chi_{Gal}$ for the character of $\gal{\Q(\mu_M)}{\Q}$ given by composing $\chi$ with the \emph{inverse} of the mod $M$ cyclotomic character, so that $\chi_{Gal}(\Frob_\ell^{-1}) = \chi(\ell)$ for every prime $\ell \nmid M$. (This inverse is required in order that $L(\chi_{Gal}, s) = L(\chi, s)$.)

 \subsection{Primitive and imprimitive \texorpdfstring{$L$}{L}-functions}
  \label{subsec:primitive-imprimitive-functions}

  Let $k, k' \geq 0$ be integers. Let $f \in S_{k+2}(N,\psi)$ and $g \in S_{k'+2}(N',\psi')$ be normalised new cuspidal eigenforms of weights $r=k+2$ and $r'=k'+2$ respectively, with $q$-expansions
  \[
   f = \sum_{n \geq 1} a_n(f)q^n, \quad g = \sum_{n \geq 1} a_n(g)q^n.
  \]
  Without losing generality we assume $k \geq k'$. We assume that the fixed prime $p$ is coprime to $NN'$. We denote by $K_f$ and $K_{f,g}$ the number fields generated by the coefficients of $f$, and $f$ and $g$ respectively. From section~\ref{sec:cohomology-kuga-sato} onwards we will also assume $N, N'\geq 5$, so the modular curves $Y_1(N), Y_1(N')$ are moduli spaces. (The case of small $N$ can be dealt with by introducing an auxiliary level group and taking invariants; we leave the details to the interested reader.)

  Associated to the modular form $f$ there is a collection of Galois representations as constructed in~\cite{deligne69}:
  \[
   \rho_{f,v} \colon G_{\Q} \to \GL(V_{f,v})
  \]
  where $v$ is any place of $K_f$ and $V_{f,v}$ is a $2$-dimensional vector space over $K_{f, v}$. When $v \nmid \ell$, $\rho_{f,v}$ is unramified at $\ell$ with local factor:
  \[
   \det(1 - X\Frob_\ell^{-1} \mid V_{f,v}) = 1 - a_\ell(f)X + \ell^{k+1}\psi(\ell)X^2.
  \]

  \subsubsection*{Rankin-Selberg $L$-functions} The primitive Rankin--Selberg $L$-function has been studied extensively in literature, and the book~\cite{jacquet72} is the most complete account. For the imprimitive $L$-function we refer to~\cite{shimura76,shimura77}.
  
  \begin{definition}
   The \emph{primitive} Rankin--Selberg $L$-function associated to $f$ and $g$ is the $L$-function of the tensor product Galois representation $\rho_{f,v}\otimes \rho_{g,v}$:
   \begin{gather*}
    L(f \otimes g, s) \coloneqq \prod_{\ell \; \text{prime}} P_\ell(\ell^{-s})^{-1}, \\
    P_\ell(X) \coloneqq
    \begin{cases}
     \det\bigl(1 - X\Frob_\ell^{-1} \mid (V_{f,v}\otimes V_{g,v})^{I_\ell}\bigr) &\text{if } v \nmid \ell, \\
     \det\bigl(1 - X\phi \mid (V_{f,v}\otimes V_{g,v} \otimes \Bcrys)^{G_{\Q_\ell}}\bigr) &\text{if } v \mid \ell.
    \end{cases}
   \end{gather*}
   This is known to be independent of $v$. We have denoted with $\phi$ the crystalline Frobenius.
  \end{definition}

  For primes $\ell \nmid NN'$ the local factor has the following explicit form:
  \[
   P_\ell(X) = (1-\alpha_{f,\ell}\alpha_{g,\ell}X)(1-\alpha_{f,\ell}\beta_{g,\ell}X)(1-\beta_{f,\ell}\alpha_{g,\ell}X)(1-\beta_{f,\ell}\beta_{g,\ell}X)
  \]
  where $\alpha_{\bullet,\ell}$ and $\beta_{\bullet,\ell}$ are the local Satake parameters at $\ell$.

  The completed $L$-function $\Lambda(f \otimes g, s) \coloneqq \Gamma_{\C}(s) \Gamma_{\C}(s - k' - 1) L(f \otimes g, s)$ admits a meromorphic continuation to the whole complex plane (with a simple pole at $s = k + 2$ if $f = g^*$, and entire otherwise), and satisfies a functional equation (see e.g.~\cite[\S 19]{jacquet72}) of the form
  \[ \Lambda(f\otimes g, 3 + k + k' -s) = \epsilon(f \otimes g, s) \Lambda(f^* \otimes g^*, s), \]
  where $\epsilon(f \otimes g, s)$ is a function of the form $A \cdot C^s$, for some $C \in \N$ (the analytic conductor, which always divides $(NN')^2$). In particular, if $f$ has character $\psi$, then $L(f \otimes f, s)$ has a simple pole in the following two cases (and is entire otherwise):
  \begin{itemize}
  \item $\psi = 1$,
  \item $\psi$ is an odd quadratic character, and $f$ has CM by the imaginary quadratic field cut out by $\psi$ (the ``exceptional CM'' case).
  \end{itemize}

  \begin{definition}
   The \emph{imprimitive} Rankin--Selberg $L$-function associated to $f$ and $g$ is defined by:
   \[
    L(f,g,s) \coloneqq L_{(N N')}(\psi\psi',2s-2-k-k')\sum_{n\geq 1} \frac{a_n(f)a_n(g)}{n^s} \qquad \text{for $\Re(s) > \frac{k + k'}{2} + 2$}
   \]
   and then continued meromorphically to the whole complex plane. (This is the same function denoted by $D_{NN'}(f, g, s)$ in \cite{Dasgupta-factorization}.) The notation $L_{(N N')}(\psi \psi', -)$ stands for an $L$-function with the local factors at primes diving $NN'$ removed.
  \end{definition}

  The imprimitive $L$-function is given by the following Euler product:
  \begin{gather*}
   L(f,g,s) = \prod_{\ell \; \text{prime}} P_\ell^{\mathrm{imp}}(\ell^{-s})^{-1}, \\
   P_\ell^{\mathrm{imp}}(X) \coloneqq (1-\alpha_{f,\ell}\alpha_{g,\ell}X)(1-\alpha_{f,\ell}\beta_{g,\ell}X)(1-\beta_{f,\ell}\alpha_{g,\ell}X)(1-\beta_{f,\ell}\beta_{g,\ell}X).
  \end{gather*}
  Here we define $(\alpha_{f,\ell},\beta_{f,\ell}) \coloneqq (a_\ell(f), 0)$ if $\ell \mid N$ (which does not conflict with the above notations, since Satake parameters are not defined in this case), and similarly for $g$ if $\ell \mid N'$. We have $P_\ell^{\mathrm{imp}}(X) \mid P_\ell(X)$, with equality if $\ell \nmid N N'$, so the primitive and imprimitive $L$-functions agree up to finitely many Euler factors.

  \begin{remark}
   The advantage of the imprimitive $L$-function is that its local factors are more explicit than those of $L(f \otimes g, s)$, and in addition, it has a more explicit integral representation as we shall recall below. Moreover, it is the imprimitive $L$-function which appears in the regulator formulae for Beilinson--Flach elements. Its chief disadvantage is that it does not, in general, have a functional equation.
  \end{remark}

  \begin{definition}
   We define
   \[ P(s) \coloneqq \frac{L(f, g, s)}{L(f \otimes g, s)} = \prod_{\ell \mid N} \frac{P_\ell(\ell^{-s})}{P_\ell^{\mathrm{imp}}(\ell^{-s})}.\]
   By construction this is a product of polynomials in $\ell^{-s}$ for $\ell \mid N$ (each of degree $\le 4$).
  \end{definition}

  \begin{proposition}\label{prop:imprimitive-eq-prim}
   If $N = N'$ and all three characters $\psi$, $\psi'$ and $\psi\psi'$ are primitive of conductor $N$, then $L(f, g, s) = L(f \otimes g, s)$, and the analytic conductor is equal to $N^3$.
  \end{proposition}

  \begin{proof}
   The hypothesis implies that for each $\ell \mid N$ (and $v \nmid \ell$), both $V_{f, v}|_{G_{\Q_{\ell}}}$ and $V_{g, v}|_{G_{\Q_\ell}}$ are the direct sum of a ramified and an unramified character, so a direct computation of Euler factors gives the result.
  \end{proof}

  \subsubsection*{Symmetric square} The (complex) $L$-function of the symmetric square, which we now introduce, fits in a similar picture. Our main references are~\cite{shimura75}, which studies the imprimitive one and its analytic properties, and~\cite{schmidt88}, which studies the primitive one in great detail and also gives an account of the comparison between the two.

  \begin{definition}
   The \emph{primitive} symmetric square $L$-function associated to $f$ is the $L$-function of the representation $\Sym^2 \rho_{f,v}$:
   \begin{gather*}
    L(\Sym^2 f,s) \coloneqq \prod_{\ell \; \text{prime}} Q_\ell(\ell^{-s})^{-1}, \\
    Q_\ell(X) =
    \begin{cases}
     \det\bigl(1 - X\Frob_\ell^{-1} \mid (\Sym^2 V_{f,v})^{I_\ell}\bigr) &\text{if } v \nmid \ell,  \\
     \det\bigl(1 - X\phi \mid (\Sym^2 V_{f,v}\otimes \Bcrys)^{G_{\Q_\ell}}\bigr) &\text{if } v\mid \ell.
    \end{cases}
   \end{gather*}
  \end{definition}

  As with the Rankin--Selberg local factors, the primitive local factors $Q_\ell$ are independent of $v$, and for $\ell \nmid N$ they are given by an explicit formula:
  \[
   Q_\ell(X) = (1-\alpha_{f,\ell}^2 X)(1-\alpha_{f,\ell}\beta_{f,\ell}X)(1-\beta_{f,\ell}^2 X).
  \]
  The critical values of $L(\Sym^2 f,s)$ are the odd integers in the range $1 \leq s \leq k+1$ and the even integers in the range $k+2 \leq s \leq 2k+2$. It is entire except in the ``exceptional CM'' case, in which case it has a simple pole at $s = k + 2$.

  \begin{definition}
   The \emph{imprimitive} symmetric square $L$-function associated to $f$ is defined by:
   \[
    \Limp(\Sym^2 f,s) \coloneqq L_{(N)}(\psi^2, 2s-2k-2)\sum_{n\geq 1} \frac{a_{n^2}(f)}{n^s} \qquad \text{for $\Re(s) > k + 2$}
   \]
   and then continued meromorphically to the whole complex plane.
  \end{definition}

  In the region $\Re(s) > k + 2$, $\Limp(\Sym^2 f,s)$ admits the following Euler product:
  \begin{gather*}
   \Limp(\Sym^2 f,s) = \prod_{\ell \; \text{prime}} Q_\ell^{\mathrm{imp}}(\ell^{-s})^{-1}, \\
   Q_\ell^{\mathrm{imp}}(X) = (1-\alpha_{f,\ell}^2 X)(1-\alpha_{f,\ell}\beta_{f,\ell}X)(1-\beta_{f,\ell}^2 X),
  \end{gather*}
  where as before we understand $(\alpha_{f,\ell},\beta_{f,\ell}) = (a_\ell(f),0)$ for $\ell \mid N$. Exactly as in the Rankin--Selberg case, we have a divisibility relation $Q_\ell^{\mathrm{imp}}(X) \mid Q_\ell(X)$ for all $\ell$, and for $\ell \nmid N$ this is an equality. Thus we can write $\Limp(\Sym^2 f,s) = L(\Sym^2 f,s)\cdot Q(s)$, where $Q$ is a product of polynomials in $\ell^{-s}$ for bad primes $\ell$.

  Analogous to Proposition \ref{prop:imprimitive-eq-prim}, if $\psi$ and $\psi^2$ have conductor exactly $N$, we have $L(\Sym^2 f, s) = \Limp(\Sym^2 f, s)$.

 \subsection{Primitive and imprimitive factorisation}

  \begin{proposition}
   We have the factorisations
   \begin{subequations}
    \begin{align}
     L(f \otimes f,s) &= L(\Sym^2 f,s)L(\psi, s - k - 1), \label{eq:complex-factorization-l} \\
     L(f,f,s) &= \Limp(\Sym^2 f,s)L_{(N)}(\psi,s-k-1) \label{eq:imprimitive-factorization-l}.
    \end{align}
   \end{subequations}
  \end{proposition}

  \begin{proof}
   The first formula follows from the definitions and \eqref{eq:decomposition-rho}, and the second from a straightforward explicit computation of Euler factors.
  \end{proof}

  \begin{corollary}
   \label{cor:defect-terms}
   The functions $P$ and $Q$ are related by:
   \[
    P(s) = Q(s)\, \prod_{\mathclap{\ell \mid N,\, \ell \nmid N_{\psi}}}\, \Bigl(1 - \frac{\psi(\ell)}{\ell^{s-k-1}} \Bigr).\myqed
   \]
  \end{corollary}

  Note that $(-1)^k = \psi(-1)$, so $L_{(N)}(\psi, j - k) = 0$ for all even $j < k$. This is also true for $j = k$ as long as $N > 1$. This gives a relation between first derivatives
  \begin{equation}
   L'(f,f,j+1) = \Limp(\Sym^2 f,j+1)L_{(N)}'(\psi,j-k)
   \label{eq:leading-terms}.
  \end{equation}
  This will be the starting point to deduce the $p$-adic factorisation from the complex one.

\section{\texorpdfstring{$p$}{p}-adic \texorpdfstring{$L$}{L}-functions}
\label{sec:padic-Lfunctions}

 In this section we introduce the $p$-adic $L$-functions we will study. We regard them as functions on the weight space $\Weight$, i.e.\ the rigid analytic space whose $K$-points are $\Weight(K) = \Hom(\Z_p^{\times},K^{\times})$ for each complete extension $K$ of $\Q_p$. We identify $\Z$ with a subset of $\Weight(\Q_p)$ in the obvious way. Arithmetic weights are those of the form $z \mapsto z^n \theta(z)$, for $\theta$ a Dirichlet character of $p$-power conductor and $n \in \Z$. We will use the additive notation ``$n + \theta$'' for such characters.

 Let $\Weight_{\pm}$ denote the two halves of the weight space whose $M$-points are, for every extension $M/\Q_p$:
 \[
  \Weight_{\pm}(M) = \{ \kappa\in\Hom(\Z_p^{\times},M^{\times}) \mid \kappa(-1) = \pm 1 \}.
 \]
 In particular, for the arithmetic weights
 \[
  \theta + n \in \Weight_{\pm} \iff \theta(-1) = \pm (-1)^n.
 \]

 \subsection{Dirichlet characters}

  Let $\chi$ be a Dirichlet character of conductor $N_{\chi}$ coprime to $p$, and define $a_{\chi}$ as $0$ if $\chi$ is even, $1$ if it is odd. Denote the Gauss sum of a character by:
  \[
   \tau(\chi) \coloneqq \sum_{n=1}^{N_{\chi}} \chi(n) \exp\left(\tfrac{2\pi i n}{N_{\chi}}\right).
  \]
  \begin{proposition}[Kubota--Leopoldt]
   There exists a unique $p$-adic meromorphic function $L_p(\chi,\cdot)$ which at arithmetic weights $\theta+n$ satisfies the interpolation property: if $n\leq 0$ and $\chi\theta(-1) = (-1)^{n+1}$, then
   \begin{align*}
    L_p(\chi,\theta+n) &=  \mathrm{E}_p(\chi\theta^{-1},n)L(\chi\theta^{-1},n)
 \intertext{while if $n\geq 1$ and $\chi\theta(-1)=(-1)^n$, then}
    L_p(\chi,\theta+n) &=  \mathrm{E}_p(\chi^{-1}\theta,1-n)\frac{2\Gamma(n)i^{a_{\chi}}}{(2\pi i)^n} \frac{\chi(N_{\theta})\tau(\theta^{-1})}{N_{\theta}^n}L(\chi\theta^{-1},n)
   \end{align*}
   where $\mathrm{E}_p(\theta,x) = 1-\theta(p) p^{-x}$ if $\theta$ is unramified at $p$, and $\mathrm{E}_p(\theta,x) = 1$ otherwise.
  \end{proposition}

  The set of points at which we have interpolation is dense in $\Weight$, hence $L_p(\chi,\cdot)$ is uniquely determined. Note that when $\chi = 1$, $L_p(\chi, \sigma)$ has poles at $\sigma = 0$ and $\sigma = 1$ (and no other characters); if $\chi \ne 1$, then $L_p(\chi, \sigma)$ is analytic.

  \begin{remark}
   There are many conventions for how to define the $p$-adic $L$-function; many authors take it to be identically zero on $\Weight_+$, or on $\Weight_-$. We have followed \cite{Dasgupta-factorization}. With these conventions, the functional equation takes a form closely analogous to the complex version:
  \end{remark}

  \begin{theorem}
   Let $\tilde{\epsilon}(\chi, \sigma) \coloneqq \tfrac{i^{a_\chi}}{\tau(\chi^{-1})} N_\chi^\sigma$ (an analytic function of $\sigma \in \Weight$). Then, for every $\sigma \in \Weight$, we have
   \[
    L_p(\chi,1-\sigma) = 
   \tilde{\epsilon}(\chi, \sigma) \cdot L_p(\chi^{-1},\sigma).\myqed
   \]
  \end{theorem}

\subsection{$p$-adic Rankin--Selberg $L$-functions}

 Recall that $L(f\otimes g,s)$ denotes the complex Rankin--Selberg $L$-function attached to $f \in S_{k+2}(N, \psi)$ and $g\in S_{k'+2}(N', \psi')$; without loss of generality suppose $k \geq k'$. If $k > k'$, this $L$-function has critical values, and we can attempt to construct a $p$-adic $L$-function by interpolating these. If $k = k'$ this does not work -- there are no critical values to interpolate -- so it is necessary to broaden our viewpoint, and allow the modular forms $f, g$ to vary in Coleman families.

  We do not aim to give here a full-fledged account of the topic of families of modular forms, we only explain the results relevant to our case, following~\cite{loefflerzerbes16}.

  \begin{definition}
   Let $U\subseteq \Weight$ be an affinoid disc defined over a finite extension $H$ of $\Q_p$, such that the set of classical weights $U \cap \N$ is dense in $U$. A \emph{Coleman family} $F$ over $U$ (of tame level $N$, with $p \nmid N$) is a power series
   \[
    F = \sum_{n\geq 0} a_n(F)q^n \in \mathcal{O}(U)[[q]]
   \]
   with $a_1(F)=1$ and $a_p(F)$ invertible in $\mathcal{O}(U)$, such that for all but finitely many classical weights $t \in U\cap \N$, the series $F_t = \sum_{n\geq 0} a_n(F)(t)q^n \in H[[q]]$ is the $q$-expansion of a classical modular form of weight $t+2$ and level $\Gamma_1(N)\cap\Gamma_0(p)$ which is a normalised eigenform.
  \end{definition}

  \begin{theorem}[\cite{loefflerzerbes16}]
   \label{thm:geometric-l-function}
   Let $U_1$ and $U_2$ be affinoid discs\footnote{To avoid trivial cases we should require that $U_1 \cap \Z$ and $U_2 \cap \Z$ are both non-empty, and hence Zariski-dense in $U_1$ and $U_2$ respectively.} in $\Weight$, and $F$ and $G$ Coleman families over $U_1$, $U_2$ respectively. Then there exists a unique meromorphic rigid function (analytic if $U_1$ is small enough)
   \[
    \Lgeom(F,G) \colon U_1\times U_2 \times \Weight \to \C_p \\
   \]
   with the following interpolation property. Suppose that $(t, t', j) \in U_1 \times U_2 \times \Weight$ is a triple of integers satisfying the inequalities
   \[ t, t' \ge 0, \qquad t' + 1 \le j \le t, \]
   and that $F$, $G$ specialise at $t$, $t'$ respectively to $p$-stabilisations $f_{\alpha_f}$ and $g_{\alpha_g}$ of eigenforms $f, g$ of prime-to-$p$ level. Then
   \begin{multline*}
     \Lgeom(F,G)(t,t',j+1) =\frac{E(f_{t},g_{t'},j+1)}{E(f_{t})E^*(f_{t})} \cdot \frac{\Lambda(f_t, g_{t'}, j + 1)}{2^{t + 3}(-i)^{t-t'}\langle f_t, f_t \rangle} \\ =\frac{E(f_{t},g_{t'},j+1)}{E(f_{t})E^*(f_{t})} \cdot \frac{j!(j-t'-1)!}{\pi^{2j-t'+1}(-i)^{t-t'}2^{2j+2+t-t'}\langle f_{t},f_{t}\rangle} 
     \cdot L(f_{t},g_{t'},j+1),
    \end{multline*}
    where
    \begin{gather*}
    E(f) \coloneqq \left(1-\tfrac{\beta_f}{p\alpha_f}\right), \quad E^*(f) \coloneqq \left(1-\tfrac{\beta_f}{\alpha_f}\right), \\
    E(f,g,j+1) \coloneqq \left(1-\tfrac{p^j}{\alpha_f\alpha_g}\right) \left(1-\tfrac{p^j}{\alpha_f\beta_g}\right) \left(1-\tfrac{\beta_f\alpha_g}{p^{1+j}}\right) \left(1-\tfrac{\beta_f\beta_g}{p^{1+j}}\right).
   \end{gather*}
  \end{theorem}

  \begin{remark}
   Such a construction was announced in the main theorem of~\cite[\S 4.4]{Urban-nearly-overconvergent}, but the construction had a gap (later filled in by \cite{andreattaiovita21}). In the meantime, an independent construction was given in \cite{loefflerzerbes16}, which we follow here. We have shifted the $j$-variable by 1 relative to \cite{loefflerzerbes16}, for consistency with \cite{Dasgupta-factorization} in the ordinary case.
  \end{remark}

  Now let $f$ be an eigenform of level $N$ and weight $k + 2 \ge 2$, and $\alpha_f$ a root of its Hecke polynomial. 
  
  \begin{definition}
  \label{def:noble}
   We say $f$, or more precisely the pair $(f, \alpha_f)$, is \emph{noble} if the following two conditions hold:
   \begin{itemize}
   \item $\alpha_f$ is distinct from the other root\footnote{This is automatic for $k = 0$, and for all $k \ge 0$ if the Tate conjecture holds \cite{colemanedixhoven98}.} $\beta_f$ (``$p$-regular'').
   \item if $v_p(\alpha_f) = k + 1$ (the maximal possible value), then $M_{\et}(f)|_{G_{\Q_p}}$ is not the direct sum of two characters (``non-critical'').\footnote{If $f$ is of CM-type, with $p$ split in the CM field, and $\alpha_f$ is the non-unit root, then $f$ is always critical. Conversely, is conjectured that non-CM cusp forms are always non-critical.}
   \end{itemize}
  \end{definition}

  If $f$ is noble, then for any small enough affinoid $U \ni k$, there exists a unique Coleman family $F$ over $U$ specialising to $f_{\alpha_f}$ at $k$ \cite[Theorem~4.6.4]{loefflerzerbes16}. We can thus define a $p$-adic Rankin--Selberg $L$-function for $f$,
  \[ L_{p, \alpha_f}(f, f, -) : \Weight \to \C_p, \qquad L_{p, \alpha_f}(f, f, \sigma) \coloneqq \Lgeom(F, F)(k, k, \sigma).\]

  \begin{remark}
   Note that $L_{p, \alpha_f}(f, f, \sigma)$ is an analytic function of $\sigma$, even when $f^* = f$ (unlike its complex counterpart, which has a pole in this case). In the light of the factorization formula that we shall prove below, this can be seen as a consequence of the poles of the $p$-adic Riemann zeta being cancelled out by ``trivial zeroes'' of the $p$-adic symmetric square $L$-function.
  \end{remark}

\section{Cohomology and regulators}

 In this section we introduce the main tools from arithmetic geometry which we will use in our construction.

 \subsection{Motivic cohomology}

  For a regular scheme $X$, let $K_{\bullet}(X)$ be the $K$-theory of $X$. For every integer $n\in\N$, denote by $\mathrm{gr}^{\gamma}_n K_{\bullet}(X)$ the $n$-th graded piece of the $\gamma$-filtration of $K_{\bullet}(X)$, that is the eigenspace where the Adams operator $\psi^k$ acts as $k^n$ for every $k\in\N$.

  \begin{definition}[\cite{beilinson84}]
   If $X$ is a regular scheme, its \emph{motivic cohomology groups} for $n\in\N$ are:
   \[
    H^i_{\mot}(X,n) \coloneqq \Q \otimes_{\Z} \mathrm{gr}_n^{\gamma} K_{2n-i}(X).
   \]
  \end{definition}

  We briefly recall some functorial properties of this theory. Let $X$ be a smooth and proper variety of even dimension $2d$ over a field $K$ of characteristic zero. Let $0\leq j\leq d-1$ be an integer. Let $\pi \colon X \to \Spec K$ be the structural morphism, then there is a push-forward map in any degree and for all $n\in\N$
  \[
   \pi_* \colon H_{\mot}^{\bullet}(X,n) \to H_{\mot}^{\bullet-4d}(K, n-2d).
  \]
  By composing the cup product and the push-forward along $\pi$ we can define a pairing in motivic cohomology by composition:
  \[
   \langle\cdot ,\cdot \rangle \colon H^{2d+1}_{\mot}(X,2d-j) \times H^{2d}_{\mot}(X,d) \xrightarrow{\cup} H^{4d+1}_{\mot}(X,3d-j) \xrightarrow{\pi_*} H^1_{\mot}(K, d-j),
  \]
  where we have written $H^*_{\mot}(K, -)$ as an abbreviation for $H^*_{\mot}(\Spec K, -)$. One of the main results of this article is the construction of an element in the rightmost motivic cohomology group which gives back $L$-values when realised in different cohomology theories. In this section we lay down the machinery to compute said realisations.

  \begin{remark}
   When $d = 1$ (so $X$ is a surface) and $j = 0$, the above pairing has a geometric interpretation as the Bloch intersection pairing. Indeed, we have isomorphisms
   \begin{gather*}
    H_{\mot}^{3}(X, 2) \simeq \CH^{2}(X, 1), \qquad
    H_{\mot}^{2}(X, 1) \simeq \CH^1(X), \\
    H_{\mot}^1(K,1) \simeq K^{\times}.
   \end{gather*}
   Elements of $\CH^{2}(X, 1)$ are given by sums $\sum_i (Z_i, f_i)$, where $Z_i$ are irreducible curves in $X$ and $f_i \in k(Z_i)^\times$, such that $\sum_i \operatorname{div}(Z_i) = 0$ as a zero-cycle on $X$. For a curve $Z'$ which intersects transversely with $\bigcup_i Z_i$, the pairing is given by
   \[
       \langle \sum_i (Z_i, f_i), [Z'] \rangle = \prod_i \prod_{x \in Z_i \cap Z'} f_i(x).
   \]
   This is the point of view used in \cite{Dasgupta-factorization}.
  \end{remark}
  
  In the next two sections, we shall consider two other cohomology theories -- Deligne cohomology and $p$-adic (absolute) \'etale cohomology -- which admit natural ``regulator'' maps from motivic cohomology; these will allow us to relate the motivic Beilinson--Flach classes to complex and $p$-adic $L$-values respectively. We denote these theories by $H^*_\theory(X, \Q_\theory(n))$, for $\theory \in \{ \del, \et\}$, where $\Q_\theory$ is the appropriate coefficient ring, and 
  \[ r_{\theory} : H^*_{\mot}(X, n) \to H^*_{\theory}(X, \Q_\theory(n)) \]
  the regulator map.
  

 \subsection{Deligne cohomology} We shall only consider Deligne cohomology for varieties defined over $\C$; for varieties over subfields $K$ of $\C$, we define Deligne cohomology by base-extending to $\C$ (not attempting to keep track of the additional complex-conjugation action if $K$ is contained in $\R$). 

  
  \subsubsection*{Real coefficients} 
  
  If $X$ is a smooth complex variety, and $j : X \hookrightarrow \bar{X}$ is a smooth compactification such that $D = \bar{X} - X$ is a normal crossing divisor, we define $H_{\del}^\bullet(X, \R(q))$ as in \cite[Definition 2.9]{esnaultviehweg}, to be the hypercohomology of the complex of sheaves on the complex manifold $\bar{X}^{\mathrm{an}}$ given by
  \[ Rj_* \underline{\R(q)}_{X^{\mathrm{an}}} \oplus F^q \Omega^\bullet_{\bar{X}^{\mathrm{an}}}\langle D\rangle \to j_* \Omega^\bullet_{X^{\mathrm{an}}}, \]
  where $\R(q) = (2\pi i)^q \R \subset \C$ and $\Omega^\bullet_{\bar{X}^{\mathrm{an}}}\langle D\rangle$ denotes the complex of holomorphic differential forms with logarithmic singularities along $D$. This is independent of the choice of $\bar{X}$, and by construction fits into a long exact sequence (functorial in $X$)
  \begin{equation}
   \label{eq:deligneexact}
   \dots \to H^p_{\del}(X, \R(q)) \to H^p_B(X, \R(q)) \to \frac{H^p_{\dR}(X, \C)}{\Fil^q} \to \dots,
  \end{equation}
  where $H^p_B(X, \R(q))$ denotes Betti cohomology of the manifold $X(\C)$. By Beilinson's work \cite{beilinson84}, we have regulator maps
  \[ r_{\del} : H^p_{\mot}(X, q) \to H^p_{\del}(X, \R(q))\]
  for all $p, q$, and these are compatible with pushforward, pullback, and cup-product.
  
  Deligne cohomology with $\R$-coefficients is closely related to extension groups in the category $\MH_\R$ of real mixed Hodge structures. These Ext groups are given as follows (see \cite[Equation 2.4.2]{jannsen88b}): for an object $(H, W_\bullet, F^\bullet)$, we have:
  \begin{equation}
   \label{eq:ext-groups}
   \begin{aligned}
    \Ext^0_{\MH_\R}(\R, H) &\simeq W_0 H \cap F^0 (W_0 H_{\C}), \\
    \Ext^1_{\MH_\R}(\R, H) &\simeq \frac{W_0 H_{\C}}{W_0 H + F^0(W_0 H_{\C})}, \\
    \Ext^i_{\MH_\R}(\R, -) &= 0 \quad \text{for $i>1$}.
   \end{aligned}
  \end{equation}
  
  If $X$ is smooth and proper, and $p \le 2q$, then $H^p_B(X, \R(q))$ is naturally a pure Hodge structure of weight $\le 0$, so we can interpret \eqref{eq:deligneexact} as a short exact sequence
  \[
   0 \to \Ext^1_{\MH_\R}\left(\R, H^{p-1}_B(X, \R(q))\middle) \to H_{\del}^{p}(X, \R(q)) \to \Ext^0_{\MH_\R}\middle(\R, H^{p}_B(X, \R(q))\right) \to 0.
  \]
  In particular, if $p < 2q$ then $H^{p}_B(X, \R(q))$ is a pure Hodge structure of negative weight; so the $\Ext^0$ term vanishes, and $W_0 H^{p-1}_B(X, \R(q)) = H^{p-1}_B(X, \R(q))$, hence
  \[ H_{\del}^{p}(X, \R(q)) \cong \frac{H^{p-1}_{\dR}(X, \C)}{(2\pi i)^q H^{p-1}_B(X, \R) + F^q H^{p-1}_{\dR}(X/\C)}.\]

  When $X$ is smooth and affine (but not necessarily proper), we do not necessarily have an interpretation in terms of $\Ext$ of Hodge structures (i.e.~Betti and absolute Hodge cohomology no longer coincide); but there is still a long exact sequence \eqref{eq:deligneexact} relating Deligne, de Rham and Betti cohomology. Moreover, for an affine variety of dimension $2d$, the de Rham and Betti cohomologies vanish in degrees strictly greater than $2d$, so:

  \begin{proposition}
   \label{prop:isomorphism-for-affine-deligne}
   If $X$ is smooth and affine of dimension 2d, then there is an isomorphism
   \[
    H_{\del}^{2d+1}(X, \R(d-j)) \xrightarrow{\simeq}
    \frac{H^{2d}_{\dR}(X, \C)}{(2\pi i)^{d-j} H^{2d}_B(X, \R) + F^{d-j} H^{2d}_{\dR}(X/\C)}.
   \]
  \end{proposition}

 \subsubsection*{Complex coefficients}
  
  So far we assumed the coefficients for Deligne cohomology were real; but we shall want to work with modular forms of non-trivial character, whose coefficient fields usually are not totally real, so we shall  need to work over $\C$. Hence we define \emph{Deligne cohomology with complex coefficients} by 
  \[ H^p_{\del}(X, \C(q)) \coloneqq \C \otimes_{\R} H^p_{\del}(X, \R(q)).\]
  
  \begin{remark}
   Note that this notion of ``complex Deligne cohomology'' is \emph{not} the same theory that one would obtain by replacing the constant sheaf $\underline{\R}(q)$ with $\underline{\C}(q)$ in the definition of Deligne cohomology above. The latter theory is not interesting for us, since it coincides with truncated de Rham cohomology. (We are grateful to the anonymous referee for pointing out a confusion between these two theories in an earlier draft of this paper.)
  \end{remark}
    
  As in the real case, for $X$ proper and $p \le 2q$ we can interpret this theory via Ext groups in the category $\MH_\C = \C \otimes_{\R} \MH_{\R}$ of ``complexified real Hodge structures'' (whose objects are real Hodge structures equipped with an $\R$-algebra map from $\C$ to their endomorphism rings, and whose morphisms are required to be $\C$-linear.) For a smooth variety $X / \C$, we shall consider $H^p_B(X, \C(q)) = \C\otimes_\R H^p_B(X, \R(q))$ as an object of this category. Although the category may seem rather artificial, the Ext-groups have a straightforward description:
  
  \begin{proposition}
   If $X$ is smooth and proper, then for any $q \ge 0$ we have
   \[ \Ext^0_{\MH_\C}\left(\C, H^{2q}_B(X, \C(q))\right) = H^{q, q}_{\dR}(X, \C). \]
  \end{proposition}
  
  \begin{proof} Since complex conjugation interchanges the $(r, s)$ and $(s, r)$ terms in the Hodge decomposition, we have
  \begin{align*}
   \Ext^0_{\MH_\C}\left(\C, H^{2q}_B(X, \C(q))\right) &= \C \otimes_{\R} \left( H^{2q}_B(X, \R(q)) \cap F^q H^{2q}_{\dR}(X, \C)\right) \\
   &= F^q H^{2q}_{\dR}(X, \C) \cap \overline{F^q H^{2q}_{\dR}(X, \C)} \\
   &= H^{q, q}_{\dR}(X, \C).\qedhere
  \end{align*}
  \end{proof}
 \subsection{A complex diagram}
   \label{sec:complex-diagram}
  
  Now suppose $X$ is smooth and proper over $\C$ of dimension $2d$, as above, and let $\rD \colon H_{\mot} \to H_{\del}$ be the Deligne regulator. As $\rD$ is compatible with cup product and push-forward, we have a commutative diagram (for $0 \le j < d$):

  \begin{equation}
   \label{diag:complex-compatibility}
   \begin{tikzcd}[column sep=small]
    H^{2d+1}_{\mot}(X,2d-j) \times H^{2d}_{\mot}(X,d) \arrow["\cup"]{r} \arrow["\rD \times \rD"]{d} &
    H^{4d+1}_{\mot}(X,3d-j) \arrow["\pi_*"]{r} \arrow["\rD"]{d} &
    H^1_{\mot}(K, d-j) \arrow[d,"\rD"] \\
    \begin{array}{l}
     H_{\del}^{2d+1}(X_{\C}, \R(2d-j)) \\[0.2ex]
    \qquad\qquad \times\  H_{\del}^{2d}(X_{\C}, \R(d))
    \end{array}
    \arrow["\cup"]{r} &
    H_{\del}^{4d+1}(X_{\C}, \R(3d-j)) \arrow["\pi_*"]{r} &
    H_{\del}^1(\C, \R(d-j))
   \end{tikzcd}
  \end{equation}

  For computations it is useful to add a third row relating Deligne cohomology to de Rham and Betti: we have canonical maps
  \[
   \begin{aligned}
    H_{\del}^{2d+1}(X_{\C},\R(2d-j)) &\xrightarrow{\simeq} \Ext^1_{\MH_\R}\left(\R, H_B^{2d}(X_{\C},\R(2d-j))\right),\\
    H_{\del}^{2d}(X_{\C},\R(d)) &\to \Ext^0_{\MH_\R}\left(\R, H_B^{2d}(X_{\C},\R(d))\right).
   \end{aligned}
  \]
  The cup-product between these lands in
  \[
   H^1_{\del}(\C, \R(d - j)) = \Ext^1_{\MH_\R}(\R, \R(d - j)) = \frac{\C}{(2\pi i)^{d-j}\R}.
  \]

  \begin{theorem}
   If $X$ is smooth and proper, then we can extend \eqref{diag:complex-compatibility} to a commutative diagram as in figure~\ref{diag:complex-diagram}.
  \end{theorem}

  \begin{figure}[b]
   \centering
   \begin{mdframed}
    \caption{Complex diagram}
    \label{diag:complex-diagram}
    \[
     \begin{tikzcd}
      H^{2d+1}_{\mot}(X,2d-j) \times\ H^{2d}_{\mot}(X,d)
      \arrow["\pi_* \circ\, \cup"]{r} \arrow["\rD \times \rD"]{d} &
      H^1_{\mot}(K,d-j) \arrow[d,"\rD"] \\
       H_{\del}^{2d+1}(X_{\C}, \R(2d-j))
       \times\ H_{\del}^{2d}(X_{\C}, \R(d))
      \arrow["\pi_* \circ\, \cup"]{r} \arrow{d} &
      H^1_{\del}(\C, \R(d-j)) \arrow[equal]{d} \\
      \begin{array}{l}
       \Ext^1_{\MH_\R}(\R, H_B^{2d}(X, \R(2d-j))) \\[0.5ex]
       \qquad \times\ \Ext^0_{\MH_\R}(\R, H_B^{2d}(X,\R(d)))
       \end{array} \arrow["\pi_* \circ\, \cup"]{r} &
      \Ext^1_{\MH_\R}(\R, \R(d - j))
     \end{tikzcd}
    \]
   \end{mdframed}
  \end{figure}

  Note that the composite of the maps in the left column,
  \[
   H^{2d}_{\mot}(X, d) 
   \to \Ext^0_{\MH_\R}(\R, H_B^{2d}(X_{\C},\R(d))) = H^{d, d}_{\dR}(X, \C) \cap H^{2d}_B(X, \R(d))
  \]
  is the \emph{de Rham cycle class map}
  \begin{equation}
   \label{def:cycle-class-map}
   \begin{aligned}
    \cl_{\dR} \colon \CH^d(X) &\to H^{d,d}_{\dR}(X,\C)^{\vee} \\
    [V] &\mapsto \Bigl( \rho \mapsto \int_{V\setminus V^{\mathrm{sing}}} \rho \Bigr)
   \end{aligned}
  \end{equation}
  where we identify $H^{d,d}_{\dR}(X,\C)$ with its own dual by Poincar\'e duality.
  
 \subsection{A \texorpdfstring{$p$}{p}-adic diagram}
  \label{sec:p-adic-diagram}

  We now suppose that $K$ is either a number field, or a $p$-adic local field, and $X$ a variety over $K$; and we let $H^*_{\et}(X, \Q_p(n))$ denote the continuous $p$-adic etale cohomology of $X$ in the sense of Jannsen (where we regard $X$ as a $K$-variety, not its base-extension to $\overline{K}$). 
  
  As in the previous subsection, we have regulator maps $\ret \colon H_{\mot} \to H_{\et}$ which are compatible with cup products and push-forward maps; see \cite[Proposition~B.4.6 and Appendix~B]{huberwildeshaus}. So for $X$ smooth proper of dimension $2d$ and $0 \le j < d$, as before, we have a diagram
  \begin{equation}
   \label{diag:etale-compatibility}
   \begin{tikzcd}[column sep=small, row sep=large]
    H^{2d+1}_{\mot}(X,2d-j) \times H^{2d}_{\mot}(X,d) \arrow["\cup"]{r} \arrow["\ret \times \ret"]{d} & H^{4d+1}_{\mot}(X,3d-j) \arrow["\pi_*"]{r} \arrow["\ret"]{d} & H^1_{\mot}(K,d-j) \arrow[d,"\ret"] \\
    H_{\et}^{2d+1}(X,\Q_p(2d-j)) \times H_{\et}^{2d}(X,\Q_p(d)) \arrow["\cup"]{r} & H_{\et}^{4d+1}(X,\Q_p(3d-j)) \arrow["\pi_*"]{r} & H_{\et}^1(K,\Q_p(d-j)).
   \end{tikzcd}
  \end{equation}

  We now bring in the Hochschild-Serre spectral sequence: recall that for every $n\in\Z$ its second page is
  \[
   E_2^{ij} = H^i(K,H^j_{\et}(X_{\overline{K}},\Q_p(n))) \implies H^{i+j}_{\et}(X_K,\Q_p(n)).
  \]
  In particular one has a canonical edge map
  \[
   H_{\et}^{2d}(X,\Q_p(d)) \to H^0(K, H^{2d}_{\et}(X_{\overline{K}},\Q_p(d))).
  \]

  On the other hand, we also have an edge map
  \[ H^{2d+1}_{\et}(X,\Q_p(2d-j)) \to H^0(K,H^{2d+1}_{\et}(X_{\overline{K}}, \Q_p(2d-j))).\]
  We define a class in $H^{2d+1}_{\et}(X, 2d-j)$ to be \emph{homologically trivial} if it maps to 0, and write $H^{2d+1}_{\et}(X, 2d-j)_0$ for the group of such classes; then we have a refined edge map
  \[ H^{2d+1}_{\et}(X, 2d-j)_0 \to H^1(K,H^{2d}_{\et}(X_{\overline{K}},\Q_p(2d-j))), \]
  and exactly as in the previous section, we have the following:

  \begin{theorem}
   For $X$ smooth and proper over $K$, the diagram in figure~\ref{diag:etale-diagram} is commutative (where $ H^{2d+1}_{\mot}(X,2d-j)_0$ denotes the preimage of $H_{\et}^{2d+1}(X, \Q_p(2d-j))_0$).
  \end{theorem}

  \begin{figure}[b]
   \centering
   \begin{mdframed}
    \caption{$p$-adic étale diagram}
    \label{diag:etale-diagram}
    \[
     \begin{tikzcd}
      H^{2d+1}_{\mot}(X,2d-j)_0 \times H^{2d}_{\mot}(X,d)
      \arrow["\pi_*\circ \, \cup"]{r}
      \arrow["\ret \times \ret"]{d} &
      H^1_{\mot}(K,d-j) \arrow[d,"\ret"] \\
      H_{\et}^{2d+1}(X, \Q_p(2d-j))_0 \times H_{\et}^{2d}(X,\Q_p(d))
      \arrow["\pi_*\circ \, \cup"]{r} \arrow{d} &
      H^1(K, \Q_p(d-j)) \arrow[equal]{d} \\
      \begin{array}{l}
       H^1(K,H_{\et}^{2d}(X_{\overline{K}}, \Q_p(2d-j))) \\[0.5ex]
       \qquad \times\ H^0(K,H_{\et}^{2d}(X_{\overline{K}},\Q_p(d)))
      \end{array} \arrow["\pi_*\circ \, \cup"]{r} &
      H^1(K,\Q_p(d-j))
     \end{tikzcd}
    \]
   \end{mdframed}
  \end{figure}

  In many cases the homological triviality of classes in $H^{2d+1}_{\mot}(X,2d-j)_0$ is automatic. For example, suppose first that $K$ is a number field. If $X$ is smooth and proper over $K$, then $H^{2d+1}_{\et}(X_{\overline{K}},\Q_p(2d-j))$ is a pure $p$-adic representation of $G_K$ of weight $1+2j-2d$ (i.e.~the geometric Frobenii at good primes $\ell$ have eigenvalues of absolute value $\ell^{(1 + 2j - 2d)/2}$). Since $1+2(j-d) < 0$, there cannot be any morphisms from the trivial representation (which is pure of weight zero) to it. Therefore
  \[
   H^0(K,H^{2d+1}_{\et}(X_{\overline{K}},\Q_p(2d-j))) = 0.
  \]

  A similar argument applies if $K$ is a local field, and $X$ has a smooth proper model over the ring of integers of $K$. If $X$ does not have such a model we do not have good control of ``purity'' for the \'etale cohomology (absent a proof of the weight-monodromy conjecture); but if $K = \Q_p$ and $X$ is the base-extension of a $\Q$-variety, any class in the image of the motivic cohomology over $\Q$ is homologically trivial, by the preceding argument. This is the strategy we will use in Subsection~\ref{subsec:p-adic-argument}.

  When $X$ is smooth and affine (but not necessarily proper), then Artin's vanishing theorem for the \'etale cohomology gives
  \[
   H^i_{\et}(X_{\overline{K}},\Q_p) = 0 \quad \forall i > 2d.
  \]
  Thus any class in $H_{\et}^{2d+1}(X,\Q_p(2d-j))$ is homologically trivial, and we have:

  \begin{proposition}
   \label{prop:morphism-for-affine-etale}
   If $X$ is smooth and affine, then there is a morphism
   \[
    H_{\et}^{2d+1}(X,\Q_p(2d-j)) \to H^1(K,H_{\et}^{2d}(X_{\overline{K}},\Q_p(2d-j))).\myqed
   \]
  \end{proposition}

\section{Cohomology of Kuga-Sato varieties}
\label{sec:cohomology-kuga-sato}

 In this section we introduce the key geometric objects of this paper, namely Kuga--Sato varieties, along with their cohomology; for further details see \cite{deligne69, scholl90} for the theory over $\Q$, and Conrad's appendix to \cite{BDP13} for integral models. For simplicity we assume henceforth that $N \geq 5$, so that all the objects we construct exist as varieties; the case of small $N$ can be handled using stacks, or (more concretely) by introducing full level structure at some auxiliary prime and taking invariants.

 \subsection{Motivic cohomology}

  Let $Y_1(N)$ be the moduli space of elliptic curves with a point of order $N$. The curve $Y_1(N)$ has a model over $\Q$, whose $\C$-points are $\Gamma_1(N) \backslash \HP$ (associating to $\tau \in \HP$ the pair $(\C / (\Z\tau + \Z), \tfrac{1}{N} \bmod \Z\tau + \Z)$). Notice that with our choice of model for $Y_1(N)$, the cusp at infinity is only defined over $\Q(\mu_N)$, not over $\Q$.

  Let $\Ell \xrightarrow{\pi} Y_1(N)$ be the universal elliptic curve. For every $d\in\N_{\geq 1}$ we can form the $d$-th fibre product of $\Ell$ over the modular curve:
  \[
   \Ell^d \xrightarrow{\pi_d} Y_1(N).
  \]
  The variety $\Ell^d$ is $d$-dimensional over $Y_1(N)$ and $(d+1)$-dimensional over $\Spec \Q$. We can further form the fibre product $\Ell^d \times \Ell^d$, where this time we consider both copies of $\Ell^d$ as varieties over $\Spec \Q$. We have then a commutative diagram:
  \[
   \begin{tikzcd}
    \Ell^d \arrow[hook]{r} \arrow["\pi_d"]{d} & \Ell^d \times \Ell^d \arrow["\pi_d \times \pi_d"]{d} \\
    Y_1(N) \arrow[hook]{r} & Y_1(N) \times Y_1(N)
   \end{tikzcd}
  \]

  Let $X_1(N)$ be the Baily--Borel compactification of $Y_1(N)$, and $\overline{\Ell} \to X_1(N)$ the universal generalised elliptic curve. In~\cite{deligne69} Deligne proved the existence of a \emph{canonical} desingularisation which we denote by $W_d$:
  \[
   W_d \to \overline{\Ell}^d \to X_1(N).
  \]
  The variety $W_d$ is smooth and projective over $\Q$ (although the structural morphism $\overline{\pi} \colon W_d \to X_1(N)$ is not smooth). 

  \begin{remark}
   \label{remark:hecke-ops}
   Since the desingularisation is canonical, the Hecke correspondences on $X_1(N)$ associated to the Hecke operators extend as correspondences on $W_d$, as explained in~\cite[§4]{scholl90}. In particular they act on the cohomology of $W_d$.
  \end{remark}

  \begin{definition}
   The \emph{Kuga--Sato variety} $\KS_d$ is the self-fibre product of $W_d$ over the base scheme $\Spec \Q$:
   \[
    \KS_d \to X_1(N)^2, \quad \KS_d \coloneqq W_d \times_{\Spec \Q} W_d.
   \]
  \end{definition}

  \subsection{Cohomology with coefficients}

  Let $\theory\in\{\et,\dR,\mathrm{Betti},\del,\syn,\rig\}$ be a cohomology theory. For each of these theories we have a notion of ``coefficient sheaf'', and we write $\Q_{\theory}$ for the constant coefficient sheaf. Let $\HS_{\theory}$ be the sheaf on $Y_1(N)$ defined as $R^1\pi_* \Q_{\theory}(1)$, where $\pi: \Ell \to Y_1(N)$ is the structure map as above.
  
  \begin{remark}
   Compare \cite[\S 3]{KLZ20}, where we instead defined $\HS_{\theory} = (R^1\pi_* \Q_{\theory})^\vee$; however, these descriptions are equivalent, since Poincar\'e duality identifies this with $R^1\pi_* \Q_{\theory}(1)$.
  \end{remark}

  Let $S_d$ be the group of permutations on $d$ objects. There is an obvious action of $S_d$ on the fibre product $\Ell^d$ given by permuting the coordinates. Likewise, the group $\mu_2$ acts on $\Ell$ as the multiplication by $-1$, and this of course extends to an action of $\mu_2^d$ on $\Ell^d$. We gather these together into an action of the group $\mu_2^d \rtimes S_d = \ideal{I}_d$.

  \begin{definition}
   The character $\epsilon_d$ of $\ideal{I}_d$ is defined by:
   \[
    \epsilon_d \colon \ideal{I}_d \to \mu_2, \quad (a_1,\ldots,a_d,\sigma) \mapsto a_1\cdots a_d \cdot \sgn(\sigma).
   \]
   For a $\Z[\tfrac{1}{2d!}, \ideal{I}_d]$-module $M$ we write $M(\epsilon_d)$ for the eigenspace on which $\ideal{I}_d$ acts via the character $\epsilon_d$.
  \end{definition}

  \begin{proposition}[{\cite[Proposition~4.1.1]{scholl90}}]
   For every cohomology theory $\theory\in\{\et,\dR,\allowbreak\mathrm{Betti},\allowbreak\del,\syn,\rig\}$, there is an isomorphism:
   \[
    H_{\theory}^i(Y_1(N),\TSym^k \HS_{\theory}(j)) \simeq H_{\theory}^{i+k}(\Ell^k, \Q_{\theory}(j+k))(\epsilon_k)
   \]
   and similarly
   \[
    H^i_{\theory}(Y_1(N)^2, \TSym^{[k,k]} \HS_{\theory}(j)) = H^{i+2k}_{\theory}(\Ell^k \times \Ell^k,\Q_{\theory}(j+2k))(\epsilon_k, \epsilon_k).
   \]
  \end{proposition}
  
  \begin{remark}
   Note that Scholl works with the usual symmetric powers $\Sym^k$ rather than $\TSym^k$; but since $k!$ is invertible in the coefficient rings for all our cohomology theories (which is necessary to define the projector), we have a canonical isomorphism between $\Sym^k$ and $\TSym^k$.
  \end{remark}
  
  This can be extended to include motivic cohomology $\theory = \mot$ using the formalism of relative Chow motives as in \cite{ancona15}; and the regulator maps commute with the action of the character $\epsilon_d$, so they go through to
  \[
   \regulator{\theory} \colon H^i_{\mot}(Y_1(N),\TSym^k \HS_{\mot}(j)) \to H^i_{\theory}(Y_1(N),\TSym^k \HS_{\theory}(j)),
  \]
  and likewise for $Y_1(N)^2$.
  
  For $X_1(N)$ we cannot make sense of $\HS_{\mot}$ as a relative Chow motive (since the morphism $\overline{\pi}$ is not smooth); but the projectors $\epsilon_d$ still act on $W_d$, so we can make the following definition:

  \begin{definition}
   Define motivic cohomology groups as follows:
   \[
    \text{``} H^i_{\mot}(X_1(N)^2,\TSym^{[k,k]} \HS_{\mot}(j))\text{''} \coloneqq
    H^{i+2k}_{\mot}(\KS_k, j+2k)(\epsilon_k,\epsilon_k).  \]
  \end{definition}

  We emphasise that in our approach this is a definition, not a theorem; but it is consistent with the theorems of Scholl quoted above for the open variety $Y_1(N)$.

 \subsection{Motives for cusp forms and Rankin--Selberg convolution}
  \label{subsec:motives}

  In this subsection we identify ``pieces'' of various cohomology groups that are naturally associated to cusp forms and their convolutions. This is important for our work, since these objects are smaller than the whole cohomology groups but still retain all the information we need. We follow~\cite[§4]{scholl90}.

  Recall that by the remark at page~\pageref{remark:hecke-ops}, the Hecke operators act on the cohomology of $W_k$ and $W_k\times W_{k'}$. For any Hecke operator $T$, denote by $T'$ the dual Hecke operator, i.e.\ its adjoint with respect to the Petersson inner product. When $\ell$ is a prime not dividing the level, $T_\ell' = T_\ell\langle \ell \rangle^{-1}$.

  From now on, we denote by $\etbar$ étale cohomology over $\overline{\Q}$.
  \begin{definition}
   Let $\theory \in \{\dR, B, \etbar, \rig\}$, assuming $p \nmid N$ if $\theory = \rig$. Define:
   \begin{itemize}
    \item $M_{\theory}(f)$ as the maximal $K_f\otimes \Q_{\theory}$-submodule of
    \[ H^{k+1}_{\theory}(\Ell^k,\Q_{\theory})(\epsilon_k) \otimes_{\Q} K_f \]
    on which the Hecke operators $T_\ell$ act as multiplication by $a_\ell(f)$ for all primes $\ell$;
    \item $M_{\theory}(f)^*$ as the maximal quotient of
    \[ H^{k+1}_{\theory}(\Ell^k,\Q_{\theory}(k+1))(\epsilon_k) \otimes_{\Q} K_f \]
    on which the dual Hecke operators $T_\ell'$ act as multiplication by $a_\ell(f)$ for all primes $\ell$.
   \end{itemize}
   We shall abuse notation slightly by writing $M_{\et}(f)$ rather than $M_{\etbar}(f)$, and similarly $M_{\et}(f)^*$.
  \end{definition}

  Both $M_{\theory}(f)$ and $M_{\theory}(f)^*$ are $2$-dimensional over $\Q_\theory \otimes K_f$, and Poincar\'e duality identifies $M_{\theory}(f)^*$ with the dual of $M_{\theory}(f)$, justifying the notation.

  The different realisations enjoy extra structure:
  \begin{itemize}
   \item $M_{\dR}(f)$ is a filtered $K_f$-vector space;
   \item $M_{\rig}(f)$ has a $K_f$-linear Frobenius $\varphi$;
   \item $M_B(f)$ is a pure Hodge structure over $\Q$ with an additional $\Q$-linear action of $K_f$, whose only non-zero graded pieces are in bidegrees $(0, k+1)$ and $(k+1, 0)$;
   \item $M_{\et}(f)$ is a pure $p$-adic Galois representation of weight $k+1$.
  \end{itemize}

  An important property that these modules enjoy is that they can be found in several cohomology groups. Indeed, one can give the same definition for the compactified variety $W_k$, and it turns out that the resulting objects are canonically isomorphic to those we defined. Additionally, $M_{\theory}(f)$ lifts to the compactly supported cohomology of $\Ell^k$. To sum up, there are \emph{canonical} isomorphisms:
  \[
   \begin{tikzcd}[column sep=small]
    M_{\theory}(f) \arrow[<->,"\simeq"]{r} \arrow[hook]{d} &
    M_{\theory}(f) \arrow[<->,"\simeq"]{r} \arrow[hook]{d} &
    M_{\theory}(f) \arrow[hook]{d} \\
    H_{c,\theory}^{k+1}(\Ell^k, \Q_{\theory})(\epsilon_k) \ar[r] &
    H_{\theory}^{k+1}(W_k, \Q_{\theory})(\epsilon_k) \ar[r] &
    H_{\theory}^{k+1}(\Ell^k, \Q_{\theory})(\epsilon_k)
   \end{tikzcd}
  \]
  and similarly for $M_\theory(f)^*$.

  We also have comparison isomorphisms linking the various realisations:
  \begin{align}
   M_B(f) \otimes \C &\simeq M_{\dR}(f) \otimes \C, \tag{De Rham}\\
   M_{\rig}(f) \otimes \Bcrys &\simeq M_{\dR}(f)_{\Q_p} \otimes \Bcrys \simeq M_{\et}(f)_{\Q_p} \otimes \Bcrys, \tag{$C_{\mathrm{crys}}$} \\
   M_{\rig}(f) \otimes \BdR &\simeq M_{\dR}(f)_{\Q_p} \otimes \BdR \simeq M_{\et}(f)_{\Q_p} \otimes \BdR \tag{$C_{\dR}$}.
  \end{align}

  For the Rankin convolution one defines the motives using the above ones as building blocks.
  \begin{definition}
   Let $\theory\in\{\dR,B,\et,\rig\}$. Define:
   \begin{itemize}
    \item $M_{\theory}(f\otimes g) \coloneqq M_{\theory}(f) \otimes M_{\theory}(g)$;
    \item $M_{\theory}(f\otimes g)^* \coloneqq M_{\theory}(f)^* \otimes M_{\theory}(g)^*$.
   \end{itemize}
   Both $M_{\theory}(f\otimes g)$ and $M_{\theory}(f\otimes g)^*$ are $4$-dimensional, as the building blocks have dimension $2$.
  \end{definition}

  Similar remarks apply to these spaces: they lift to the cohomology of $W_k \times W_{k'}$, or to the compactly supported cohomology of $\Ell^k \times \Ell^{k'}$; and they are related by the comparison theorems.

  Notice that by functoriality the projections
  \begin{gather*}
   \pr_f \colon H^{k+1}_{\theory}(W_k,\Q_{\theory}(k+1)) \twoheadrightarrow M_{\theory}(f)^*, \\
   \pr_{f,g} \colon H^{k+k'+2}_{\theory}(W_k\times W_{k'},\Q_{\theory}(2+k+k')) \twoheadrightarrow M_{\theory}(f\otimes g)^*.
  \end{gather*}
  are compatible with the extra structures on each cohomology theory. More precisely:

  \begin{itemize}
   \item for the Betti realisation, the maps $\pr_{f,g}$ are morphisms of pure Hodge structures, so we have maps of $\Ext$ groups (for any $i, j$)
    \[
     \Ext_{\MH_\R}^i(\R, H^{k+k'+2}_B(W_k\times W_{k'},\R(2+k+k' - j))) \otimes_{\R} \C
      \xrightarrow{\pr_{f,g}} \Ext^i_{\MH_\C}(\C, M_B(f\otimes g)_{\C}^*(-j)).
    \]
    \item for the \'etale realisation, there are maps in Galois cohomology
    \[
     H^i(\Q, H^{k+k'+2}_{\et}((W_k\!\times\! W_{k'})_{\overline{\Q}},\Q_p(2+k+k'-j))) \xrightarrow{\pr_{f,g}} H^i(\Q, M_{\et}(f\otimes g)^*(-j)).
    \]
  \end{itemize}

 \subsection{The case $f = g$}
  \label{sect:feqg}

  When $f=g$, $M_{\theory}(f\otimes f)$ enjoys the standard decomposition in symmetric and antisymmetric tensors:
  \[
   M_{\theory}(f\otimes f) \simeq \Sym^2 M_{\theory}(f) \oplus \wedge^2 M_{\theory}(f).
  \]
  Let $s \colon M_{\theory}(f\otimes f) \to M_{\theory}(f\otimes f)$ be the involution swapping the components of the tensor product, then the above are by definition the $s=1$ and $s=-1$ eigenspaces. It is a standard fact that $\Sym^2 M_{\theory}(f)$ is $3$-dimensional, while $\wedge^2 M_{\theory}(f)$ is $1$-dimensional. This decomposition mirrors~\eqref{eq:decomposition-rho} in various cohomology theories, in particular it is exactly that decomposition of Galois representations when $\theory = \et$. Correspondingly we have:
  \begin{equation}
   \label{eq:decomposition-m}
   M_{\theory}(f\otimes f)^* \simeq \Sym^2 M_{\theory}(f)^* \oplus \wedge^2 M_{\theory}(f)^*.
  \end{equation}

  \begin{remark}
   Note that the automorphism $\rho'$ of $\KS_k = W_k \times W_k$ which swaps the two factors acts on $M_{\theory}(f \otimes f)$ (or its dual) not as $s$, but rather as $(-1)^{(k+1)} s$. This is because the identification of $M_\theory(f) \otimes M_\theory(f)$ with a subspace of the cohomology of $\KS_k$ is given by the cup product pairing, which is graded-commutative rather than commutative, and $M_{\theory}(f)$ lives in cohomological degree $k + 1$. Similarly, the symmetry involution $\rho$ of $Y_1(N)^2$ (treating the coefficient sheaves as living in degree 0) acts as $(-1)^k \rho' = -s$.
  \end{remark}

  Composing the projections $\pr_{f, f}$ with the projections induced by the direct sum decomposition~\eqref{eq:decomposition-m}, we obtain maps
  \[
   \begin{tikzcd}
    \Ext_{\MH_\R}^i(\R, H^{2k+2}_B(\KS_k,\R(2k+2))) \otimes \C \arrow["\pr_{f,f}"]{d} & \Ext^i_{\MH_\C}(\C, \wedge^2 M_B(f)_\C^*) \\
    \Ext^i_{\MH_\C}(\C, M_B(f\otimes f)^*) \arrow{r} \arrow{ur} &
    \Ext^i_{\MH_\C}(\C,\Sym^2 M_B(f)^*_\C)
   \end{tikzcd}
  \]
and the analogous ones in the category of Galois modules:
\[
    \begin{tikzcd}
        H^i(\Q,H^{2k+2}_{\et}(\KS_{k,\overline{\Q}},\Q_p(2k+2))) \arrow["\pr_{f,f}"]{d} & H^i(\Q,\wedge^2 M_{\et}(f)^*)\\
        H^i(\Q,M_{\et}(f\otimes f)^*) \arrow{r} \arrow{ur} & H^i(\Q,\Sym^2 M_{\et}(f)^*)
    \end{tikzcd}
\]

 \subsection{De Rham cohomology and modular forms}
\label{sec:de-rham-kuga-sato}

  In this subsection we derive explicit descriptions for the de Rham cohomology of Kuga-Sato varieties. To start with, we have the following description of the canonical filtration restricted to the $\epsilon_k$-eigenspaces, as explained in~\cite[202--205]{kato04}:
  \[
   \Fil^i H^{k+1}_{\dR}(W_k,\C)(\epsilon_k) =
   \begin{cases}
    H^{k+1}_{\dR}(W_k, \C)(\epsilon_k) &\text{$i\leq 0$} \\
    S_{k+2}(N) &\text{$0<i\leq k+1$} \\
    0 &\text{$i>k+1$}
   \end{cases}
  \]
  which implies that we have cusp forms spaces in degree $k+1$ while all the rest is concentrated in degree $0$. Thus the Hodge decomposition is given by
  \[
   H^{k+1,0}_{\dR}(W_k,\C)(\epsilon_k) = S_{k+2}(N),\qquad H^{0,k+1}_{\dR}(W_k,\C)(\epsilon_k) = \overline{S_{k+2}(N)},
  \]
  and hence
  \begin{equation}
   \label{eq:cohomology-decompositions}
   H^{k+1}_{\dR}(W_k,\C)(\epsilon_k) \simeq S_{k+2}(N) \oplus \overline{S_{k+2}(N)}.
  \end{equation}

  For $\KS_k$, using the K\"unneth formula and \eqref{eq:cohomology-decompositions} we obtain
  \begin{subequations}
   \label{eq:kunneth-components}
   \begin{align}
    H^{2k+2,0}_{\dR}(\KS_k,\C)(\epsilon_k,\epsilon_k) &= S_{k+2}(N)\otimes S_{k+2}(N), \label{eq:hol-kunneth} \\
    H^{k+1,k+1}_{\dR}(\KS_k,\C)(\epsilon_k,\epsilon_k) &= (S_{k+2}(N)\otimes \overline{S_{k+2}(N)}) \oplus (\overline{S_{k+2}(N)} \otimes S_{k+2}(N)), \label{eq:middle-kunneth}\\
    H^{0,2k+2}_{\dR}(\KS_k,\C)(\epsilon_k,\epsilon_k) &= \overline{S_{k+2}(N)} \otimes \overline{S_{k+2}(N)}. \label{eq:nonhol-kunneth}
   \end{align}
  \end{subequations}

\section{The Beilinson--Flach Euler system}

 In this section we define the cohomology classes that we will use as ingredients to construct our motivic classes. We first introduce Beilinson--Flach classes and explain how they can be extended to the cohomology of Kuga-Sato varieties, and then exploit some differential forms attached to cusp forms.

 \subsection{From Eisenstein to Beilinson--Flach classes}

  In this subsection we introduce the Beilinson--Flach classes, constructed in~\cite{leiloefflerzerbes14,KLZ20} as the push forward of Eisenstein classes (constructed by Beilinson) from the motivic cohomology of modular curves, to that of the product of two modular curves.

  \begin{definition}[{\cite[Definition~5.3.1]{KLZ20}}]
   The Beilinson--Flach classes defined over $\Q$ are motivic classes
   \[
    \BF^{[k,k',j]}_{\mot,N} \in H^3_{\mot}(Y_1(N)^2,\TSym^{[k,k']} \HS_{\mot}(2-j)).
   \]
   We will suppress the subscript $N$, since variation in $N$ plays no role in this paper. For any theory $\theory\in\{\et,\dR,\mathrm{Betti},\del,\syn,\rig\}$, we write $\BF^{[k,k',j]}_{\theory} \coloneqq \regulator{\theory}(\BF^{[k,k',j]}_{\mot})$.
  \end{definition}

  We note that in the case of parallel weights $k = k'$, the involution $\rho'$ induced by swapping the two factors of $\KS_k$ (see \S \ref{sect:feqg} above) acts on $\BF^{[k,k,j]}_{\mot,N}$, and hence also on its realizations in other theories $\theory$, as $(-1)^{k + j}$. \footnote{Compare \S 5 of \cite{KLZ17}, where the statement is that the symmetry involution $\rho$ of $Y_1(N)^2$ acts as $(-1)^j$. As we have seen above, $\rho' = (-1)^k \rho$.}

  \begin{remark}
   If $N < 4$, then $Y_1(N)$ does not exist as a moduli space; but we can still define $\BF^{[k,k',j]}_{\mot, N}$ by working with auxiliary level structure and taking invariants, \emph{unless} $N = 1$ and $j = k = k'$, since in this case there is no sensible definition of the Siegel unit ${}_c g_{0, 1/N}$.
  \end{remark}

 \subsection{Compactification of Beilinson--Flach classes}
 \label{subsec:compactification-bf}

  In this subsection we recall results of Brunault and Chida \cite[\S 7-8]{brunaultchida16} which will allow us to lift the Beilinson--Flach classes from the open Kuga--Sato variety $\Ell^k \times \Ell^k$ to its compactification $\KS_k$, in order that we can apply the Bloch intersection pairing.
  
  \begin{remark}
   Brunault and Chida work with a slightly more general version of the Beilinson--Flach class depending on a parameter $\beta \in \Q[(\Z / N\Z)^2]$. Here we are taking $\beta$ to be the characteristic function of $(0, 1) \in (\Z / N\Z)^2$, in which case the elements considered in \emph{op.cit.}~reduce to those studied in \cite{KLZ20}.
   
   Note that there are also some notational conflicts, since \cite{brunaultchida16} uses the symbol ``$\BF$'' to denote a class on the compactification, and ``$\Xi$'' for the un-compactified class, while we do almost the opposite; we are grateful to the referee for pointing out this incompatibility. Unfortunately there seems to be no way to resolve this without breaking compatibility with \cite{KLZ20}; we apologise for any confusion caused.\end{remark}

  We denote by $\hat{\Ell}^d$ the locus of $\overline{\Ell}^d$ where the projection to $X_1(N)$ is smooth, i.e.\ the Néron model of $\Ell^d$ over $X_1(N)$, and by $\hat{\Ell}^{d,*}$ the connected component of the identity. We write $Z^k \coloneqq \hat{\Ell}^{k,*} \setminus \Ell^k$ (the preimage of the cusps in $\hat{\Ell}^{k,*}$), and:
  \begin{align*}
   \Ell^{k,k'} &\coloneqq \Ell^k \times \Ell^{k'}, &\; \hat{\Ell}^{k,k',*} &\coloneqq \hat{\Ell}^{k,*}\times \hat{\Ell}^{k',*}, \\
   Z^{k,k'} &\coloneqq Z^k\times Z^{k'}, & U^{k,k'} &\coloneqq \hat{\Ell}^{k,k',*} \setminus Z^{k,k'}.
  \end{align*}
  Note $U^{k,k'}$ is strictly larger than $\Ell^{k, k'}$; it contains both $Z^k \times \Ell^{k'}$ and $\Ell^k \times Z^{k'}$. However, the composite map
  \[
      \Ell^{k+k'} \hookrightarrow \Ell^{k, k'} \hookrightarrow U^{k,k'}
  \]
  is still a closed embedding.
  
  \begin{theorem}[{\cite[Proposition~9.1]{brunaultchida16}}]
   \label{thm:brunault-chida}
   For every quadruple $(i,d,d',j) \in \N^4$ with $d,d' \geq 1$, we have
   \begin{align*}
    H_{\mot}^i(W_d \times W_{d'}, j)(\epsilon_d,\epsilon_{d'}) &\simeq H_{\mot}^i(\hat{\Ell}^{d,d',*}, j)(\epsilon_d,\epsilon_{d'})
    \intertext{in particular}
    H_{\mot}^i(\KS_d, j)(\epsilon_d,\epsilon_d) &\simeq H_{\mot}^i(\hat{\Ell}^{d,d,*}, j)(\epsilon_d,\epsilon_d).
   \end{align*}
  \end{theorem}

Brunault and Chida define lifted classes $\lBF^{[k,k',j]}_{\mot}$ as the push-forward of Eisenstein classes to the cohomology of $U^{k,k'}$ rather than of $\Ell^{k,k'}$. The lifted classes then live in $H^{3+k+k'}_{\mot}(U^{k,k'},2+k+k'-j)$ and pull-back to the standard ones. The localisation sequence $Z^{k,k'} \to \hat{\Ell}^{k,k',*} \to U^{k,k'}$ induces a long exact sequence
\begin{multline*}
    \cdots\to H^{3+k+k'}_{\mot}(\hat{\Ell}^{k,k',*},2+k+k'-j)(\epsilon_k,\epsilon_{k'}) \\
    \to H^{3+k+k'}_{\mot}(U^{k,k'},2+k+k'-j)(\epsilon_k,\epsilon_{k'}) \\
    \xrightarrow{\Res} H^{k+k'}_{\mot}(Z^{k,k'},k+k'-j)(\epsilon_k,\epsilon_{k'}) \to\cdots
\end{multline*}
The residue $\Res(\lBF^{[k,k',j]}_{\mot})$ is non-zero only when $j=0$, by \cite[Proposition 8.1]{brunaultchida16}. On one hand, for $j>0$ by exactness the lifted class has a preimage in the first group. On the other hand, when $j=0$ one deduces the same thesis by diagram chasing: the ``cuspidal embedding'' $i_{\mathrm{cusp}} \colon Z^k \times \Ell^{k'} \hookrightarrow U^{k,k'}$ induces a push-forward morphism
\[
    H^{1+k+k'}_{\mot}(Z^k \times \Ell^{k'}, 1+k+k') \xrightarrow{i_{\mathrm{cusp},*}} H^{3+k+k'}_{\mot}(U^{k,k'},2+k+k').
\]
It turns out that there exists an element $\xi \in H^{k+k'+1}_{\mot}(Z^k \times \Ell^{k'}, k+k'+1)(\epsilon_k,\epsilon_{k'})$ satisfying $\Res\circ i_{\mathrm{cusp},*}(\xi) = \Res(\lBF^{[k,k',0]}_{\mot})$. This shows that the element
\begin{equation}
 \label{eq:brunault-chida-lift}
 \begin{cases} \lBF^{[k,k',0]}_{\mot} - i_{\mathrm{cusp},*}(\xi) &\text{if $j=0$}\\ \lBF^{[k,k',j]}_{\mot} &\text{if $j>0$} \end{cases} \tag{$\sharp$}
\end{equation}
has residue always equal to zero. We then denote by $\lcBF^{k,k',j}$ an arbitrary choice of a preimage of this element in $H^{3+k+k'}_{\mot}(\hat{\Ell}^{k,k',*},2+k+k'-j)(\epsilon_k,\epsilon_{k'})$. Recall now that Theorem~\ref{thm:brunault-chida} establishes the existence of an isomorphism
\[
    H_{\mot}^{3+k+k'}(W_k \times W_{k'}, j)(\epsilon_k,\epsilon_{k'}) \simeq H_{\mot}^{3+k+k'}(\hat{\Ell}^{k,*} \times \hat{\Ell}^{k',*}, j)(\epsilon_k,\epsilon_{k'})
\]
which shows that $\lcBF^{k,k',j}$ lifts to $W_k\times W_{k'}$, which is $\KS_k$ when $k=k'$. Reversing the argument shows that $\lcBF^{k,k',j}$ pulls back to $\BF^{[k,k',j]}_{\mot}$. Indeed, its pull back from $\hat{\Ell}^{k,k',*}$ to $\Ell^{k,k'}$ factors through $U^{k,k'}$. But by definition, its image in the cohomology of $U^{k,k'}$ is the class~\eqref{eq:brunault-chida-lift}, which clearly pulls back to $\BF^{[k,k',j]}_{\mot}$ as the cuspidal contribution maps to zero.

Hence $\lcBF^{k,k',j}$ is a lift of $\BF^{[k,k',j]}_{\mot}$ in the cohomology of $W_k\times W_{k'}$, as wanted. These motivic classes are the \emph{first version} of compactified Beilinson--Flach classes, as defined by Brunault and Chida. For any cohomology theory $\theory\in\{\et,\dR,\mathrm{Betti},\del,\syn,\rig\}$ we write $\lcBF^{k,k',j}_{\theory} \coloneqq \regulator{\theory}(\lcBF^{k,k',j})$.
\begin{remark}
    \label{rmk:pairing-unchanged}
    Notice that the defect between $\BF^{[k,k',j]}_{\mot}$ and $\lcBF^{k,k',j}$ is supported purely on $Z^k\times \Ell^{k'}$. Therefore, its pairing with differential forms attached to cusp forms vanishes. Hence, as long as we stick to cusp forms, the classes $\lcBF^{k,k',j}$ and $\BF^{[k,k',j]}_{\mot}$ satisfy the same regulator formulæ. A precise statement is given in~\cite[Proposition~9.4]{brunaultchida16}, which we will recall later.
\end{remark}

We now define a second version of compactified Beilinson--Flach classes for the coincident weights case. Since the kernel of the map induced by $\hat{\Ell}^{k,k',*} \to U^{k,k'}$ is not guaranteed to be trivial, the class $\lcBF^{k,k',j}$ is not uniquely determined in general. Suppose now $k=k'$ and let $\rho'$ be the involution of $\KS_k = W_k\times W_k$ swapping the two components of the product.

 \begin{definition}
  The \emph{second version} of the motivic compactified Beilinson--Flach classes is
  \[
      \cBF^{k,k,j} \coloneqq \frac{1}{2}(\lcBF^{k,k,j} + (-1)^{k+j}(\rho')^*(\lcBF^{k,k,j})) \in H^{3+2k}_{\mot}(\hat{\Ell}^{k,k,*},2+2k-j)(\epsilon_k,\epsilon_k).
  \]
  By the same argument as before, we regard $\cBF^{k,k,j}\in H^{3+2k}_{\mot}(\KS_k,2+2k-j)(\epsilon_k,\epsilon_k)$. For any cohomology theory $\theory\in\{\et,\dR,\mathrm{Betti},\del,\syn,\rig\}$ we write $\cBF^{k,k',j}_{\theory} \coloneqq \regulator{\theory}(\cBF^{k,k',j})$.
 \end{definition}

 The element $\cBF^{k,k,j}$ is still a lifting of the non-compactified classes: its pull-back to $\Ell^{k,k}$ coincides with
 \[
  \frac{1}{2}(\BF^{[k,k,j]}_{\mot} + (-1)^{k+j}(\rho')^*\BF^{[k,k,j]}_{\mot}) = \frac{1}{2}(\BF^{[k,k,j]}_{\mot} + \BF^{[k,k,j]}_{\mot}) = \BF^{[k,k,j]}_{\mot}
 \]
 where we have used the fact that $(\rho')^*$ acts as $(-1)^{k+j}$ on $\BF^{k,k,j}$, as remarked in the previous section. Thus $\cBF^{k,k,j}$ is a lifting of $\BF^{k,k,j}$ compatible with the $\rho'$-action.

\subsection{Classes in the cohomology of \texorpdfstring{$f$}{f}}
In this subsection we explain how it is possible to pass from Beilinson--Flach classes associated to the triple $(k,k',j)$ to classes associated to the triple $(f,g,j)$. In the sequel we will mainly be interested in the case $f=g$. It is important to underline that the arguments of this subsection apply only to a restricted choice of cohomology theories, in particular they do \emph{not} apply to motivic classes.

\subsubsection{Classes in Deligne cohomology}
In Deligne cohomology we have classes $\lcBF^{k,k',j}_{\del}$. Following Subsection~\ref{sec:complex-diagram} there is an isomorphism
\begin{multline*}
    H^{3+k+k'}_{\del}(W_k \times W_{k'}, \C(2+k+k'-j)) \\
    \xrightarrow{\simeq} \Ext^1_{\MH_{\C}}(\C, H^{2+k+k'}_B(W_k\times W_{k'}, \C(2+k+k'-j)))
\end{multline*}
which we compose with the functorial morphism induced by $\pr_{f,g}$
\begin{multline*}
    \Ext^1_{\MH_{\C}}(\C, H^{2+k+k'}_B(W_k\times W_{k'},\C(2+k+k'-j))) \\
    \xrightarrow{\pr_{f,g}} \Ext^1_{\MH_{\C}}(\C, M_B(f\otimes g)_{\C}^*(-j)).
\end{multline*}
We define $\lcBF^{f,g,j}_{\del}$ as the image of $\lcBF^{k,k',j}_{\del}$ under this composition. When $k=k'$ we define $\cBF^{f,g,j}_{\del}$ as the image of $\cBF^{k,k,j}_{\del}$ under the same composition.

The first isomorphism still exists when we replace $W_k\times W_{k'}$ with $\Ell^{k, k'}$. This follows from the isomorphism expressing the cohomology of $\Ell^{k, k'}$ as the cohomology of $Y_1(N)^2$, which is affine, and by Proposition~\ref{prop:isomorphism-for-affine-deligne}. Indeed, $Y_1(N)^2$ has dimension $2$, hence its third Betti cohomology group is trivial, and obviously does not admit non-zero morphisms from the trivial structure. Furthermore, the morphism $\pr_{f,g}$ is defined from the $\Ext$-groups with coefficients in the cohomology of $\Ell^{k, k'}$ too. Hence the whole composition is well-defined even when starting from the Deligne cohomology of $\Ell^{k, k'}$. We then define $\BF^{[f,g,j]}_{\del}$ as the image $\BF^{[k,k',j]}_{\del}$ under it.

\subsubsection{Classes in étale cohomology}
In étale cohomology we have classes $\lcBF^{k,k',j}_{\et}$. Following Subsection~\ref{sec:p-adic-diagram} there is a morphism
\begin{multline*}
    H^{3+k+k'}_{\et}(W_k \times W_{k'}, \Q_p(2+k+k'-j)) \\
    \to H^1(\Q(\mu_N),H^{2+k+k'}_{\et}((W_k\times W_{k'})_{\overline{\Q}},\Q_p(2+k+k'-j)))
\end{multline*}
which we compose with the functorial morphism induced by $\pr_{f,g}$
\begin{multline*}
    H^1(\Q(\mu_N),H^{2+k+k'}_{\et}((W_k\times W_{k'})_{\overline{\Q}},\Q_p(2+k+k'-j))) \\
    \xrightarrow{\pr_{f,g}} H^1(\Q(\mu_N),M_{\et}(f\otimes g)^*(-j)).
\end{multline*}
We define $\lcBF^{f,g,j}_{\et}$ as the image of $\lcBF^{k,k',j}_{\et}$ under this composition. When $k=k'$ we define $\cBF^{f,g,j}_{\et}$ as the image of $\cBF^{k,k,j}_{\et}$ under the same composition.

The first morphism still exists when we replace $W_k\times W_{k'}$ with $\Ell^{k,k'}$. As before, this follows from the identification of the cohomology of $\Ell^{k,k'}$ and of $Y_1(N)^2$, which is affine, and by Proposition~\ref{prop:morphism-for-affine-etale} (again because the third cohomology group of the $2$-dimensional affine scheme $Y_1(N)^2$ is trivial when we base change to $\overline{\Q}$). Furthermore, the morphisms $\pr_{f,g}$ is defined from the Galois cohomology groups with coefficients in the cohomology of $\Ell^{k,k'}$ too. Hence the whole composition is well-defined also when starting from the étale cohomology of $\Ell^{k,k'}$. We then define $\BF^{[f,g,j]}_{\et}$ as the image of $\BF^{[k,k',j]}_{\et}$ under it.

From the above discussion, when $f=g$ we have three different classes in the cohomology of $M_{\bullet}(f\otimes f)^*(-j)$ for $\bullet\in\{B,\et\}$. Explicitly
\begin{align*}
    \BF^{[f,f,j]}_{\del},\; \lcBF^{f,f,j}_{\del},\; \cBF^{f,f,j}_{\del} &\in \Ext^1_{\MH_\C}(\C, M_B(f\otimes f)^*(-j)), \\
    \BF^{[f,f,j]}_{\et},\; \lcBF^{f,f,j}_{\et},\; \cBF^{f,f,j}_{\et} &\in H^1(\Q(\mu_N),M_{\et}(f\otimes f)^*(-j)).
\end{align*}

We record here a useful result which we will need in Subsection~\ref{subsec:integrality}.

\begin{proposition}
    \label{prop:classes-of-f-coincide}
    The three classes $\BF^{[f,f,k]}_{\et}$, $\lcBF^{f,f,k}_{\et}$ and $\cBF^{f,f,k}_{\et}$ coincide, as do the three classes $\BF^{[f,f,k]}_{\del}$, $\lcBF^{f,f,k}_{\del}$ and $\cBF^{f,f,k}_{\del}$.
\end{proposition}

\begin{proof}
 This follows from the fact that restriction from $\KS_k$ to the open variety $\Ell^{k, k}$ induces an isomorphism on the $(f, f)$-isotypic component (in Betti cohomology, and hence also in \'etale cohomology over $\overline{\Q}$).
\end{proof}

\subsubsection{Properties under swapping involutions}
\label{subsubsec:swapping}
We collect here some properties regarding the behaviour of both compactified and non-compactified Beilinson--Flach classes under three different involutions, when $f=g$.

For $\theory\in\{\del,\et\}$ the classes $\BF^{[f,f,j]}_{\theory}$ are in cohomology groups with coefficients in $M_{\theory'}(f\otimes f)^*(-j)$ for $\theory'\in\{B,\et\}$ respectively. As we have noted above, if $s$ is the involution swapping the two factors of $M_{\theory'}(f\otimes f) = M_{\theory'}(f) \otimes M_{\theory'}(f)$, then we have $s = - \rho^* = (-1)^{k+1} (\rho')^*$. So
\[
    s(\BF^{[f,f,j]}_{\theory}) = -\rho^* \BF^{[f,f,j]}_{\theory} = (-1)^{k+1}(\rho')^* \BF^{[f,f,j]}_{\theory} = (-1)^{j+1}\BF^{[f,f,j]}_{\theory}.
\]
Therefore, the parity of $j$ determines whether Beilinson--Flach classes take values in $\Sym^2 M_{\theory'}(f)^*$ or in $\wedge^2 M_{\theory'}(f)^*$: they land in $\wedge^2$ if $j$ is even, and $\Sym^2$ if $j$ is odd.

\begin{remark}
Controlling when Beilinson--Flach classes have coefficients in $\wedge^2 M_{\theory'}(f)^*$ is crucial for us. Indeed, we want to construct classes to compare with higher cyclotomic classes, drawing from the decomposition of Galois representations. Higher cyclotomic classes give values for the $L$-function associated to the representation $\wedge^2 \rho_{f,v}$, hence if we want to construct classes to compare them with, the only natural choice is the $\wedge^2$-component. The functional equation will allow us to propagate this information to the $\Sym^2$ component.
\end{remark}

\subsection{Differential forms attached to cusp forms}
\label{subsec:differential-forms}

 In this subsection we seek a differential form to pair with the image of the compactified Beilinson--Flach classes under $\rD$ (see diagram in figure~\ref{diag:complex-compatibility}). Such a differential would belong to the cohomology group $H^{k+1,k+1}_{\dR}(\KS_k, \C)(\epsilon_k)$. In order to construct it we start from modular forms. Recall from Subsection~\ref{sec:de-rham-kuga-sato} that the middle Hodge component is isomorphic to $(S_{k+2}(N)\otimes \overline{S_{k+2}(N)}) \oplus (\overline{S_{k+2}(N)} \otimes S_{k+2}(N))$. Starting from cusp forms, the Eichler-Shimura isomorphism associates:
 \begin{align*}
     S_{k+2}(N) &\xrightarrow{\simeq} \Fil^1 H^{k+1}_{\dR}(W_k,\C)(\epsilon_k) \\
     f &\mapsto \omega_f
 \end{align*}
 where $\omega_f$ is the unique $(k+1)$-form whose pull-back to $Y_1(N)$ (through $\Ell^k$) is
 \[
     (2\pi i)^{k+1} f(\tau)\ud\tau w^{(k,0)} = (2\pi i)^k f(q) \frac{\ud q}{q} w^{(k,0)} \in H^{1}_{\dR}(Y_1(N),\TSym^k \HS_{\dR,\C}).
 \]
 Here we are denoting with $w = \ud z$ the standard section of $\HS_{\dR, \C}(-1)$, and $w^{(r,s)} = w^r \overline{w}^s \in \TSym^k \HS_{\dR, \C}(-k)$. This prescription does not only define a de Rham cohomology class, but a true differential form.

 Since the Eichler-Shimura isomorphism is Hecke-equivariant, the form $\omega_f$ is in the $f$-isotypical component of the de Rham cohomology group. The dual of $\omega_f$ under the Poincaré pairing belongs to the dual of $H^{k+1,0}_{\dR}(W_k,\C)$, which is $H^{0,k+1}_{\dR}(W_k,\C)$, and we denote the basis vector dual to $\omega_f$ by $\eta_{f^*}$. We can actually make this explicit:
 \[
     \eta_{f^*} = \frac{1}{(\omega_f,\overline{\omega_f})} \overline{\omega_f} \in H^{0,k+1}_{\dR}(W_k,\C)(\epsilon_k) \simeq \overline{S_{k+2}(N)}.
 \]
 The form $\eta_{f^*}$ lies in the $f^*$-isotypical component, as the Hecke operators act on it with eigenvalues which are conjugates of those of $f$. We also have the relation 
 \begin{equation}
  \label{eq:petersson}
  (\omega_f,\overline{\omega_f}) = (-4\pi)^{k+1} \langle f,f \rangle.
 \end{equation}
 
 Suppose now that $f$ is a \emph{newform}. By Strong Multiplicity One $\omega_f$ spans the $f$-isotypical component inside the $(k+1,0)$ Hodge component, which is $1$-dimensional. Analogously, $\eta_f$ spans the $f$-isotypical component inside the $(0,k+1)$ Hodge component.
 
 \begin{lemma}
  If $f,g \in S_{k+2}(N)^{new}$ are newforms, then the $(f,g)$-isotypical component $M_{\dR}(f\otimes g) \otimes \C \subseteq H^{2k+2}_{\dR}(\KS_k, \C)(\epsilon_k)$ is $4$-dimensional with basis
  \[
    \{\omega_f\otimes \omega_g, \omega_f \otimes \eta_g, \eta_f\otimes \omega_g, \eta_f\otimes \eta_g\}.
  \]
 \end{lemma}
 
 \begin{proof}
  This follows from the decompositions~\eqref{eq:kunneth-components} simply by taking the relevant eigenspace.
 \end{proof}

Let $\Omega_{f,g} \coloneqq \omega_f \otimes \eta_g - \eta_f\otimes \omega_g \in M_{\dR}(f\otimes g)$. It is clear that $\Omega_{f,g} \neq 0$ as it comes from the composition of two linearly independent differential forms, hence it determines a $1$-dimensional subspace filtration:
\[
    0 \subseteq \langle \Omega_{f,g} \rangle \subseteq M_{\dR}(f\otimes g) \otimes \C.
\]

When $f=g$, $M_{\dR}(f\otimes f)$ already enjoys a filtration given by the direct sum decomposition:
\[
    M_{\dR}(f\otimes f) \simeq \wedge^2 M_{\dR}(f) \oplus \Sym^2 M_{\dR}(f)
\]
induced by the action of the involution swapping the two components in the tensor product. As $M_{\dR}(f)$ is $2$-dimensional, $\wedge^2 M_{\dR}(f\otimes f)$ is $1$-dimensional, and since $\Omega_{f,f}$ is non-zero and belongs to the anti-symmetric component by construction, the following result follows.
\begin{proposition}
    $\Omega_{f,f} \in \wedge^2 M_{\dR}(f)$ and it generates the anti-symmetric component.
\end{proposition}
This amounts to say that the above filtrations coincide. We note that $\Omega_{f,f}$ always lies in the $\wedge^2$-component, in accordance with the fact that we want to exploit classes there, see Subsubsection~\ref{subsubsec:swapping}.

\section{Regulator formulae}
In this section we prove complex and $p$-adic regulator formulae linking compactified Beilinson--Flach classes to special values of $L$-functions.

\subsection{Archimedean argument}
\label{subsec:complex-argument}
\subsubsection{A complex regulator formula}
\label{subsubsec:complex-formula}
We explain here the complex regulator formula, following~\cite{brunaultchida16,KLZ20}.

Recall diagram in figure~\ref{diag:complex-diagram} at page~\pageref{diag:complex-diagram} and specialise it to $X = \KS_k$, which has dimension $2k+2$ over $\Q$. As we will compute the pairing using the bottom row, we first analyse carefully the cohomology group appearing there. Even though $\KS_k$ is defined over $\Q$, we regard it as being defined over $K_f$, as in~\ref{subsubsec:regulator-motivic-cohomology} we will need to extend coefficients to $K_f$. By the isomorphisms~\eqref{eq:ext-groups} we obtain
\begin{align*}
    \Ext^0_{\MH_\R}(\R, K) &\simeq W_0 K \cap F^0(K \otimes \C), \\
    \Ext^1_{\MH_\R}(\R, H) &\simeq \frac{(W_0 H \otimes \C)}{W_0 H + F^0(H\otimes \C)}.
\end{align*}
with $H = H_B^{2k+2}(\KS_k, \R(2+2k-j))$, $K = H_B^{2k+2}(\KS_k, \R(k+1))$.

 Moreover, since $K = H^{2k+2}_B(\KS_k, \R(k+1))$ is concentrated in weight $0$, the first isomorphism becomes
 \[\begin{split}
      \Ext^0_{\MH_\R}&(\R, H_B^{2k+2}(\KS_k, \R(k+1))) \\
      &= H^{2k+2}_B(\KS_k, \R(k+1)) \cap \Fil^{k+1} H^{2k+2}_{\dR}(\KS_k, \C).
  \end{split}\]
 So after base-extension to $\C$ we simply obtain the bidegree $(k+1,k+1)$ part
 \begin{equation} 
  \label{eq:isomorph-ext-0}
  \Ext^0_{\MH_\C}(\C, H_B^{2k+2}(\KS_k, \C(k+1))) = H^{k+1, k+1}_{\dR}(\KS_k, \C).
 \end{equation}

Regarding the $\Ext^1$, we attach the functorial morphism $\pr_{f,f}$ induced by the projection $H^{2k+2}_B(\KS_k, \C(2k+2-j)) \twoheadrightarrow M_B(f\otimes f)_\C^*(-j)$ in $\MH_\C$, as explained in Subsection~\ref{subsec:motives}. In this way we have obtained the following diagram, where the $\Ext$-groups are all in the category $\MH_\C$:
\begin{equation}
    \label{diag:complex-ext}
    \begin{tikzcd}[column sep=small]
        H_{\mot}^{3+2k}(\KS_k,2+2k-j) \arrow["\rD"]{d} & H_{\mot}^{2k+2}(\KS_k,k+1) \arrow["\rD"]{d} \\
        H_{\del}^{3+2k}(\KS_k, \R(2+2k-j)) \arrow["\mathrm{edge}"]{d} & H_{\del}^{2k+2}(\KS_k,K_f(k+1)) \arrow["\text{edge}"]{d} \\
        \Ext^1(\C,H^{2k+2}_B(\KS_k, \C(2+2k-j))) \arrow["\pr_{f,f}"]{d} & \Ext^0(\C,H_B^{2k+2}(\KS_k, \C(k+1))) \arrow[shift left=1ex,"\pi_{f,f}"]{d} \\
         \Ext^1(\C,M_B(f\otimes f)_{\C}^*(-j)) & \Ext^0(\C,M_B(f\otimes f)_\C(k+1)) \arrow[shift left=1ex,"\subseteq",hook]{u}
    \end{tikzcd}
\end{equation}
Notice that in the right column we are projecting onto a subobject: this is possible since $H^{2k+2}_B(\KS_k, \C(k+1))$ decomposes as a direct sum of isotypical components corresponding to cusp forms (see Subsection~\ref{sec:de-rham-kuga-sato}). We are thus using the projection over a direct summand. The $\Ext^0$-group with coefficients in (some twist of) $M_B(f\otimes f)_\C$ is the correct group to land in, as we want to pair elements in the bottom row.

  \begin{lemma}
   We have $\Omega_{f,f} \in \Ext^0_{\MH_\C}(\C, M_B(f\otimes f)_\C(k+1))$.
  \end{lemma}
  
  \begin{proof}
   Since \eqref{eq:isomorph-ext-0} is functorial, it is compatible with the Hecke operators; thus $\Ext^0_{\MH_{\C}}(\C, M_B(f\otimes f)_\C(k+1))$ is identified with the direct summand of $H^{k+1, k+1}_{\dR}(\KS_k, \C)$ on which the Hecke operators in either factor act via the Hecke eigensystem of $f$. By construction, $\Omega_{f, f}$ lies in this summand.
  \end{proof}
\begin{remark}
    By construction $\Omega_{f,f} \in \wedge^2 M_{\dR}(f)$, therefore the above lemma proves that it is in the subspace $\Ext^0_{\MH_\C}(\C, \wedge^2 M_B(f\otimes f)_\C(k+1))$.
\end{remark}
To summarise, we have shown that $\Omega_{f,f}$ is in the second group in the bottom row of the diagram \ref{diag:complex-ext} and, by definition, $\cBF_{\del}^{f,f,j}$ is in the first group in the same row. There is a natural pairing between those two groups, induced by the natural one between $M_B(f\otimes f)$ and its dual, with target $\Ext^1_{\MH_\C}(\C, \C(1+k-j)) \cong \C \otimes_{\R} \frac{\C}{(2 \pi i)^{1 + k - j} \R} \cong \C$. By commutativity of the corresponding diagram, pairing $\cBF_{\del}^{f,f,j}$ with $\Omega_{f,f}$ computes the same value as in Deligne cohomology (which is also the same obtainable by pairing elements in the $\Ext$-groups of Betti cohomology).

The resulting computation expresses the relation between the bottom layer of the Beilinson--Flach Euler system and a special value of the derivative of the Rankin--Selberg $L$-function and is foundational for our paper. By the remark at page~\pageref{rmk:pairing-unchanged}, pairing (the Deligne realisations of) $\lcBF^{k,k,j}$ or $\BF^{[k,k,j]}_{\mot}$ with differential forms arising from cusp forms yields the same result. Since $\Omega_{f,f}$ comes from cusp forms, we can compute the pairing as if we were using the original classes. More precisely, we have the following proposition.
\begin{proposition}[{\cite[Proposition~9.4]{brunaultchida16}}]
    \label{prop:brunault-chida-pairing}
    We have
    \[
        \langle \rD(\BF^{[k,k,j]}_{\mot}),\Omega_{f,f} \rangle = \langle \rD(\lcBF^{k,k,j}),\Omega_{f,f} \rangle.
    \]
\end{proposition}
\begin{remark}
    The proposition makes sense because $\Omega_{f,f}$ is canonically associated to a cohomology class with compact support as noted in Subsection~\ref{subsec:motives}. Notice that in the above equality, on the left hand side we have less information on the Beilinson--Flach classes, so we have to make up for this with additional information on $\Omega_{f,f}$---namely, that it corresponds to a compactly supported class.
\end{remark}
\begin{proposition}
    We have
    \[
        \langle \rD(\cBF^{k,k,j}),\Omega_{f,f} \rangle = \frac{1+(-1)^j}{2}\langle \rD(\lcBF^{k,k,j}),\Omega_{f,f} \rangle.
    \]
\end{proposition}
\begin{proof}
    By linearity and compatibility of $\rD$ with pull-backs:
    \begin{align*}
        \langle \rD(\cBF^{k,k,j}),\Omega_{f,f} \rangle &= \frac{1}{2}\langle \rD(\lcBF^{k,k,j}),\Omega_{f,f} \rangle + \frac{(-1)^{k+j}}{2}\langle \rD((\rho')^*\lcBF^{k,k,j}),\Omega_{f,f} \rangle \\
        &= \frac{1}{2}\langle \rD(\lcBF^{k,k,j}),\Omega_{f,f} \rangle + \frac{(-1)^{k+j}}{2}\langle \rD(\lcBF^{k,k,j}),(\rho')^*\Omega_{f,f} \rangle.
    \end{align*}
    By definition $\Omega_{f,f}$ is an antisymmetric tensor with components of degree $k+1$, hence $(\rho')^*\Omega_{f,f} = (-1)^{k+1}s(\Omega_{f,f}) = (-1)^k\Omega_{f,f}$. This proves the claim.
\end{proof}
We are now in a position to prove the following version of the regulator formula over $X_1(N)^2$.
\begin{theorem}
    \label{thm:regulator-formula}
    If $j$ is even and $0\leq j\leq k$, then the following formula holds.
    \[
        \begin{split}
            \langle \rD(\lcBF^{k,k,j}), \Omega_{f,f} \rangle &= \langle \rD(\cBF^{k,k,j}), \Omega_{f,f} \rangle = \langle \cBF^{f,f,j}_{\del}, \Omega_{f,f} \rangle \\
            &= \frac{(2\pi)^{2(k-j)}}{2(\omega_f,\overline{\omega_f})}\Bigl(\frac{k!}{(k-j)!}\Bigr)^2 L'(f,f,j+1).
        \end{split}
    \]
\end{theorem}
\begin{proof}
    The first equality is given by the last proposition since $j$ is even. The second equality comes from the fact that if we extend diagram~\ref{diag:complex-diagram} with the morphisms from~\eqref{diag:complex-ext}, the resulting diagram is still commutative, because the involved morphisms commute with cup-products and push-forwards.

    Proposition~\ref{prop:brunault-chida-pairing} shows that the pairings all coincide with $\langle \BF^{[k,k,j]}_{\del},\Omega_{f,f}\rangle$. Applying~\cite[Theorem~6.2.9]{KLZ20} we deduce the explicit value of the pairing. Indeed, the pull-back of $\Omega_{f,f}$ to $Y_1(N)^2$ equals the differential form used there by construction, as we assumed $j$ even. The cohomology class they denote $\AJ_{\mathcal{H},f,f}(\mathrm{Eis}^{[k,k,j]}_{\mathcal{H},1,N})$ corresponds to the class we denoted $\BF^{[f,f,j]}_{\del}$. This follows from the fact that if a Hodge structure has non-positive weight, its Deligne and absolute Hodge cohomologies are isomorphic. Hence, the cited theorem computes the pairing which in our notation is
    \[
        \langle \BF^{[f,f,j]}_{\del}, \Omega_{f,f} \rangle.
    \]
    By the same argument that we used for compactified classes, we deduce that this equals $\langle \BF^{[k,k,j]}_{\del}, \Omega_{f,f} \rangle$. Therefore, all the pairings in the statement of the theorem compute to the expression on the right hand side of~\cite[Theorem~6.2.9]{KLZ20}, which proves the claim.
\end{proof}
\begin{remark}
    More generally, the regulator formula links the Deligne realisation of $\lcBF^{k,k',j}$ with $L'(f,g,j+1)$ through the pairing with $\Omega_{f,g}$. However, we will only need the version with $f=g$, so we restricted to this case for simplicity.
\end{remark}
In~\cite{KLZ20} Deligne cohomology is replaced with absolute Hodge cohomology. In the case under consideration the two cohomologies agree, so the above form of the theorem is equivalent to that with absolute Hodge cohomology.

\subsubsection{Remark on symmetry eigenspaces}
As remarked above, $\Omega_{f,f}$ lives in the subspace
\[ \Ext^0_{\MH_{\C}}(\C, \wedge^2 M_B(f)_\C(k+1)),\]
so it pairs non-trivially only with the $\wedge^2$-component of $\lcBF^{f,f,j}_{\del}$, i.e.\ that in the cohomology of $\wedge^2 M_{\dR}(f)^*$. When $j$ is even, $\BF^{[f,f,j]}_{\del}$ is in the antisymmetric component too, but $\lcBF^{f,f,j}_{\del}$ has no such property in general. This is the main reason why we defined the symmetrised classes $\cBF^{k,k,j}$. In any case, after Brunault and Chida's result the complex regulator formula does not detect this.

It is however important to keep in mind that we are ultimately studying classes in the cohomology of $\wedge^2 M_{\theory}(f)$ rather than of the full $M_{\theory}(f\otimes f)$. For this reason we defined $\Omega_{f,f}$ to detect the antisymmetric component. As for Beilinson--Flach classes, we ensured that the non-compact ones belong to the correct eigenspace with the requirement on $j$, and we worked around the cuspidal defect, as explained in~\ref{subsubsec:swapping}.

\subsubsection{The regulator formula in motivic cohomology}
\label{subsubsec:regulator-motivic-cohomology}
We now interpret the complex regulator formula purely in terms of motivic cohomology.

By looking at the diagrams, it is natural to ask if one could find a preimage of $\Omega_{f,f}$ in the motivic cohomology group, so that both pieces of the regulator formula would have a source in motivic cohomology. For this purpose we take advantage of the extra row in de Rham cohomology. The isomorphism~\eqref{eq:isomorph-ext-0} is such that its composition with $\rD$ is the de Rham cycle class map, as explained in~\cite[§7]{esnaultviehweg}. This is coherent with the numerology, since $H^{2k+2}_{\mot}(\KS_k,k+1) \simeq \CH^{k+1}(\KS_k)$. Therefore, it suffices to show an element of the motivic cohomology group whose image under the cycle class map is $\Omega_{f,f}$. Summarising, we are searching for a rational equivalence class $[Z_f]\in\CH^{k+1}(\KS_k)$ of cycles of codimension $k+1$ such that
\begin{align*}
    \cl_{\dR} \colon H_{\mot}^{2k+2}(\KS_k,k+1) &\to H^{k+1,k+1}_{\dR}(\KS_k,\C) \\
    [Z_f] &\mapsto C\cdot \Omega_{f,f}, \quad C \in \C^{\times}.
\end{align*}
In the remainder of this subsubsection we construct the subvariety $Z_f$ combining together three different correspondences.
\begin{list}{}
  {%
    \setlength{\leftmargin}{0pt}
    \setlength{\labelwidth}{-0.25cm}
    \setlength{\labelsep}{0.25cm}
    \setlength{\itemsep}{0.5cm}
    \renewcommand{\makelabel}[1]{\textbf{#1}:\hfil}%
  }
    \item[Projection onto $f$-eigenspace] Since $f$ is a newform, Strong Multiplicity One implies that its eigenspace in $S_{k+2}(N,\psi)$ is $1$-dimensional. The projection onto it can be realised by a Hecke operator: let $f=f_0,f_1,\ldots,f_c$ be a basis of $S_{k+2}(N,\psi)$ over $K_f$ consisting of eigenforms. For every $i=1,2,\ldots,c$ there exists a prime $\ell_i$ such that $a_{\ell_i}(f) \neq a_{\ell_i}(f_i)$. If $f_i\neq f_j$ are in the same Galois $G_{K_f}$-orbit, then we choose $\ell_i=\ell_j$. This is possible because $a_{\ell_i}(f_j) = a_{\ell_i}(f_i)^{\sigma}$ for some $\sigma\in G_{K_f}$, so either both or none of them equals $a_{\ell_i}(f)\in K_f$.

        Define the operator:
        \[
            T_f \coloneqq \prod_{i=1}^c \frac{T_{\ell_i} - a_{\ell_i}(f_i)}{a_{\ell_i}(f) - a_{\ell_i}(f_i)} \in \mathbb{T}\otimes_{\Q} K_f.
        \]
        It is easy to see that $T_f(f_i) = 0$ for every $i=1,2,\ldots,c$ and $T_f(f) = f$. Therefore $T_f$ is the $f$-isotypical projector on $S_{k+2}(N,\psi)$. We regard $T_f$ also as a correspondence on $W_k$, yielding functorially a morphism in cohomology. Note that $T_f$ is a combination of correspondences defined over $\Q$ with coefficients in $K_f$. Indeed, the chosen basis contains all the Galois conjugates of a given eigenform, under the action of $G_{K_f}$ . As a consequence, the finite linear combination defined above is fixed by $G_{K_f}$, because the coefficients are permuted among themselves thanks to our choices. Therefore, it defines an element of the Hecke algebra $\mathbb{T}\otimes_{\Q} K_f$.
    \item[Projection onto the middle Künneth component] Consider the morphisms
        \[
            \begin{tikzcd}
                W_k \arrow[bend right=25,"i_1",swap]{r} & \KS_k \arrow[bend left=25,"\pi_2"]{r} \arrow[bend right=25,"\pi_1",swap]{l} & W_k \arrow[bend left=25,"i_2"]{l}
            \end{tikzcd}
        \]
        Let $\delta_{\infty} \colon \CH^*(\KS_k) \to \CH^*(\KS_k)$ be the correspondence on $\KS_k$ given by the prescription
        \[
            \delta_{\infty}([Z]) = [Z] - i_{1,*}\circ\pi_{1,*}([Z]) - i_{2,*}\circ\pi_{2,*}([Z]).
        \]
        In~\cite[§2]{darmonrotger12} it is proved the action of $\delta_{\infty}$ on cohomology is to project onto the middle Künneth component in any degree. In particular
        \[
            \delta_{\infty}^* \colon H^{2k+2}_{\dR}(\KS_k,\C) \to \bigl(H^{k+1}_{\dR}(W_k,\C)\bigr)^{\otimes 2}.
        \]
    \item[Atkin-Lehner involution] Denote by $w_N$ the Atkin-Lehner involution induced by the action of $\bigl(\begin{smallmatrix} 0 & -1 \\ N & 0 \end{smallmatrix}\bigr)$, and with $\lambda_N(f)$ the pseudo-eigenvalue of $f$. Notice that $w_N$ is not defined over $\Q$, but only over $\Q(\mu_N)$, so it is actually only a correspondence on $W_{k,\Q(\mu_N)}$.
\end{list}
Denote also with $\mathrm{gr}$ the graph of a correspondence. For a correspondence $c$ from $C_1$ to $C_2$, its graph is a subset $\mathrm{gr}(c) \subseteq C_1 \times C_2$.

\begin{definition}
    Let $[Z_f]\in\CH^{k+1}(\KS_{k,\Q(\mu_N)})\otimes K_f$ be the class of the cycle $\delta_{\infty}(\mathrm{gr}(T_{f^*}\circ w_N))$.
\end{definition}
The composition $T_{f^*}\circ w_N$ is a correspondence on $W_k$, so its graph is a subvariety of $\KS_k$, and $\delta_{\infty}$ has a well-defined action on its equivalence class.
\begin{proposition}
    \label{prop:image-cycle-dr}
 The above class satisfies $\cl_{\dR}(Z_f) = \lambda_N(f^*)\Omega_{f,f}$.
\end{proposition}
In the statement we are comparing classes in de Rham cohomology with coefficients in $\C$, so the base change to $\Q(\mu_N)$ and the extension of coefficients to $K_f$ do not play any role. Thanks to the flatness of $\C$, the equality would still hold over smaller fields, even though we will not use this fact.
\begin{proof}
    Let $T = T_{f^*} \circ w_N$ for notational simplicity. Denote by $\Delta \subseteq \KS_k$ the diagonal of $\KS_k$. By definition
    \[
        \cl_{\dR}(Z_f) = \rho \iff \int_{Z_f} \omega = \int_{\KS_k} \omega\rho = (\omega,\rho) \quad \forall \omega.
    \]
    In particular, if we vary $\omega$ in a basis, we can compute $\cl_{\dR}(Z_f)$ as an element of $H^{k+1,k+1}_{\dR}(\KS_k,\C)^{\vee}$ and then express it as a differential form by duality.
    Consider a basis of $H^{k+1,k+1}_{\dR}(\KS_k,\C)$ determined by the direct sum decomposition given by the Künneth isomorphism. Let $\omega \in H^{k+1,k+1}_{\dR}(\KS_k,\C)$ be a differential form in this basis, then by a standard property of the pull-back
    \[
        \int_{Z_f} \omega = \int_{\mathrm{gr}(T)} \delta_{\infty}^* \omega = \begin{cases}
            \int_{\mathrm{gr}(T)}\omega &\text{if $\omega \in (H^{k+1}_{\dR}(W_k,\C))^{\otimes 2}$} \\
            0 &\text{otherwise}
        \end{cases}
    \]
    The last equality holds because $\delta_{\infty}^*$ projects on the middle Künneth component. It suffices then to consider forms in a basis of $(H^{k+1}_{\dR}(W_k,\C))^{\otimes 2}$. Consider a basis of $H^{k+1}_{\dR}(W_k,\C)$ consisting of differential forms associated to cusp forms, which exists by the isomorphism~\eqref{eq:cohomology-decompositions}.
    Write $\omega = \omega_1 \otimes \omega_2$ with $\omega_1$ and $\omega_2$ in this basis. Again by properties of the pull-back we have
    \[
        \int_{Z_f} \omega = \int_{\mathrm{gr}(T)} \delta_{\infty}^* \omega = \int_{\mathrm{gr}(T)} \omega = \int_{\Delta} \omega_1 \otimes w_N^* T_{f^*}^* \omega_2.
    \]
    By construction $T_{f^*}$ projects onto $M_{\dR}(f^*)\otimes \C$, which is generated by $\omega_{f^*}$ and $\eta_{f^*}$. In the first case we are left with
    \[
     = \int_{\Delta} \omega_1 \otimes w_N^* \omega_{f^*} = \int_{\Delta} \omega_1 \otimes \lambda_N(f^*)\omega_f = \lambda_N(f^*)(\omega_1,\omega_f).
    \]
 Since we are working with basis vectors, the cup product $(\omega_1,\omega_f)$ is non-zero only when $\omega_1$ is a scalar multiple of the dual of $\omega_f$, and equals $1$ exactly when $\omega_1 = -(\omega_f)^* = -\eta_{f^*}$. The argument then shows that
   \[
     \cl_{\dR}(Z_f)(-\lambda_N(f^*)^{-1}\eta_{f^*}\otimes\omega_{f^*}) = 1. \]
  Similarly, we obtain
  \[\cl_{\dR}(Z_f)(\lambda_N(f^*)^{-1}\omega_{f^*}\otimes\eta_{f^*}) = 1.\]
 Therefore $\cl_{\dR}(Z_f)$ coincides with $(-\lambda_N(f^*)^{-1}\eta_{f^*} \otimes \omega_{f^*})^* + (\lambda_N(f^*)^{-1}\omega_{f^*} \otimes \eta_{f^*})^*$. To compute these observe:
    \[
        (\omega_{f^*}\otimes \eta_{f^*},\eta_f\otimes \omega_f) = (\omega_{f^*},\eta_f)\cdot (\eta_{f^*},\omega_f) = -1.
    \]
 So we get as a result $\lambda_N(f^*)(\omega_f \otimes \eta_f -\eta_f\otimes \omega_f )$. This proves the proposition.
\end{proof}
In order to use the correspondences $T_{f^*}$ and $w_N$, we base changed to $\Q(\mu_N)$ and extended coefficients to $K_f$. We then need to modify our original diagram to reflect these changes. Notably, we now start from the motivic cohomology of $\KS_{k,\Q(\mu_N)}$. However, the two bottom rows of diagram~\eqref{diag:complex-ext} are unchanged. Indeed, $\KS_{k,\Q(\mu_N)}$ is still defined over $\Q$ (hence over $K_f$) and in Betti and de Rham cohomology we deduce by the Künneth isomorphism:
\begin{align*}
    H_B^i(\KS_{k,\Q(\mu_N)},K_f(n)) &\simeq H_B^i(\KS_k,K_f(n)) \otimes H_B^0(\Q(\mu_N),K_f) = H_B^i(\KS_k,K_f(n)), \\
    H_{\dR}^i(\KS_{k,\Q(\mu_N)},K_f(n)) &\simeq H_{\dR}^i(\KS_k,K_f(n)).
\end{align*}
Thanks to this remark we can use the commutativity of the diagram in figure~\ref{diag:complex-diagram} to finish the proof of the the next theorem.

  \begin{definition}\label{def:BFKelt}
   The $j$-th \emph{Beilinson--Flach $K$-element} associated to $f$ is the element
   \[
    b_{f, j} \coloneqq \langle\cBF^{k,k,j}, Z_f\rangle \in H^1_{\mot}(\Q(\mu_N), 1+k-j)\otimes K_f.
   \]
  \end{definition}
  
\begin{theorem}[Complex regulator formula]
    \label{thm:complex-l-value}
    If $j$ is even and $0\leq j\leq k$, then $b_{f, j}$ satisfies:
    \begin{equation}
        \label{eq:deligne-regulator-bf}
        \begin{split}
            \rD(b_{f, j}) &= \langle \rD(\cBF^{k,k,j}), \lambda_N(f^*)\Omega_{f,f} \rangle \\
            &= (2\pi i)^{2(k-j)}\lambda_N(f^*)\frac{(-1)^{k-j}}{2(\omega_f,\overline{\omega_f})}\Bigl( \frac{k!}{(k-j)!} \Bigr)^2 L'(f,f,j+1).
        \end{split}
    \end{equation}
\end{theorem}
\begin{proof}
    This is a direct application of the complex regulator formula proved above. The first equality is implied by the commutativity of the diagram in figure~\ref{diag:complex-diagram}. The second equality is Theorem~\ref{thm:regulator-formula}, which applies because $j$ is even.
\end{proof}

\begin{remark}
    $b_{f, j}$ is a motivic source for $L$-values, as it lies in the motivic cohomology group $H^1_{\mot}(\Q(\mu_N),1+k-j)\otimes K_f$ and its image under $\rD$ yields \emph{complex} $L$-values. In the sequel we will also relate $b_{f, j}$ to \emph{$p$-adic} $L$-values, by computing its image under $\ret$.
\end{remark}

 \subsection{Non-archimedean argument}
  \label{subsec:p-adic-argument}

  In this subsection we apply the machinery of Subsection~\ref{sec:p-adic-diagram} to the case under consideration. Recall that at the end of that subsection we constructed a commutative diagram (figure~\ref{diag:etale-diagram}) by studying edge maps in the Hochschild-Serre spectral sequence. If we specialize this to $X = \KS_{k,\Q(\mu_N)}$ (then $d = k+1$), the bottom line of the diagram reads:
  \begin{multline*}
   H^1\left(\Q(\mu_N), H^{2k+2}_{\et}(\KS_{k,\overline{\Q}},\Q_p(2+2k-j))\right) \times H^0\left(\Q(\mu_N),H^{2k+2}_{\et}(\KS_{k,\overline{\Q}},\Q_p(1+k))\right) \\
   \to H^1\left(\Q(\mu_N),\Q_p(1+k-j)\right).
  \end{multline*}
  Since $M_{\et}(f\otimes f)$ is a subspace of $H^{2k+2}_{\et}(\KS_{k,\overline{\Q}},\Q_p(1+k)) \otimes_{\Q} K_f$, and its dual $M_{\et}(f\otimes f)^*$ is a quotient of $H^{2k+2}_{\et}(\KS_{k,\overline{\Q}},\Q_p(2+2k)) \otimes_{\Q} K_f$, the base-extension of the above pairing to $K_f$ induces a pairing
  \begin{multline*}
   H^1\Big(\Q(\mu_N),M_{\et}(f\otimes f)^*(-j)\Big) \times H^0\Big(\Q(\mu_N),M_{\et}(f\otimes f)(1+k)\Big) \\
   \to H^1(\Q(\mu_N),\Q_p(1+k-j)) \otimes K_f,
  \end{multline*}
  and this is compatible with the pairings in the first row of the diagram (tensored with $K_f$). By construction, the image in $H^0\left(\Q(\mu_N),H^{2k+2}_{\et}(\KS_{k,\overline{\Q}},\Q_p(1+k))\right) \otimes K_f$ of the element $Z_f$ lies in $H^0\Big(\Q(\mu_N),M_{\et}(f\otimes f)(1+k)\Big)$, so pairing with this element factors through the projection to $H^1\left(\Q(\mu_N),M_{\et}(f\otimes f)^*(-j)\right)$.

  Notice also that since $f$ is a cusp form, its eigenspace is concentrated in the middle cohomological degree. This is equivalent to the fact that the representation $\rho_{f,v}$ is pure of weight $k+1$.
  
  Furthermore, by replacing the étale regulator with the syntomic regulator we can apply the same argument to construct a commutative diagram in syntomic cohomology. Recall that under suitable hypotheses---in particular over $p$-adic fields---the first étale and syntomic cohomology groups are related by the Bloch-Kato exponential.
  
  We introduce the following notation:
  \[
      Z_f^{\diamond} \coloneqq \regulator{\diamond}(Z_f), \quad \diamond\in\{\et,\syn\}.
  \]
  Suppose we are able to compute $\langle \rsyn(\cBF^{k,k,j}), Z_f^{\syn} \rangle$. The resulting element belongs to the group $H^1_{\syn}(\Q_p(\mu_N), \Q_p(1+k-j))\otimes K_f$. There is then a map:
  \[
   H^1_{\syn}(\Q_p(\mu_N),\Q_p(1+k-j)) \to H^1(\Q_p(\mu_N),\Q_p(1+k-j))
  \]
  given by the Bloch-Kato exponential, as $\Q_p(1+k-j)$ is crystalline hence de Rham. Notice that we moved to étale cohomology over local fields.
  
  We would like to use this map to link the syntomic and étale diagrams, but there is an obstacle given by the fact that over $p$-adic fields the edge map from $H^{3+2k}_{\et}$ to $E_2^{0,3+2k}$ is not necessarily zero, as explained in Subsection~\ref{sec:p-adic-diagram}. This problem would defy the construction of our diagram, because we could not map Beilinson--Flach classes to $E_2^{1,2k+2}$. However, the kernel of the edge map to $E_2^{0,3+2k}$ contains all the cohomology classes that are localisations of global classes. Our construction is chiefly global, which is to say that we are dealing with classes in the cohomology of $\KS_k$ over a number field; and the classes that we consider in the local diagram are their localisations. Therefore, they lie in the kernel of the edge map to $E_2^{0,3+2k}$, and thus they can safely be mapped to $E_2^{1,2k+2}$.
  
  By this discussion, the following diagram commutes as we are feeding it with global classes (we omit the coefficients for better readability):
  \[
      \begin{tikzcd}[column sep=0.7em,cramped]
          H_{\syn}^{2k+3}(\KS_{k,\Q_p(\mu_N)}) \times H_{\syn}^{2k+2}(\KS_{k,\Q_p(\mu_N)}) \arrow{r} & H^1_{\syn}(\Q_p(\mu_N)) \arrow["\exp"]{dd} \\
       H_{\mot}^{2k+3}(\KS_k) \times H_{\mot}^{2k+2}(\KS_k) \arrow["\rsyn"]{u} \arrow["\mathrm{loc}\,\circ\,\ret"']{d} & \\
       H_{\et}^{2k+3}(\KS_{k,\Q_p(\mu_N)}) \times H_{\et}^{2k+2}(\KS_{k,\Q_p(\mu_N)}) \arrow{r} \arrow["\pr_{f,f}\circ\mathrm{edge}"']{d} &  H^1(\Q_p(\mu_N)) \arrow[equal]{d} \\
       H^1(\Q_p(\mu_N),M_{\et}(f\! \otimes\! f)^*_{\Q_p}) \times H^0(\Q_p(\mu_N),M_{\et}(f\! \otimes\! f)_{\Q_p}) \arrow{r} &  H^1(\Q_p(\mu_N))
      \end{tikzcd}
  \]
  Therefore, simply by applying $\exp$ we would obtain the value of the pairing in \emph{local} étale cohomology. It suffices then to compute the pairing in syntomic cohomology.
  
  The rest of the argument is organised as follows:
  \begin{itemize}
      \item we will show that $Z_f^{\syn}$ is canonically associated to a compactly supported cohomology class;
      \item we will split the pairing in syntomic cohomology in the sum of three terms;
      \item the first term computes the pairing on the open variety $\Ell^k \times \Ell^k$, we will link it to the $p$-adic value;
      \item the other terms compute the ``cuspidal contribution'', we will show it is always zero.
  \end{itemize}

\subsubsection*{Study of $Z_f^{\syn}$}

  We recall that syntomic cohomology is characterised by sitting in a long exact sequence (induced by the Leray spectral sequence):
  \begin{multline*}
      \cdots \to H_{\syn}^{2k+2}(\KS_{k,\Q_p(\mu_N)},\Q_p(k+1)) \\
      \to F^0 H_{\dR}^{2k+2}(\KS_{k,\Q_p(\mu_N)},\Q_p(k+1)) = F^{k+1} H_{\dR}^{2k+2}(\KS_{k,\Q_p(\mu_N)},\Q_p) \\
      \to H_{\rig}^{2k+2}(\KS_{k,\Q_p(\mu_N)},\Q_p(k+1)) \to \cdots
  \end{multline*}
  We can thus write $Z_f^{\syn} = (\xi,\lambda)$ with $\xi$ and $\lambda$ respectively rigid and de Rham cohomology classes. By the compatibility of the regulator $\rsyn$ and $\cl_{\dR}$ in the Leray spectral sequence~\cite[Theorem~7.5]{besser00}, the composition of $\rsyn$ and the morphism to de Rham cohomology in the long exact sequence coincides with $\cl_{\dR}$. Therefore $\lambda$ is equal to the image of $\cl_{\dR}(Z_f) = \lambda_N(f^*)\Omega_{f,f}$ under the extension of the coefficients to $\Q_p$. Our first aim is to compute it.

  We start by studying the behaviour of $\omega_f$ and $\eta_f$ in syntomic cohomology. In order to pass to de Rham cohomology with $(K_f\otimes\Q_p)$-coefficients we must start from classes defined over $K_f$ rather than over $\C$. According to~\cite[§6.1]{KLZ17} the differentials $\omega_f, \eta_f$ can be easily modified to be $K_f$-rational. Explicitly, if we define
  \[
      \omega_f' = \tau(\psi^{-1})\omega_f, \quad \eta_f' = \tau(\psi^{-1})\eta_f
  \]
  then $\omega_f', \eta_f' \in M_{\dR}(f)$, i.e.\ they are $K_f$-rational. We can then safely regard $\omega_f', \eta_f' \in M_{\dR}(f)_{\Q_p}$. 
  
  Since $p \nmid N$, the space $M_{\dR}(f)_{\Q_p}$ has an action of the Frobenius $\varphi$, by comparison with crystalline cohomology; and the eigenvalues of $\varphi$ are $\alpha_f$ and $\beta_f$.

  \begin{lemma}
   \label{lemma:frobenius-eigenspaces}
   If $(f, \alpha_f)$ is non-critical (in the sense of \Cref{def:noble}), then the $\alpha_f$-eigenspace in $M_{\dR}(f)_{\Q_p}$ is one-dimensional and does not coincide with $\Fil^1 M_{\dR}(f)_{\Q_p}$.
  \end{lemma}

  \begin{proof}
   Since $M_{\dR}(f)_{\Q_p}$ is a weakly admissible filtered $\varphi$-module, any $\varphi$-stable subspace must have its Newton polygon on or above its Hodge polygon. In particular, if $v_p(\alpha_f) < k+1$, the $\alpha_f$-eigenspace cannot intersect with $\Fil^1$ (so if $\alpha_f = \beta_f$, then $\varphi$ must be non-diagonalisable).
   
   In the case $v_p(\alpha_f) = k+1$, both the $\alpha_f$-eigenspace and the $\beta_f$-eigenspace are weakly admissible, and hence define one-dimensional $G_{\Q_p}$-subrepresentations of $M_{\et}(f)$. Thus $M_{\et}(f)|_{G_{\Q_p}}$ is the direct sum of two characters, contradicting the assumption that $(f, \alpha_f)$ is non-critical.
  \end{proof}

  We assume henceforth that $(f, \alpha_f)$ is noble. Since $\eta_f'$ is a generator of $M_{\dR}(f)_{\Q_p}/\Fil^1 M_{\dR}(f)_{\Q_p}$, by the above lemma it has a \emph{unique} lift to the eigenspace where $\phi$ acts as $\alpha_f$. Clearly there are several choices of a lift to $M_{\dR}(f)_{\Q_p}$, with an ambiguity given by an element of $\Fil^1 M_{\dR}(f)_{\Q_p}$. Thanks to the next lemma this ambiguity does not play any role in our future computations.

  \begin{lemma}
   Let $\eta_1, \eta_2 \in M_{\dR}(f)_{\Q_p}$ be two lifts of $\eta_f'$. Then $\omega_f' \otimes \eta_1 - \eta_1 \otimes \omega_f' = \omega_f' \otimes \eta_2 - \eta_2 \otimes \omega_f'$.
  \end{lemma}

  \begin{proof}
   The difference of the two is:
   \[
    (\omega_f' \otimes \eta_1 - \eta_1 \otimes \omega_f') - (\omega_f'\otimes \eta_2 - \eta_2 \otimes \omega_f') = \omega_f' \otimes (\eta_1 - \eta_2) - (\eta_1 - \eta_2)\otimes \omega_f'.
   \]
   Now $\eta_1 - \eta_2 \in \Fil^1 M_{\dR}(f)_{\Q_p}$ which is a $1$-dimensional $K_f\otimes\Q_p$-vector space. By writing $\eta_1 - \eta_2 = u\omega_f'$ for some $u\in K_f\otimes\Q_p$ the claim follows.
  \end{proof}

  Since we will only consider linear combinations of this kind, we can choose as a ``favourite'' lift the unique lift $\eta^{\alpha_f} \in V_{\phi}(\alpha_f)$ satisfying $\phi(\eta^{\alpha_f}) = \alpha_f\eta^{\alpha_f}$.

  The above argument proves that the image of $\tau(\psi^{-1})^2\Omega_{f,f}$ under extension of scalars to $\Q_p$ can be identified with $\omega_f'\otimes\eta^{\alpha_f} -\eta^{\alpha_f}\otimes\omega_f'$. We sum up this discussion in the following proposition.

  \begin{proposition}
    \label{prop:pair-representing-syn}
   $Z_f^{\syn}$ is represented by the pair $(\xi, \lambda_N(f^*)\tau(\psi^{-1})^{-2}(\omega_f'\otimes\eta^{\alpha_f} - \eta^{\alpha_f}\otimes\omega_f'))$.
  \end{proposition}

  We now study the situation over the non-proper variety $\Ell^{k,k}$. The analogue of the long exact sequence above holds for cohomology with compact support:
  \begin{multline*}
   \cdots \to H_{\syn,c}^{2k+2}(\Ell^{k,k}_{\Q_p(\mu_N)},\Q_p(k+1)) \to F^0 H_{\dR,c}^{2k+2}(\Ell^{k,k}_{\Q_p(\mu_N)}, \Q_p(k+1)) \\
   \to H_{\rig,c}^{2k+2}(\Ell^{k,k}_{\Q_p(\mu_N)},\Q_p(k+1)) \to \cdots
  \end{multline*}

  Recall that $\lambda_N(f^*)\tau(\psi^{-1})^2\Omega_{f,f}$ belongs to $M_{\dR}(f\otimes f)$, which has canonical inclusions as a direct summand of the cohomology of $\Ell^{k,k}$ with and without compact support. Therefore $\lambda_N(f^*)\tau(\psi^{-1})^2\Omega_{f,f}$ has a canonical image in compactly supported de Rham cohomology, under extension of scalars to $\Q_p$. By the above exact sequence, the preimage of a compactly supported de Rham cohomology class is a compactly supported syntomic cohomology class. By putting everything together we deduce the following result.

  \begin{proposition}
   \label{prop:syntomic-compact-support}
   There exists a canonical compactly supported cohomology class $Z^{\syn}_{f,c} \in H^{2k+2}_{\syn,c}(\Ell^{k,k}_{\Q_p(\mu_N)}, \Q_p(k+1))$ that maps to $Z_f^{\syn}$ in $H^{2k+2}_{\syn}(\Ell^{k,k}_{\Q_p(\mu_N)}, \Q_p(k+1))$.
  \end{proposition}

  \subsubsection*{Splitting the pairing}

   The pairing that we are interested in computing is:
   \[
       \langle \rsyn(\cBF^{k,k,j}), Z_f^{\syn} \rangle_{\syn, \KS_k}.
   \]
   Since our aim is to split this computation into an ``open'' and a ``cuspidal'' part, the standard strategy would be to use the Gysin long exact sequence induced by the embedding $\Ell^{k,k} \hookrightarrow \KS_k$. In our case it suffices to use the existence of Gysin maps in the category, and the following projection formula due to Besser. The theorem is stated in terms of finite-polynomial cohomology, of which syntomic cohomology is a particular case.

   \begin{theorem}[{\cite[Corollary~5.3]{besser12}}]
    Let $\iota\colon X\to Y$ be a morphism of smooth integral $\ringint{O}_K$-schemes, and identify the groups $H^{2\dim (X)+1}_{\fp,c}(X) \simeq K$ and $H^{2\dim (Y)+1}_{\fp,c}(Y) \simeq K$ via the respective trace maps. Let $d = \dim(Y) - \dim(X)$. Let $x\in H^j_{\fp,c}(X,m)$ and $y\in H^i_{\fp}(Y,n)$, and suppose $i+j = 2\dim (X)+1$ and $n+m > \dim (X)$. Then there exists a push-forward map in any degree
    \begin{gather*}
     \iota_* \colon H^{\bullet}_{\fp}(X,d_0) \to H^{\bullet+2d}_{\fp}(Y,d_0+d)
     \shortintertext{for all $d_0\in\N$, and}
     y \cup \iota_* x = \iota^* y \cup x.
    \end{gather*}
   \end{theorem}

   In our case, $\KS_k$ and $\Ell^{k,k}$ are both schemes with a model over $\Z[1/N]$, so they can be regarded as smooth $\Z_p$-schemes; and their base changes to $\Q_p(\mu_N)$ can be regarded as smooth $\ringint{O}_{\Q_p(\mu_N)}$-schemes. The embedding $\iota\colon \Ell^{k,k} \hookrightarrow \KS_k$ is a smooth morphism, so we are in a position to apply the theorem with $d=0$. Moreover, $Z_f^{\syn} \in H^{2k+2}_{\syn}(\Ell^{k,k}_{\Q_p(\mu_N)},\Q_p(k+1))$ is associated to a canonical compactly supported class $Z^{\syn}_{f,c}$ by Proposition~\ref{prop:syntomic-compact-support}. If we choose
   \[
       x = Z^{\syn}_{f,c}, \quad y = \rsyn(\cBF^{k,k,j})
   \]
   the hypotheses of the theorem are verified, since the sums of degrees and twists are
   \begin{gather*}
       2+2k+3+2k = 5+4k = 2\dim \Ell^{k,k} +1, \\
       k+1+2+2k-j = 3+3k-j > \dim \Ell^{k,k}.
   \end{gather*}
   We can thus apply the theorem to find
   \[
       \mathrm{tr}_{\KS_k} (\rsyn(\cBF^{k,k,j}) \cup \iota_* Z_{f,c}^{\syn}) = \mathrm{tr}_{\Ell^{k,k}} (\iota^* \rsyn(\cBF^{k,k,j}) \cup Z_{f,c}^{\syn}).
   \]
   We can now make the following consideration:
   \begin{enumerate}
       \item $\iota_* Z_{f,c}^{\syn} = Z_f^{\syn}$ since they are represented by the same pair of cohomology classes (in rigid and de Rham cohomology, $M_{\star}(f\otimes f)$ has canonical isomorphisms between cohomologies of $\KS_k$ and $\Ell^{k,k}$ with and without compact support);
       \item $\iota^*\rsyn(\cBF^{k,k,j}) = \rsyn(\iota^*\cBF^{k,k,j})$ by compatibility of regulators and pull-backs, and by construction of $\cBF^{k,k,j}$
           \[
               \iota^*\cBF^{k,k,j} = \frac{1}{2}\iota^*\lcBF^{k,k,j} + \frac{(-1)^{k+j}}{2}\iota^*(\rho')^*\lcBF^{k,k,j}.
           \]
           Now $\iota^*(\rho')^* = (\rho'\circ \iota)^* = (\iota\circ \rho')^*$ because the two morphisms commute. Indeed, $\rho'$ swaps the two components of the fibre product, while $\iota$ embeds $\Ell^{k,k}$ into $\KS_k$. The two compositions become then, for every $P,Q\in\Ell^k$:
           \[
            \iota\circ\rho'(P,Q) = \iota(Q,P) = (Q,P) = \rho'(P,Q) = \rho'\circ \iota(P,Q).
           \]
           We deduce
           \[
               \iota^*\cBF^{k,k,j} = \frac{1}{2}\iota^*\lcBF^{k,k,j} + \frac{(-1)^{k+j}}{2}(\rho')^*\iota^*\lcBF^{k,k,j}.
           \]
           The inner pull-back unravels as:
           \begin{equation}
               \label{eq:pull-back-lcbf}
               \begin{split}
                   \iota^* \lcBF^{k,k,j} &= \iota^*(\lBF^{[k,k,j]}_{\mot} + i_{\mathrm{cusp},*}(\xi)) \\
                   &= \iota^*(\lBF^{[k,k,j]}_{\mot}) + \iota^*(i_{\mathrm{cusp},*}(\xi)) = \BF^{[k,k,j]}_{\mot} + \iota^*(i_{\mathrm{cusp},*}(\xi)).
               \end{split}
           \end{equation}
           Finally, putting everything together:
           \begin{equation}
               \label{eq:pull-back-cbf}
               \begin{split}
                   \iota^* \cBF^{k,k,j} &= \frac{1}{2}\bigl(\BF^{[k,k,j]}_{\mot} + \iota^*(i_{\mathrm{cusp},*}(\xi))\bigr) + \frac{(-1)^{k+j}}{2}\bigl((\rho')^*\BF^{k,k,j}_{\mot} + (\rho')^*\iota^*(i_{\mathrm{cusp},*}(\xi))\bigr) \\
                   &= \BF^{[k,k,j]}_{\mot} + \frac{1}{2}\iota^*(i_{\mathrm{cusp},*}(\xi)) + \frac{(-1)^{k+j}}{2}(\rho')^*\iota^*(i_{\mathrm{cusp},*}(\xi)).
               \end{split}
           \end{equation}
   \end{enumerate}
By linearity of the trace map, we conclude that
\begin{equation}
    \label{eq:pairing-in-two}
    \begin{split}
        \langle \rsyn(\cBF^{k,k,j}),Z_f^{\syn} \rangle_{\syn, \KS_k} &= \langle \BF^{[k,k,j]}_{\syn}, Z_{f,c}^{\syn} \rangle_{\syn,\Ell^{k, k}} \\
        &+ \frac{1}{2}\langle \iota^*(i_{\mathrm{cusp},*}(\xi))_{\syn}, Z_{f,c}^{\syn} \rangle_{\syn,\Ell^{k, k}} \\
        &+ \frac{(-1)^{k+j}}{2}\langle \iota^*(i_{\mathrm{cusp},*}(\xi))_{\syn}, (\rho')^* Z_{f,c}^{\syn} \rangle_{\syn,\Ell^{k, k}}.
    \end{split}
\end{equation}

   We have therefore written the sought pairing as the sum of three terms. The first is nothing other than the pairing of the non-compactified Beilinson--Flach classes with $Z_{f,c}^{\syn}$ on the open variety, while the other two represent the contributions at the cuspidal locus. In the remainder of this subsection we will evaluate them separately, showing that the first contribution gives a $p$-adic $L$-value, while the others vanish.

   By comparing equations~\eqref{eq:pull-back-lcbf} and~\eqref{eq:pull-back-cbf} we also notice that the difference between $\iota^*\lcBF^{k,k,j}$ and $\iota^*\cBF^{k,k,j}$ is supported only on the cuspidal locus. Therefore, the pairings of their syntomic realisations with $Z_f^{\syn}$ differ only by some contributions of the form $\langle \iota^*(i_{\mathrm{cusp},*}(\xi ))_{\syn}, - \rangle_{\syn}$.

\subsubsection*{Computing the first term}
We now compute the first pairing in equation~\eqref{eq:pairing-in-two}. Reasoning as we did for étale cohomology, we can use the edge map and the projection onto the $(f,f)$-isotypical component to link compatibly the row giving the pairing in syntomic cohomology
\begin{multline*}
    H^{3+2k}_{\syn}(\Ell^{k,k}_{\Q_p(\mu_N)},\Q_p(2+2k-j)) \times H^{2+2k}_{\syn,c}(\Ell^{k,k}_{\Q_p(\mu_N)},\Q_p(k+1)) \\
    \xrightarrow{\langle , \rangle_{\syn}} H^1_{\syn}(\Q_p(\mu_N),\Q_p(1+k-j))
\end{multline*}
with the following row
\begin{multline*}
        H^1(\ringint{O}_{\Q_p(\mu_N)},M_{\rig}(f\otimes f)^*(-j)) \times H^0(\ringint{O}_{\Q_p(\mu_N)},M_{\rig,c}(f\otimes f)(k+1)) \\
    \xrightarrow{\langle , \rangle_{f,f}} H^1_{\syn}(\Q_p(\mu_N),\Q_p(1+k-j)).
\end{multline*}
The two rows together form a commutative diagram. We can therefore compute the pairing along the latter one. To go even further, one can compose every map from the first row to the second, with an isomorphism, to obtain as the composition the syntomic Abel--Jacobi map $\AJ_{\syn}$. Since these isomorphisms are compatible with the cup-product and the trace map, all three rows compute the same pairing, in particular
\[
    \langle \BF^{[k,k,j]}_{\syn},Z_{f,c}^{\syn}\rangle_{\syn} = \langle \AJ_{\syn}(\BF^{[k,k,j]}_{\syn}),\AJ_{\syn}(Z_{f,c}^{\syn})\rangle_{f,f}.
\]
By~\cite[Proposition~6.3.1]{KLZ20} the pairing $\langle \AJ_{\syn}(-), \tilde{\lambda} \rangle_{f,f}$ depends only on the class of $\tilde{\lambda}$ in de Rham cohomology. This result comes from the fact that the Abel-Jacobi map has as target a quotient of $M_{\dR}(f\otimes f)_{\Q_p}$. We can then compute the above pairing regardless of the associated rigid cohomology class. By Propositions~\ref{prop:pair-representing-syn} and~\ref{prop:syntomic-compact-support} and since $M_{\dR}(f\otimes f)$ has canonical isomorphisms between cohomologies of $\KS_k$ and $\Ell^{k,k}$ with compact support, the image of $Z_{f,c}^{\syn}$ in de Rham cohomology is the class of $\lambda_N(f^*)\tau(\psi^{-1})^{-2}(\omega_f'\otimes\eta^{\alpha_f} -\eta^{\alpha_f}\otimes\omega_f')$. Therefore the pairing computes to
\[
 \langle \BF^{[k,k,j]}_{\syn},Z_{f,c}^{\syn}\rangle_{\syn} = \lambda_N(f^*)\tau(\psi^{-1})^{-2}\langle \AJ_{\syn}(\BF^{[k,k,j]}_{\syn}),\omega_f'\otimes\eta^{\alpha_f} -\eta^{\alpha_f}\otimes\omega_f')\rangle_{f,f}.
\]

The following theorem links explicitly the pairing in syntomic cohomology with the Rankin--Selberg $p$-adic $L$-function.
\begin{theorem}[{\cite[Theorem~6.5.9 and Remark~6.5.10]{KLZ20}}]
    \label{thm:klz-regulator-formula}
    If $E(f,f,j+1)\neq 0$, then
    \begin{multline*}
        \langle \AJ_{\syn}(\BF^{[k,k,j]}_{\syn}),\eta^{\alpha_f}\otimes \omega_f'\rangle \\
        = (-1)^{k-j+1}k!\binom{k}{j}\tau(\psi^{-1})^2\frac{E(f)E^*(f)}{E(f,f,j+1)}\Lgeom(F,F)(k,k,j+1).
    \end{multline*}
\end{theorem}
\begin{remark}
    Theorem~6.5.9 in \cite{KLZ20} states the above result for the case of $f$ ordinary. Remark~6.5.10 therein covers the case of $f$ supersingular, by using the geometric $p$-adic $L$-function.
\end{remark}

\begin{lemma}
    We have
    \[
        \langle \AJ_{\syn}(\BF^{[k,k,j]}_{\syn,N}),\omega_f'\otimes\eta^{\alpha_f}\rangle = (-1)^{j+1} \langle \AJ_{\syn}(\BF^{[k,k,j]}_{\syn,N}),\eta^{\alpha_f}\otimes \omega_f'\rangle.
    \]
\end{lemma}
  \begin{proof}
   The non-compactified Beilinson--Flach classes live in the cohomology of $\Ell^{k, k}$ and of $Y_1(N)^2$. We regard them and the involved differentials as classes in the cohomology of $Y_1(N)^2$. The differential forms are then in the second cohomology group, for which the Künneth isomorphism reads
   \[
    H^2_{\dR}(Y_1(N)^2,\TSym^{[k,k]} \HS_{\dR}) \simeq \bigl(H^1_{\dR}(Y_1(N),\TSym^k \HS_{\dR})\bigr)^{\otimes 2}.
   \]
   Let $s$ be the map that swaps the factors in the tensor product, and $\rho$ the morphism exchanging the two factors of $Y_1(N)^2$. Since the isomorphism is induced by the cup product, $s$ and $\rho^*$ differ by a sign that depends on the degree. We deduce $\omega_f'\otimes\eta^{\alpha_f} = s(\eta^{\alpha_f}\otimes\omega_f') = -\rho^*(\eta^{\alpha_f}\otimes\omega_f')$. Now~\cite[Proposition~5.2.3]{KLZ17} implies
   \begin{align*}
    \langle \AJ_{\syn}(\BF^{[k,k,j]}_{\syn,N}), \omega_f'\otimes\eta^{\alpha_f} \rangle &= \langle \AJ_{\syn}(\BF^{[k,k,j]}_{\syn,N}), -\rho^*(\eta^{\alpha_f}\otimes\omega_f') \rangle \\
    &= -\langle \rho^*(\AJ_{\syn}(\BF^{[k,k,j]}_{\syn,N})), \eta^{\alpha_f}\otimes\omega_f' \rangle \\
    &= (-1)^{j+1} \langle \AJ_{\syn}(\BF^{[k,k,j]}_{\syn,N}), \eta^{\alpha_f}\otimes \omega_f' \rangle.\qedhere
   \end{align*}
  \end{proof}

  \begin{remark}
   The argument of the proof would apply to the second version of the compactified classes, though not to the first one, because of the extra correction term $i_{\mathrm{cusp},*}(\xi)$. Indeed, this term is defined asymmetrically with respect to the two factors of the fibre product. There is then little hope that $\lcBF^{k,k,j}$ is in any of the eigenspaces for the swapping involution in general (see Subsubsection~\ref{subsubsec:swapping}), while the classes $\BF^{[k,k,j]}$ and $\cBF^{k,k,j}$ show a better behaviour.
  \end{remark}

The above combine into
\begin{theorem}
    \label{thm:first-term}
    If $j$ is even, $0\leq j\leq k$ and $E(f,f,j+1)\neq 0$, then
    \[
        \langle \BF^{[k,k,j]}_{\syn},Z_{f,c}^{\syn}\rangle_{\syn} = (-1)^k k!\binom{k}{j}\lambda_N(f^*)\frac{2E(f)E^*(f)}{E(f,f,j+1)}\Lgeom(F,F)(k,k,j+1).
    \]
\end{theorem}

\subsubsection*{Computing the cuspidal terms}
It remains to compute the other terms on the right hand side of equation~\eqref{eq:pairing-in-two}. As we have shown that the first one computes to a $p$-adic $L$-value, the second and third one have to be regarded as ``error terms''. We now show that they vanish, by showing that the cohomology class corresponding to $\iota^*(i_{\mathrm{cusp},*}(\xi))$ is already zero in motivic cohomology.
\begin{lemma}
    \label{lemma:cuspidal-is-zero}
    We have $\iota^*(i_{\mathrm{cusp},*}(\xi)) = 0$.
\end{lemma}

\begin{proof} See \cite[Proposition~9.4]{brunaultchida16}.\end{proof}

\subsubsection*{Conclusion}
The above results taken together prove the following.
\begin{theorem}
    If $j$ is even, $0\leq j\leq k$ and $E(f,f,j+1)\neq 0$, then
\begin{multline*}
    \langle \rsyn(\cBF^{k,k,j}),Z_f^{\syn} \rangle_{\syn, \KS_k} = \langle \rsyn(\lcBF^{k,k,j}),Z_f^{\syn} \rangle_{\syn, \KS_k} \\
 = (-1)^k k!\binom{k}{j}\lambda_N(f^*)\frac{2E(f)E^*(f)}{E(f,f,j+1)}\Lgeom(F,F)(k,k,j+1).
\end{multline*}
\end{theorem}
\begin{proof}
    By Lemma~\ref{lemma:cuspidal-is-zero}, all the terms in equation~\eqref{eq:pairing-in-two} vanish bar the first, hence
    \[
        \langle \rsyn(\cBF^{k,k,j}),Z_f^{\syn} \rangle_{\syn, \KS_k} = \langle \BF^{[k,k,j]}_{\syn}, Z_{f,c}^{\syn} \rangle_{\syn,\Ell^{k, k}}.
    \]
    Applying Theorem~\ref{thm:first-term} shows that this computes to the $p$-adic $L$-value in the statement of the theorem. The other equality follows from the discussion after equation~\eqref{eq:pairing-in-two}. Indeed, the difference of the two pairings equals
    \[
        \frac{(-1)^{k+j}}{2}\langle \iota^*(i_{\mathrm{cusp},*}(\xi))_{\syn}, (\rho')^* Z_{f,c}^{\syn} \rangle_{\syn,\Ell^{k,k}} - \frac{1}{2}\langle \iota^*(i_{\mathrm{cusp},*}(\xi))_{\syn}, Z_{f,c}^{\syn} \rangle_{\syn,\Ell^{k,k}}
    \]
    which is again zero by the proved results.
\end{proof}

We then obtain the main theorem of this subsection.

\begin{theorem}[$p$-adic regulator formula]
    \label{thm:syntomic-l-value}
    If $j$ is even, $0\leq j\leq k$ and $E(f,f,j+1)\neq 0$, then
    \[
  \rsyn(b_{f, j}) = (-1)^k k!\binom{k}{j}\lambda_N(f^*)\frac{2E(f)E^*(f)}{E(f,f,j+1)}\Lgeom(F,F)(k,k,j+1).
    \]
\end{theorem}

\begin{proof}
    By the commutativity of the diagram expressing the compatibility of $\rsyn$ with cup product and trace map, we directly compute
    \[
        \rsyn(b_{f, j}) = \rsyn(\langle \cBF^{k,k,j}, Z_f\rangle) = \langle \rsyn(\cBF^{k,k,j}),Z_f^{\syn} \rangle_{\syn,\KS_k}
    \]
    and we conclude by the previous theorem thanks to the hypothesis on $j$.
\end{proof}

Theorem~\ref{thm:syntomic-l-value} expresses the syntomic regulator of $b_{f, j}$ in terms of $p$-adic $L$-values. Since étale cohomology is where we truly find $p$-adic Galois representations and their Galois cohomology groups (for example Selmer groups), we also give the following equivalent formulation in terms of the étale regulator of a motivic cohomology class.
\begin{corollary}
    If $j$ is even, $0\leq j\leq k$ and $E(f,f,j+1)\neq 0$, then the local étale cohomology class $\mathrm{loc}_p(\ret(b_{f, j}))\in H^1(\Q_p(\mu_N),\Q_p(1+k-j))\otimes K_f$ satisfies the formula:
    \begin{equation}
        \label{eq:etale-l-value}
        \log(\mathrm{loc}_p(\ret(b_{f, j}))) = (-1)^k k!\binom{k}{j}\lambda_N(f^*)\frac{2E(f)E^*(f)}{E(f,f,j+1)}\Lgeom(F,F)(k,k,j+1).
    \end{equation}
\end{corollary}
These results are the $p$-adic analogues of Theorem~\ref{thm:complex-l-value}. Notice also that the parity conditions coincide. We have hence shown that the motivic class $b_{f, j}$ contains both complex and $p$-adic information about special values of $L$-functions.

\begin{remark}
 Note that if $E(f, f, 1 + j) = 0$ (which can only occur if $j = k$), then the same argument goes through to show that $L_p^{\mathrm{geom}}(F, F)(k, k, k + 1) = 0$ (unless $N = 1$, in which case there is no sensible definition of $b_{f, j}$).
\end{remark}

\section{Beilinson--Flach \texorpdfstring{$K$}{K}-elements}

 \subsection{Beilinson's cyclotomic classes}

  Let $\chi$ be a Dirichlet character, and $\Q(\chi)$ the extension of $\Q$ generated by its values. If $U$ is a $\Q(\chi)$-vector space with a $G_{\Q}$-action, we define $U^{\chi}$ as the $\chi_{Gal}$-eigenspace for the $G_{\Q}$-action:
  \begin{equation}
   \label{eq:chi-eigenspace}
   U^{\chi} \coloneqq \{ x\in U \mid \sigma(x) = \chi_{Gal}(\sigma) x, \; \forall \sigma\in G_{\Q} \}.
  \end{equation}

  \begin{theorem}[Beilinson]
   \label{thm:deligne-regulator-beilinson-element}
   Let $n \in \N$ and $\chi$ a primitive Dirichlet character of conductor $N$, with $\chi \ne 1$ if $n = 0$. Then
   \[
    \dim(H^1_{\mot}(\Z[\mu_N], n+1)\otimes\Q(\chi))^{\chi} = \begin{cases} 1 &\text{if $\chi(-1) = (-1)^n$} \\ 0 &\text{if $\chi(-1)=(-1)^{n+1}$.} \end{cases}
   \]
   Moreover, when $\chi(-1) = (-1)^n$, there exists a canonical element $\phi_{n, \chi} \ne 0 \in (H^1_{\mot}(\Z[\mu_N], n+1)\otimes \Q(\chi))^{\chi}$ such that
   \begin{equation}
    \label{eq:deligne-regulator-beilinson-element}
     \rD(\phi_{n, \chi}) = - \tau(\chi) L(\chi^{-1}, n+1) = -\frac{2}{n!} \left(\frac{2\pi i}{N}\right)^n L'(\chi, -n).
   \end{equation}
  \end{theorem}

  \begin{remark} \
   \begin{enumerate}[(i)]
    \item This is \cite[Corollary 7.1.6]{beilinson84}; see also \cite{huberwildeshaus} for a less compressed account.
    \item Note that when $n > 0$ we have $H^1_{\mot}(\Z[\mu_N], n+1) = H^1_{\mot}(\Q(\mu_N), n+1)$; but for $n = 0$ it is important to work with the integral subring, since $H^1_{\mot}(\Q(\mu_N), n+1) = \Q(\mu_N)^\times \otimes \Q$ is infinite-dimensional.

    \item In the $n = 0$ case, we have $H^1_{\mot}(\Z[\mu_N], n+1) = \Z[\mu_N]^\times \otimes \Q$, and Beilinson's element is given by
    \[ \varphi_{0, \chi} = \sum_{a \in (\Z / N)^\times} (1 - \zeta_N^a) \otimes \chi(a),\]
    which is the cyclotomic unit $u_\chi$ as defined in \cite[\S 4.1]{Dasgupta-factorization}\footnote{Dasgupta has $\frac{1 - \zeta_N^a}{1 - \zeta_N}$ in place of $1 - \zeta_N^a$, but this makes no difference, since when $n = 0$, we are requiring that $\chi$ is non-trivial and hence $\sum_a \chi(a) = 0$.}.  Moreover, in this case the Beilinson regulator is just the usual complex-analytic logarithm map.
   \end{enumerate}
  \end{remark}

  The above formulae have a $p$-adic counterpart, for any $p \nmid N$. There is a $p$-adic regulator map
  \[ r_p : H^1_{\mot}(\Q(\mu_N), n+1) \otimes \Q_p \to \Q_p(\mu_N) \]
  analogous to the Deligne regulator; it can be defined via syntomic cohomology, or alternatively as the composite of the $p$-adic \'etale regulator and the Bloch--Kato logarithm map. (For $n = 0$ it is the usual $p$-adic logarithm.)

  \begin{theorem}
   \label{thm:syntomic-regulator-beilinson-element}
   Under the hypothesis $\chi(-1)=(-1)^n$, we have the $p$-adic regulator formula
   \begin{equation}
    \label{eq:syntomic-regulator-beilinson-element}
    \left(1-\tfrac{\chi^{-1}(p)}{p^{1+n}}\right) r_p(\varphi_{n, \chi})= -\tau(\chi) L_p(\chi^{-1}, 1+n) = -i^{a_\chi}N^{-n} L_p(\chi, -n).
   \end{equation}
  \end{theorem}

  \begin{proof}
   The case $n = 0$ is the well-known $p$-adic anaytic class number formula of Leopoldt. The general case follows from the compatibility between Beilinson's cyclotomic elements and the Soul\'e twists of the cyclotomic units (conjectured in \cite[Conjecture 6.2]{blochkato90} and proved in \cite{huberwildeshaus}). See \cite[\S 4.2]{perrinriou95} for the case $\chi = 1$.
  \end{proof}

 \subsection{Beilinson-Flach $K$-elements}
  
  We now study in more detail the Beilinson--Flach $K$-element $b_{f, j}$ defined in Definition \ref{def:BFKelt}. Note that $b_{f, j}$ is trivially zero if $j$ is odd, since $\Xi^{k, k, j}$ and $Z_f$ lie in opposite eigenspaces for the swapping involution. We assume henceforth that $j$ is even.

  \begin{proposition}
   \label{prop:galois-action}
   The Galois group $\gal{\Q(\mu_N)}{\Q}$ acts on $b_{f, j}$ through the character $\psi_{Gal}^{-1}$, i.e.\ $b_{f, j}^{\sigma} = \psi_{Gal}^{-1}(\sigma)b_{f, j}$. Explicitly, the automorphism $\sigma_q \colon \zeta \mapsto \zeta^q$ acts as $b_{f, j}^{\sigma_q} = \psi(q)b_{f, j}$.
  \end{proposition}

  \begin{proof}
   The intersection pairing is Galois-equivariant, hence
   \[
    b_{f, j}^{\sigma} = \langle (\cBF^{k,k,j})^{\sigma}, Z_f^{\sigma} \rangle.
   \]
   The Beilinson--Flach class is defined over $\Q$, so it is left invariant by $\sigma$. To compute $Z_f^{\sigma}$, we write $\sigma = \sigma_q$ for some $q\in(\Z/N\Z)^{\times}$, where $\sigma_q$ is the unique automorphism of $\gal{\Q(\mu_N)}{\Q}$ characterised by $\zeta_N \mapsto \zeta_N^q$.

   In the composition defining $Z_f$, the only correspondence which is not defined over $\Q$ is that induced by the Atkin-Lehner involution. According to~\cite[3.5.2]{ohta95}, the action of the Galois group on the Atkin-Lehner involution is given by composition with a diamond operator: $w_N^{\sigma_q} = w_N \circ \langle q \rangle$. Since $\langle q \rangle^*$ acts on $\omega_f$ and $\eta_f$ as the multiplication by $\psi(q) = (\psi)_{Gal}^{-1}(\sigma_q)$, the claim is proved.
  \end{proof}

  The proposition shows in particular that $b_{f, j} \in \bigl(H^1_{\mot}(\Q(\mu_N),1+k-j)\otimes K_f\bigr)^{\psi}$. We want to compare it with the class $\phi_{k-j, \psi}$ from Theorem~\ref{thm:deligne-regulator-beilinson-element}, using the regulator formulae from the preceding section. First we need to show that both $b_{f, j}$ and $\phi_{k-j,\psi}$ land in the same 1-dimensional space, for which we need to control their ``integrality'' at the primes dividing $N$.

 \subsection{Integrality of \texorpdfstring{$K$}{K}-elements}
 \label{subsec:integrality}

  In this subsection we discuss integrality of Beilinson--Flach $K$-elements. The question we are addressing is whether or not they live in $(H_{\mot}^1(\ringint{O}_{\Q(\mu_N)},1+k-j) \otimes K_f)^{\psi}$. We separate the cases $j < k$ and $j=k$. (Note that the case $j = k = 0$ has been treated by Dasgupta in \cite{Dasgupta-factorization}, using very different methods.)

  \subsubsection*{Case $j<k$}
  In this case $2k-2j+1 \geq 3$, which entails the equality
  \[
   K_{2k-2j+1}(\Q(\mu_N))\otimes_{\Z} \Q = K_{2k-2j+1}(\Z[\mu_N])\otimes_{\Z} \Q.
  \]
  So there is nothing to prove: we can already regard $b_{f, j}$ as being integral, since all classes in this $K$-group have this property. Hence we must have $b_{f, j} \in K_f \cdot \varphi_{k - j, \psi}$.

  \subsubsection*{Case $j=k$}

  In this case the $K$-elements live in $H^1_{\mot}(\Q(\mu_N),1)\otimes K_f = \Q(\mu_N)^{\times} \otimes K_f$. Moreover, from~\cite[Remark~4.2.1]{scholl90} we know that $\KS_k$ has a smooth and proper model over $\Z[1/N]$; as our motivic cohomology classes $\Xi^{k, k, j}$ and $Z_f$ extend to this model, we have $b_{f, k} \in (\Z[\mu_N, 1/\ups]^\times \otimes K_f)^{\psi}$, where $\ups$ is the set of primes of $\Q(\mu_N)$ dividing $N$.

  It remains to study local behaviour of the classes $b_{f, k}$ at primes above $N$. For this, we shall use $p$-adic analytic techniques. (Recall that by assumption $p \nmid N$.) Let $v$ be the prime of $K_f$ determined by our choice of embedding $\overline{\Q} \hookrightarrow \overline{\Q}_p$. Then we have the following:

  \begin{proposition}\label{prop:p-adic-unram}
   An element $x \in (\Q(\mu_N)^{\times} \otimes K_f)^{\psi}$ is integral at a prime $\ell \ne p$ iff its image in $H^1(\Q, K_{f, v}(1)(\psi))$ is unramified at $\ell$.\qed
  \end{proposition}

  From the compatibility of our pairings, the image of $b_{f, k}$ in Galois cohomology is the image of the \'etale Beilinson--Flach class $\BF^{f, f, k}_{\et}$ under the map of Galois representations
  \[ M_{\et}(f \otimes f)^*(-k) \to (\wedge^2 M_{\et}(f)^*)(1) = K_f(1)(\psi) \]
  given by pairing with the \'etale realisation of $Z_f$. So if $\BF^{f, f, k}_{\et}$ is unramified at $\ell$ (i.e.~satisfies the Bloch--Kato ``$H^1_{\mathrm{f}}$'' local condition), it follows that $b_{f, k}$ is integral at the primes above $\ell$.

  \begin{proposition}
   Assume that the Hecke polynomial of $f$ at $p$ has a root $\alpha_f$ such that $(f, \alpha_f)$ is noble and the Euler factor $E(f, f, 1 + k)$ does not vanish. Then the class $\BF^{f, f, k}_{\et}$ is unramified at all primes $\ell \ne p$, and hence $b_{f, k}$ is a $K_f$-multiple of the Beilinson element $\varphi_{1+k-j, \psi}$.
  \end{proposition}

  \begin{proof}
   In \cite{loefflerzerbes16}, the second author and Zerbes constructed classes in Pottharst-style ``analytic Iwasawa cohomology'' interpolating the Beilinson--Flach classes, and showed that any class admitting such an interpolation is necessarily in $H^1_{\mathrm{f}}$ away from $p$; see Theorem 8.1.4 (i) of \emph{op.cit.}.

   (Note that in Assumption 3.5.6 of \emph{op.cit.} it is assumed that $f$ not be a twist of $g$, clearly ruling out the case $f = g$ considered here; but this is used only to control torsion in Iwasawa cohomology. Since we do not require a full Euler system but only the $p$-direction, it suffices that $\psi \ne 1$; and our hypotheses imply $\psi \ne 1$, since if $\psi(p) = 1$ then $\alpha_f \beta_f = p^{k + 1}$ and hence $E(f, f, 1 + k) = 0$.)

   However, in order to interpolate in this manner, one needs to ``drop some Euler factors'': more precisely, the class $c_1^{\alpha_f \alpha_f}$ in \emph{op.cit.} depends on a choice of root $\alpha_f$ of the Hecke polynomial, and it is related to the $\BF^{f, f, k}_{\et}$ considered here by the formula (see Prop 3.5.5 of \emph{op.cit.})
   \[ c_1^{\alpha_f \alpha_f} = \left(1 - \frac{p^k}{\alpha_f^2}\middle)
   \middle(1 - \frac{\alpha_f \beta_f}{p^{k+1}}\right)^2
   \left(1 - \frac{\beta_f^2}{p^{k+1}}\right)
    \BF^{f, f, k}_{\et}. \]
   So if this product of Euler factors is non-zero, we deduce that $\BF^{f, f, k}_{\et}$ is unramified. This Euler factor is not quite the same as $E(f, f, 1 + k)$, which is given by
   \[ \left(1 - \frac{p^k}{\alpha_f^2}\middle)
      \middle(1 - \frac{p^k}{\alpha_f \beta_f}\middle)
      \middle(1 - \frac{\alpha_f \beta_f}{p^{k+1}}\middle)
      \middle(1 - \frac{\beta_f^2}{p^{k+1}}\right) ;\]
   however, since $\alpha_f \beta_f = p^{k+1} \psi(p)$, the factors with $p^k$ in the numerator cannot vanish, and hence the first Euler factor vanishes iff the second one does.
  \end{proof}

  \begin{remark}
   Since we are proving a property of the motivic class $b_{f, k}$ which does not depend on $p$, we are free to vary our prime $p$. One can check, using the Sato--Tate conjecture (or from various weaker equidistribution results), that if $\psi$ is non-trivial, there exist infinitely many primes such that the preceding theorem applies, so $b_{f, k} \in \Z[\mu_N]^\times \otimes K_f$. However, we shall not need this here, since we are only introducing $b_{f, k}$ in order to prove a relation between $p$-adic $L$-values, and we shall see shortly that this holds trivially (as $0 = 0$) when $E(f, f, 1 + k) = 0$.
  \end{remark}

\section{Factorisation and interpolation in half of the weight space}
 \label{sec:comparison}
 
 In this section we gather the work done in the previous sections to prove Theorem~\ref{intro:final2}, and Theorem~\ref{intro:final1} in a half of the weight space. We will then discuss how to generalise the interpolation equations when we relax the hypotheses.

 \subsection{Comparison of motivic classes}

  In this section we prove a relation between motivic cohomology classes. We will need Theorems~\ref{thm:deligne-regulator-beilinson-element} and~\ref{thm:complex-l-value}, which allow us to ``lift'' the complex factorisation in terms of motivic classes. We suppose that $j$ is \emph{even}. The mentioned results provide equations linking higher cyclotomic classes and Beilinson--Flach $K$-elements to complex $L$-values. In particular, $\phi_{k-j, \psi}$ is related to the value $L'(\psi,j-k)$ and $b_{f, j}$ is related to the value $L'(f,f,j+1)$.

  Recall equation~\eqref{eq:leading-terms} relating leading terms of the complex $L$-functions:
  \[
   L'(f,f,j+1) = L'(\psi,j-k)\Limp(\Sym^2 f,j+1)
   \prod_{\substack{\ell \mid N \\ \ell \nmid N_{\psi}}} \Bigl(1 - \frac{\psi(\ell)}{\ell^{j-k}} \Bigr).
  \]
  Equations~\eqref{eq:deligne-regulator-beilinson-element} and~\eqref{eq:deligne-regulator-bf} proved in the previous sections, together with the formula \eqref{eq:petersson} for $(\omega_f, \bar{\omega_f})$, show that the above translates to
  \[ \rD(b_{f, j}) = \Upsilon \cdot \rD(\phi_{k-j, \psi}), \]
  where
  \[ 
   \Upsilon \coloneqq  \frac{(2\pi i N_\psi)^{k-j} \lambda_N(f^*) k!^2}{4 (4\pi)^{k+1} \langle f, f \rangle (k-j)!}\Limp(\Sym^2 f,j+1)\smashoperator{\prod_{\substack{\ell \mid N \\ \ell \nmid N_{\psi}}}} \Bigl(1 - \frac{\psi(\ell)}{\ell^{j-k}}\Bigr) \in \C.
  \]
 
  We can now prove the key motivic result which is the heart of our factorization formula:
  \begin{theorem}
   \label{thm:motivic-relation}
   Assume $j < k$, or $j = k$ and $E(f, f, 1 + k) \ne 0$. Then $\Upsilon \in K_f$, and moreover we have
   \[ b_{f, j} = \Upsilon \cdot \phi_{k-j, \psi}\quad \text{as elements of}\quad H^1_{\mot}(\Z[\mu_N], 1 + k - j) \otimes K_f.\]
  \end{theorem}
  
  \begin{proof}
   Under these conditions, we have seen that $b_{f, j}$ is necessarily in the one-dimensional space spanned by $\varphi_{k-j, \psi}$. Since $r_{\del}(\varphi_{k-j, \psi}) \ne 0$, the constant of proportionality must be $\Upsilon$.
  \end{proof}

  \begin{remark}
   One can show directly that $\Upsilon \in K_f$ (even in the cases where we have not shown that $b_{f, j}$ is integral), using the rationality results for symmetric-square $L$-values (originally due to Sturm) summarized in \cite[Corollary~2.6]{schmidt88}.
  \end{remark}

 \subsection{\texorpdfstring{$p$}{p}-adic factorisation and interpolation}

  If we were in the ordinary case, the next step would be to deduce a relationship between the elements $b_{f, j}$ and $\phi_{k-j, \psi}$ in the $p$-adic setting by applying the syntomic regulator to the last theorem, and then prove the theorem by making the $p$-adic symmetric square appear from the extra factor $\Upsilon$---which has $\Limp(\Sym^2 f,j+1)$ as a factor---thanks to its interpolation equation. That argument works well in the ordinary case, because the only $L$-function for which we miss a regulator formula is the symmetric square one, but we can still pass from its complex (critical) values to $p$-adic ones via interpolation. By contrast, in the supersingular case we face the problem of the \emph{existence} of a $p$-adic symmetric square $L$-function. As we cannot ``translate'' the complex factor in a $p$-adic one, such missing piece jeopardises the whole argument. We then choose a different route: we \emph{define} a $2$-variable $p$-adic $L$-function via the putative factorisation, and then prove that this interpolates a family of $L$-functions $\Limp(\Sym^2 \cdot, \cdot)$ where both the modular form and the complex variable are allowed to vary analytically.

  Recall that we assume $\alpha_f\neq\beta_f$, to guarantee the existence of $F$ and $\Lgeom(F,F)$.
  \begin{definition}
   Let $U$ be the largest affinoid open around $k$ such that the function $\Lgeom(F,F)$ is well-defined and analytic over $U\times U\times\Weight$. The imprimitive $p$-adic symmetric square $L$-function $\Limp_p(\Sym^2 F)$ is defined as:
   \begin{align*}
    \Limp_p(\Sym^2 F) \colon U\times \Weight &\to \overline{\Q}_p \\
    (\kappa, \sigma) &\mapsto \frac{\Lgeom(F,F)(\kappa, \kappa, \sigma)} {L_p(\psi,\sigma-\kappa-1)}
   \end{align*}
   for all $(\sigma_1,\sigma_2)\in U\times\Weight$ such that $L_p(\psi,\sigma-\kappa-1) \neq 0$. (If $\psi = 1$ and $\sigma = \kappa + 1$ or $\kappa+2$, this ratio is 0, since $\Lgeom$ is analytic while the denominator has a pole.)
  \end{definition}

  We may be introducing poles at the zeros of $L_p(\psi)$, so in general $\Limp_p(\Sym^2 F)$ is a meromorphic function. Our goal will be to show that this function interpolates symmetric square $L$-values, so it can be validly interpreted as a $p$-adic symmetric square $L$-function for the Coleman family $F$. Let us consider the specialization at $(k, j)$, for $0 \le j \le k$ and $j$ even.

  First we suppose that the hypotheses of Theorem \ref{thm:motivic-relation} hold. Applying the syntomic regulator to both sides of the relation in Theorem~\ref{thm:motivic-relation} gives $\rsyn(b_{f, j}) = \Upsilon\cdot \rsyn(\phi_{k-j, \psi})$. By theorems~\ref{thm:syntomic-regulator-beilinson-element} and~\ref{thm:syntomic-l-value}, both of these syntomic regulators are related to $p$-adic $L$-values (the parity conditions of those theorems are satisfied as we already assumed that $j$ is even). Substituting these formulae into the relation $\rsyn(b_{f, j}) = \Upsilon \cdot \rsyn(\phi_{k-j, \psi})$, we obtain
  \[
   (-1)^k k! \binom{k}{j} \lambda_N(f^*) \frac{2 E(f) E^*(f)}{E(f, f, j + 1)} L_p^{\mathrm{geom}}(F, F)(k, k, j + 1) = \Upsilon \cdot \frac{-i^{a_\psi} N_\psi^{j-k} L_p(\psi, j - k)}{\left(1 - \tfrac{\psi^{-1}(p)}{p^{1 + k - j}}\right)},
  \]
  and if $L_p(\psi, j - k) \ne 0$, this rearranges into
  \begin{multline}
   \label{eq:auxp-unraveled}
   L_p^{\mathrm{imp}}(\Sym^2 F)(k, j + 1) = 
   \\ 
   \frac{(-1)^{k + 1}i^{a_\psi} (2 \pi i)^{k - j} j! E(\Sym^2 f, j + 1)}{8 E(f)E^*(f) (4\pi)^{k+1} \langle f, f \rangle} \Limp(\Sym^2 f,j+1)\prod_{\substack{\ell \mid N \\ \ell \nmid N_{\psi}}} \Bigl(1 - \frac{\psi(\ell)}{\ell^{j-k}}\Bigr).
  \end{multline}
  Here we have defined 
  \[ E(\Sym^2 f, j + 1) \coloneqq \left(1 - \frac{p^j}{\alpha_f^2}\middle) \middle(1 - \frac{\alpha_f\beta_f}{p^{1 + j}}\middle) \middle(1 - \frac{\beta_f^2}{p^{1 + j}}\right), \]
  which is the expected shape of the interpolation factor for $\Sym^2$ $p$-adic $L$-functions in the left half of the critical strip. (The denominator is non-zero, as $E(f)\neq 0$ always, and $E^*(f)\neq 0$ because we assumed $\alpha_f\neq\beta_f$.)
  
  We now suppose $j = k$ and $E(f, f, k+1) = 0$. Then $E(\Sym^2 f, k + 1)$ is also zero, so the right-hand side of \eqref{eq:auxp-unraveled} vanishes. On the other hand, if $N > 1$, then $\Lgeom$ vanishes, as we have seen above; and if $N = 1$ then necessarily $\psi = 1$, so the denominator term $L_p(\psi, \sigma - \kappa - 1)$ is $\infty$. So in either case $\Limp_p(\Sym^2 F)(k, k+1) = 0$ and \eqref{eq:auxp-unraveled} holds as $0 = 0$.

  The same computation applies if we specialize instead at $(t, j)$, for some $t \in \N$ lying in our disc around $k$ (and $0 \le j \le t$ even), with $f$ replaced by the specialization $f'$ of $F$ in weight $t + 2$. (More precisely, $f'$ is the level $N$ eigenform such that $F$ specializes to a $p$-stabilization of $f'$.) This proves that $\Limp_p(\Sym^2 f)$ actually interpolates the critical values (in the left half of the critical strip) of the complex function $\Limp(\Sym^2 f')$, as $f'$ varies over specializations of $F$, as long as we avoid points where $L_p(\psi, \sigma - \kappa - 1) = 0$. So we have proved the following:

  \begin{theorem}
   \label{thm:first-half-final}
   Suppose that $(f, \alpha_f)$ is a noble eigenform. Then, for all sufficiently small affinoid discs $U \ni \{k\}$, there exists a unique $2$-variable meromorphic $p$-adic $L$-function
   \[ \Limp_p(\Sym^2 F)\colon U\times\Weight_- \to \overline{\Q}_p, \]
   analytic away from the zeroes of $L_p(\psi, \sigma - \kappa - 1)$, with the following interpolation property: at every pair $(t,j)\in \N^2 \cap\left( U\times \Weight\right)$ with $0 \le j \le t$, $j$ even, and $L_p(\psi, j - t) \ne 0$, we have
   \begin{multline}
    \label{eq:final-interpolation}
    \Limp_p(\Sym^2 F)(t,j+1) = (-1)^{t+1}j! \frac{(2\pi i)^{t-j} i^{a_{\psi}}}{8(4\pi)^{t+1}\langle f_{t},f_{t} \rangle} \\
    \cdot \frac{E(\Sym^2 f_{t},j+1)}{E(f_{t})E^*(f_{t})} \Limp(\Sym^2 f_{t},j+1)\prod_{\substack{\ell \mid N \\ \ell \nmid N_{\psi}}} \Bigl(1 - \frac{\psi(\ell)}{\ell^{j-t}}\Bigr),
   \end{multline}
   where $f_{t}\in M_{t+2}(N,\psi)$ is the weight $t$ specialization of $F$; and the following factorization of $p$-adic $L$-functions holds on $U \times \Weight_-$:
   \begin{equation}
    \label{eq:final-factorization}
    \Lgeom(F,F)(\sigma_1,\sigma_1,\sigma_2) = \Limp_p(\Sym^2 F)(\sigma_1,\sigma_2)\cdot L_p(\psi,\sigma_2-\sigma_1-1).\myqed
   \end{equation}
  \end{theorem}
  
  This proves ``half'' of Theorems \ref{intro:final1} and \ref{intro:final2} from the introduction.
  
  \begin{remark}
   By a somewhat lengthy calculation, one can verify that this interpolating property agrees with the interpolating property of the $p$-adic symmetric square $L$-function considered in \cite{Dasgupta-factorization} in the ordinary case, up to the factor $-2^{(2t + 1)}$ and the extra Euler factors at the primes dividing $N$ but not $N_\psi$.
  \end{remark}

 \subsection{One-variable interpolation and ramified twists}

  If we simply specialize Theorem \ref{thm:first-half-final} at $\kappa = k$, then we deduce that there exists an analytic function $\Limp_p(\Sym^2 f, -) : \Weight_- \to \overline{\Q}_p$ with the following properties:
  \begin{enumerate}[(i)]
   \item its values at $\sigma = j + 1$, for even $j$ with $0 \le j \le k$, interpolate critical $\Sym^2$ $L$-values;
   \item the $p$-adic Rankin--Selberg $L$-function $\Limp_p(f, f, -)$ factorizes as
   \[ \Limp_p(f, f, \sigma) = \Limp_p(\Sym^2 f, \sigma) \cdot L_p(\psi, \sigma - k - 1).\]
  \end{enumerate}

  However, for a fixed $k$ the interpolating property (i) only sees finitely many values of $\Limp_p(\Sym^2 f, -)$ and thus does not determine it uniquely; so it is not \emph{a priori} clear if it matches up with other constructions of $p$-adic symmetric square $L$-functions in a fixed weight, such as that of \cite{dabrowskidelbourgo97}.

  Although we shall not give full details here (for reasons of space), this can be dealt with as follows:
  \begin{itemize}
   \item The reciprocity law for Beilinson--Flach elements (\cite{KLZ20, loefflerzerbes16}) holds as an identity of 3-variable functions. Thus, although it is proved by analytic continuation from crystalline specializations, we can nonetheless specialize it at $(k, k, j + 1 + \theta)$ for ramified characters $\theta$ of $\Z_p^\times$, to obtain a formula for $r_{\syn}$ of a ``$\theta$-twisted'' Beilinson--Flach element in terms of $p$-adic $L$-values.

   \item If $0 \le j \le k$ and $j + \theta$ is even, then this syntomic regulator factors through the $\wedge^2$ direct summand (as in the case $\theta = 1$ considered above), and we may repeat the argument to obtain a relation between $\Limp_p(\Sym^2 f, j + 1 + \theta)$ and the critical value $\Limp(\Sym^2 f, \theta^{-1}, 1 + j)$, extending \eqref{eq:final-interpolation}, assuming as usual that $L_p(\psi, j -k + \theta) \ne 0$.
  \end{itemize}

  This will imply, in particular, that $\Limp_p(\Sym^2 f, \sigma)$ agrees with the $p$-adic $L$-function of \cite{dabrowskidelbourgo97} when $v_p(\alpha_f) < \tfrac{k + 1}{2}$ (which is required for the latter to be well-defined), up to correcting for bad Euler factors at primes $\ell \mid N$.

  \begin{remark}
   This is in a sense not the ``morally right'' argument, which is another reason why we have not given the details in full. A much superior approach would be to work out a direct analytic construction of $p$-adic $\Sym^2$ $L$-functions for Coleman families, compatible with the Dabrowski--Delbourgo construction in a fixed weight. This analytically-defined object would then have to agree with our motivically-defined $\Limp_p(\Sym^2 F)$ (since the interpolation property of the latter holds at a Zariski-dense set if we do not fix $k$),  and we could then deduce the factorization in a fixed weight by specializing. This would also show that $\Limp_p(\Sym^2 f)$, and its two-variable analogue, are actually analytic rather than merely meromorphic.
  \end{remark}

\section{Functional equation and interpolation in the whole weight space}
 \label{sec:functional-eq-interpolation-whole}

 In the previous section, we studied the function $\Limp_p(\Sym^2 F)$ on $U \times \Weight_-$, showing that it interpolates $\Sym^2$ $L$-values in the ``left half'' of the critical strip and is related to $L_p(F, F, -)$ by a factorization formula \eqref{eq:final-factorization}. In this section we consider the restriction to $U \times \Weight_+$, and relate this to $\Sym^2$ $L$-values in the ``right half'' of the critical strip. For this, we shall need to use functional equations, which do not work well for imprimitive $L$-functions. So, for simplicity, we shall impose the following hypothesis throughout \S\ref{sec:functional-eq-interpolation-whole}:
  \begin{itemize}
  \item Both $\psi$ and $\psi^2$ are primitive modulo $N$.
  \end{itemize}

  Then the imprimitive and primitive $\Sym^2$ $L$-functions agree, and likewise the Rankin--Selberg $L$-functions, so we shall drop the ``$\mathrm{imp}$'' labels.

 \subsection{Comparison of functional equations}

  Recall that the local $\epsilon$-factor appearing in the functional equation of $L(f \otimes g, s)$ has the form $\epsilon(f \otimes g, s) = A \cdot N^{3s}$, where $A \in \overline{\Q}^\times$ is a non-zero constant which can be given explicitly in terms of the Hecke eigenvalues of $f$ and $g$ at primes $\ell \mid N$. It follows that there is a function $\tilde\epsilon(F \otimes F)$ on $V_1 \times V_2 \times \Weight$ specializing at classical points $(t, t', s) \in U \times U \times \Weight$ to $\epsilon(f_t \otimes f_{t'}, s)$ (and this function has the form $A(\kappa_1, \kappa_2) \cdot N^{3\sigma}$, although we shall not use this directly).
    
  \begin{proposition}[Benois--Horte]
   We have
   \[ 
    \Lgeom(F, F)(\kappa_1, \kappa_2, \kappa_1 + \kappa_2 + 3 - \sigma) = \tilde\epsilon(F \otimes F)(\kappa_1, \kappa_2, \sigma) \cdot \Lgeom(F^*, F^*)(\kappa_1, \kappa_2, \sigma). 
   \]
  \end{proposition}
  
  \begin{proof}
   This follows by comparing the two sides at critical values (which are Zariski-dense in $V_1 \times V_2 \times \Weight$) and using the classical functional equation; see \cite[Proposition~6.4.2]{benoishorte23}.
  \end{proof}
  
  From this we can derive a functional equation for $L_p(\Sym^2 F)$. We set $\tilde{\epsilon}(\Sym^2 F)(\kappa, \sigma) = \tfrac{\tilde{\epsilon}(F \otimes F)(\kappa, \kappa, \sigma)}{\tilde{\epsilon}(\psi, \sigma - \kappa - 1)}$, which is an analytic function on $U \times \Weight$; by construction this interpolates the $\varepsilon$-factor of the symmetric square. We can now deduce Theorem \ref{intro:final3}:
  
  \begin{theorem}
   \label{thm:functional-equation}
   Suppose that $\psi$, $\psi^2$ are primitive modulo $N$. Then $L_p(\Sym^2 F)$ satisfies the following functional equation for every $(\kappa,\sigma)\in U\times\Weight$:
   \[
    L_p(\Sym^2 F)(\kappa, 2\kappa + 3 - \sigma) = \tilde{\epsilon}(\Sym^2 F)(\kappa, \sigma) \cdot L_p(\Sym^2 F^*)(\kappa,\sigma).
   \]
  \end{theorem}
  
  \begin{proof}
   Since we have by definition $L_p(\Sym^2 F)(\kappa, \sigma) = \Lgeom(F,F)(\kappa,\kappa,\sigma) L_p(\psi,\sigma-\kappa-1)^{-1}$, this follows immediately from the functional equations for $\Lgeom(F,F)$ and $L_p(\psi)$.
  \end{proof}
  
 \subsection{Interpolation of critical values}
 
  We now use Theorem \ref{thm:functional-equation} and the classical functional equation to recover an interpolating property at critical values in the right half of the strip.

  \begin{theorem}
   \label{thm:interpolation-righthalf}
   At every pair $(t,j)\in \N^2 \cap U\times \Weight$ with $t + 1 \le j \le 2t + 1$, $j$ odd, and $L_p(\psi, j - t) \ne 0$, we have
   \begin{multline}
    \label{eq:final-interpolation2}
    L_p(\Sym^2 F)(t,j+1) =\\ \frac{E^*(\Sym^2 f_t, 1 + j)}{E(f_{t})E^*(f_{t})} \cdot \frac{(-1)^{t + 1} (2\pi)^{t-2j - 1} j! (j - t - 1)!}{2^{t+3}\langle f_t, f_t\rangle} \cdot L(\Sym^2 f_{t},j+1),
   \end{multline}
   where 
   \[ E^*(\Sym^2 f_t, 1 + j) = \left(1 - \frac{p^j}{\alpha_{f_t}^2}\middle) \middle(1 - \frac{p^j}{\alpha_{f_t} \beta_{f_t}}\middle)\middle(1 - \frac{\beta_{f_t}^2}{p^{1 + j}}\right).\]
  \end{theorem}
  
  \begin{proof}
   This follows from the interpolating property for $L_p(\Sym^2 F^*)$ at $(t, 2t + 2 - j)$, via a tedious but routine calculation using the functional equations for the complex and $p$-adic symmetric square $L$-functions.
  \end{proof}

  \begin{remark}
   Note that $E^*(\Sym^2 f, 1 + j)$ is the expected form of Euler factor for symmetric square $L$-functions in the \emph{right} half of the critical strip, which differs slightly from its analogue in the left half; it arises in our computation since we have the symmetry relation $E^*(\Sym^2 f, 1 + j) = E(\Sym^2 f^*, 2k + 2 - j)$.
  \end{remark}
  
 \section{Remark on character twists}
 
  In the ordinary case, Dasgupta actually proves a slightly more general factorisation theorem than we have proved above. If $\chi$ is an arbitrary Dirichlet character, then there is a factorization of complex $L$-functions extending \eqref{eq:complexfact}, 
  \[ L(f \otimes f \otimes \chi, s) = L(\Sym^2 f \otimes \chi, s) L(\chi \psi, s - k - 1). \]
  If $\chi$ is a square in the group of Dirichlet characters, this more general formula follows by applying \eqref{eq:complexfact} to a twist of $f$, but this does not cover all $\chi$. The main result of \cite{Dasgupta-factorization}\footnote{We caution the reader that $\chi$ and $\psi$ are interchanged relative to Dasgupta's notation.} gives a $p$-adic counterpart of this for general $\chi$ (in which one can assume without loss of generality that $\chi$ has prime-to-$p$ conductor, since the $p$-part of $\chi$ can be absorbed into the $\sigma$ variable). 
    
  We briefly sketch how the arguments of this paper can be extended to this setting, leaving the details to the interested reader. As explained in \cite{KLZ17}, the Beilinson--Flach elements $\BF_{\mot, N}^{[k, k, j]}$ are the $M = 1$ case of a more general family of classes $\BF_{\mot, M, N, a}^{[k, k, j]}$ defined over $\Q(\mu_M)$, for $M \ge 1$ and $a \in \Z / M\Z$ (which satisfy Euler-system norm relations as $M$ varies). If $\chi$ is a primitive Dirichlet character modulo $M$, we can set
  \[ \BF_{\mot, N, \chi}^{[k, k, j]} = \sum_{a \in (\Z / M)^\times} \chi(a) \cdot \BF_{\mot, M, N, a}^{[k, k, j]}, \]
  which is defined over $\Q(\mu_M)$ and transforms under $\operatorname{Gal}(\Q(\mu_M) / \Q)$ via $\chi_{Gal}$. Pairing this element with $Z_f$ gives a $K$-theory class defined over $\Q(\mu_{MN})$, lying in the $(\chi \psi)_{Gal}$ eigenspace; and if $\chi(-1) = (-1)^j$, the complex and $p$-adic regulators of this class are related to the complex and $p$-adic Rankin $L$-values of $f \otimes f \otimes \chi$. We can now argue exactly as in the $\chi = 1$ case to deduce the existence of a $p$-adic $L$-function for $\Sym^2 F \otimes \chi$, and a factorization formula relating this to $p$-adic Rankin $L$-series.

\hyphenation{Man-u-scrip-ta}
\hyphenation{mo-tives}

\providecommand{\bysame}{\leavevmode\hbox to3em{\hrulefill}\thinspace}
\providecommand{\MR}[1]{}
\renewcommand{\MR}[1]{%
 MR \href{http://www.ams.org/mathscinet-getitem?mr=#1}{#1}.
}
\providecommand{\href}[2]{#2}
\newcommand{\articlehref}[2]{\href{#1}{#2}}

\end{document}